\documentclass[11pt,twoside,a4paper]{article}
\usepackage{amsmath,amssymb,mathabx}
\usepackage{graphicx,caption}
\usepackage{verbatim}
\usepackage{xcolor}
\usepackage{mathrsfs}
\usepackage{authblk}
\usepackage{bm}

\usepackage[numbers,sort&compress]{natbib}

\usepackage{tikz}
\usetikzlibrary{arrows.meta}
\usepackage{amsthm}
\usepackage{thmtools}
\usepackage{hyperref}
\usepackage{cleveref}
\usepackage{orcidlink}
\hypersetup{
    colorlinks=true,
    linkcolor=green,
    citecolor=magenta,
}
\newcommand{\doi}[1]{doi: \begingroup\urlstyle{rm}\href{https://doi.org/#1}{\textcolor{blue!70!black}{\nolinkurl{#1}}}\endgroup}

\declaretheorem[name=Theorem,numberwithin=section]{theorem}
\newtheorem{proposition}[theorem]{Proposition}
\newtheorem{corollary}[theorem]{Corollary}
\newtheorem{lemma}[theorem]{Lemma}

\theoremstyle{definition}
\newtheorem{definition}[theorem]{Definition}
\newtheorem{example}[theorem]{Example}
\newtheorem{remark}[theorem]{Remark}

\usepackage[left=2cm, right=2cm, top=3cm, bottom=3cm]{geometry}

\numberwithin{equation}{section}

\newcommand{\wt}{\tilde}

\def\R{\mathbb{R}}
\def\N{\mathbb{N}}
\newcommand{\SoftMax}{\mathrm{softmax}}

\newcommand{\cC}{{\mathcal C}}
\newcommand{\cI}{{\mathcal I}}
\newcommand{\cS}{{\mathcal S}}
\newcommand{\cW}{{\mathcal W}}
\newcommand{\cX}{{\mathcal X}}
\DeclareMathOperator*{\argmax}{arg\,max}
\DeclareMathOperator*{\argmin}{arg\,min}
\DeclareMathOperator{\osc}{osc}
\DeclareMathOperator{\diag}{diag}

\newcommand{\F}{\mathcal{F}}

\renewcommand{\P}{\mathbb{P}}

\newcommand{\keywords}[1]{\noindent\textbf{Keywords:} #1}
\newcommand{\subjclass}[1]{\noindent\textbf{AMS2020 subject classification:} #1}

\allowdisplaybreaks

\title{ Colorful Exponential Random Graph Models }

\author{Bhaswar B. Bhattacharya\thanks{University of Pennsylvania and National University of Singapore. \texttt{Email:bhaswar@wharton.upenn.edu}}, \quad Pierfrancesco Dionigi\,\orcidlink{0000-0003-2180-8669}\thanks{Barcelona School of Economics, Barcelona, Spain. \texttt{Email:pier.dionigi@bse.eu}}, \quad Ankan Ganguly\,\orcidlink{0000-0001-6233-9390}\thanks{Boston University, Boston, USA. \texttt{Email:ankang@bu.edu}}, \quad Giulio Zucal\,\orcidlink{0009-0000-8261-1291}\thanks{Max Planck Institute of Molecular Cell Biology and Genetics, and Center for Systems Biology, Dresden, Germany.  \texttt{Email:zucal@mpi-cbg.de}} }

\date{}

\begin{document}
\maketitle

\begin{abstract}
In this paper, we initiate the study of \emph{colored exponential random graph
models} (ERGMs), a class of exponential-family models for networks with multiple
types of edge relations. Using the framework of probability graphons, we first derive
a variational representation for the limiting free energy, whose maximizers determine the asymptotic structure of typical samples from the model. Then we identify several general families of colored ERGMs exhibiting replica symmetry, where the variational problem has constant maximizers and the model asymptotically concentrates on product colorings with independent edges. For general colored ERGMs, we derive Euler--Lagrange fixed-point equations for the variational maximizers, which in turn yield a general high-temperature uniqueness criterion. In the complementary zero-temperature regime, we establish a two-level selection principle: the leading energy term determines the ground states, while the lower-order energy terms, combined with entropy, act as a tie-breaker to determine the asymptotic zero-temperature structure of the model. We illustrate this principle through the induced wedge and rainbow triangle ERGMs. Both models have natural interpretations in multitype networks, and their zero-temperature limits exhibit interesting structures that connect to well-known results in extremal combinatorics. We further
establish finite-temperature symmetry breaking for both these models and complement
the rigorous results with numerical experiments.\par\medskip

\par\medskip

\end{abstract}

\section{Introduction}

Exponential random graph models (ERGMs) are probability distributions on
networks in which local statistics, such as counts of triangles, 
stars, or degree-based statistics, appear as sufficient statistics. These models originated from the early works of
\citet{HL81} and \citet{frank1986markov}, and have
since emerged as a flexible framework for modeling dependent relational data  \cite{wasserman1994social,lusher2013exponential,park2004statistical,snijders2006new,
cranmer2020inferential,bhamidi2008mixing}. Using the machinery of graph limit theory (graphons) \cite{BORGS20081801,borgs2011convergentAnnals,LovaszGraphLimits},  \citet{ChatterjeeDiaconisExponential2013} in their seminal paper, developed the theoretical framework of understanding the asymptotic properties of ERGMs, such as the limiting free energy and the typical macroscopic structure of a sampled graph. Their work revealed interesting connections between ERGMs, statistical physics, and extremal combinatorics, which have since inspired many new results on the large-scale structure of exponential random graph models (see \cite{cook2024typical,ganguly2024subcritical,HL81,mukherjee2023statistics,winstein2025concentration,mukherjee2020degeneracy,Winsteinclt25,BCM24,BNN24,EG18,PN04,RR19,RY13,Yin13,FLSW25,FLSZ24,magnanini2026ergmsblockmodels,bhamidiWeightedExponentialRandom2018,demuse2018phase} and references therein).

Many networks arising in applications, however, contain several distinct types of relations. An edge may encode, for example, friendship, collaboration, communication, rivalry, or some combination of these relations. This naturally leads to the notions of colored/multirelational or multiplex network models, where each color corresponds to a different type of relation. 
In the context of ERGMs, models with multiple relations have been studied in \cite{caimo2020multilayer,krivitsky2020multilayer,pattison1999logit,robins1999logit,koehly2005random,chen2021statistical,fritz2025exponential}, among others. Recent work on the limit theory of weighted/decorated graphs shows that multiplex networks can be naturally embedded within the broader framework of colored graphs (see \cite{multiplexconvergence,abraham2023probabilitygraphons} and the references therein). Indeed, a network with $r \geq 1$ distinct types of binary relations (an $r$-layer multiplex network) may be equivalently represented as an edge-coloring of the complete graph with $2^r$ colors, where each color corresponds to a subset of $[r]:=\{1, 2, \ldots, r\}$ that encodes precisely which relations are present between a pair of nodes. Thus, the colored-graph viewpoint provides a unified representation of multiplex data, and colored ERGMs provide natural exponential-family models for capturing interactions among the different relations.
To make this precise, for an integer $n \geq 1$, denote by $K_n$ the complete graph with vertex set $[n] :=\{1, 2, \ldots, n\}$. Let 
$$\binom{[n]}{2} := \{(u,v): 1 \leq u < v \leq n\}$$
the collection of two-element subsets of $[n]$, which we also identify with the edge set of $K_n$. Given a fixed integer $k \geq 1$, a \emph{$k$-coloring} of $K_n$ is a function $\mathscr{X}: \binom{[n]}{2} \to [k]:=\{1, 2, \ldots, k\}$, where $\mathscr{X}((u,v))$ denotes the color of the edge $(u,v)$. For each $a \in [k]$, denote by $E_a(\mathscr{X}) = \{(u, v) \in \binom{[n]}{2}: \mathscr{X}((u, v)) = a \}$, the collection of edges with color $a$ in the coloring $\mathscr{X}$, and $K_n^{(a)} = ([n], E_a(\mathscr{X}))$ the subgraph of $K_n$ containing all the edges with color $a$. Clearly, $\sum_{a=1}^k |E_a(\mathscr{X})| = \binom{n}{2}$ is the total number of edges in $K_n$. Let $\mathcal C_n^k$ denote the collection of all $k$-colorings of the complete graph $K_n$.  A $k$-{\it colored exponential random graph model} (ERGM) is a distribution on $\mathcal C_n^k$ with probability mass function of the following form: For $\mathscr{X} \in \mathcal C_n^k$, 
\begin{align}\label{eq:Fun}
\P_n(\mathscr{X}) = \frac{1}{Z_n} \exp{\left\{ \sum_{s=1}^r \beta_s F_s(\mathscr{X}) \right\} } , 
\end{align}
where $\bm \beta = (\beta_1, \beta_2, \ldots, \beta_r)$ is a vector of real parameters, $F_1, F_2, \ldots, F_r$ are real-valued functions on $\mathcal C_n^k$, and $Z_n$ is the partition function that is determined by the condition $\sum_{\mathscr{X} \in \mathcal C_n^k} \P(\mathscr{X}) = 1$. The functions $\{F_s\}_{s \geq 1}$ are usually taken to be counts of various colored subgraphs in $K_n$ or, more generally, any collection of `continuous functions' on an appropriate limiting space of colored graphs (see Section \ref{sec:notation-prev-results} for the formal definitions).

\begin{remark}[Classical ERGMs]  
The usual uncolored ERGM is recovered when \(k=2\), by interpreting one color as an edge and the other as a non-edge. More formally, when $k=2$, the space $\mathcal C_n^2$ can be identified with $\mathcal G_n$, the collection of all simple graphs on $n$ labeled vertices, through the map $ \mathcal T: \mathcal C_n^2 \rightarrow \mathcal G_n$, where $\mathcal T(\mathscr{X}) = ([n],E_1(\mathscr{X}))$, for $\mathscr{X} \in \mathcal C_n^2$. Note that the inverse map $\mathcal T^{-1}: \mathcal G_n \rightarrow \mathcal C_n^2$ is given by: 
$$\mathcal T^{-1}(G)(u,v) = 
\begin{cases} 
1, & \text{if } (u,v)\in E(G),\\ 
2, & \text{otherwise} ,  
\end{cases} $$ 
for $G \in \mathcal G_n$. In other words, edges colored $1$ in $K_n$ are identified with the present edges of $G$, while edges colored $2$ are identified with the absent edges of $G$. 
\end{remark}

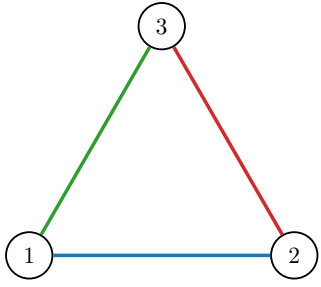
\begin{figure}[htbp]
\centering
\resizebox{0.25\textwidth}{!}{%
\begin{tikzpicture}[
    every node/.style={font=\small},
    vertex/.style={
        circle, draw=black, fill=white, thick,
        minimum size=7mm, inner sep=0pt
    }
]

\definecolor{colone}{RGB}{215,39,40}
\definecolor{coltwo}{RGB}{31,119,180}
\definecolor{colthree}{RGB}{44,160,44}

\node[vertex] (v1) at (0,0) {$1$};
\node[vertex] (v2) at (4,0) {$2$};
\node[vertex] (v3) at (2,3.46) {$3$};

\draw[coltwo, line width=1.5pt] (v1) -- (v2);
\draw[colone, line width=1.5pt] (v2) -- (v3);
\draw[colthree, line width=1.5pt] (v3) -- (v1);

\end{tikzpicture}
}
\caption{ \small{ A rainbow triangle. }}
\label{fig:rainbow_triangle}
\end{figure}

Analogous to the role of graphons in graph limit theory, the natural limiting
objects for dense $k$-colored graphs are $k$-\emph{colored probability
graphons}, which arise as the finite-decoration instance of the emerging notion of
probability graphons
\cite{abraham2023probabilitygraphons,dionigiLargeDeviationsProbability2025}. Formally, a $k$-colored probability graphon is a vector-valued function
$\bm W=(W_1,\ldots,W_k)$, where each
$W_a:[0,1]^2\to[0,1]$ is measurable and symmetric, for $a\in[k]$, and $\sum_{a=1}^k W_a(x,y)=1$, for almost every $(x,y)\in[0,1]^2$. Intuitively, $\bm W(x,y)$ describes the
limiting probability distribution of the color assigned to the edge joining
the continuum vertices $x$ and $y$. Using this framework, we develop a general asymptotic theory for understanding the asymptotic structure of colored graphs sampled from \eqref{eq:Fun}. A particular example that we will use to illustrate our general
results is the \emph{rainbow triangle} ERGM. This is a probability distribution
on \(\mathcal C_n^3\), the space of $3$-edge-colorings of the complete graph on
$n$ vertices, in which triangles whose three edges have pairwise distinct
colors (rainbow triangles; see Figure~\ref{fig:rainbow_triangle}) are rewarded or penalized. More precisely, for parameters
$(\beta_1,\beta_2,\beta_3,\beta_4)\in\mathbb R^4$, one has: 
\begin{align}\label{eq:rbn}
\mathbb P_n(\mathscr X) = \frac{1}{Z_n} \exp\left\{ 2\sum_{a=1}^3\beta_a|E_a(\mathscr X)| + \frac{\beta_4}{n}\triangle_{\mathrm{rb}}(\mathscr X) \right\},
\end{align}
for $\mathscr X \in \cC^3$, where \(\triangle_{\mathrm{rb}}(\mathscr X)\) denotes the number of unordered rainbow triangles in \(\mathscr X\) (see \eqref{eq:tWrainbow} for the formal definition). The normalization of the Hamiltonian ensures a nontrivial scaling
limit, while the constant factors appearing in the model arise naturally from
the corresponding homomorphism-density normalization (see
\eqref{eq:empirical_rainbow_exact}). Note that, when $\beta_4=0$, the model \eqref{eq:rbn} assigns colors to the
edges of $K_n$ independently, with
$$
\P(\mathscr{X}(u,v)=a)
=
\frac{e^{2\beta_a}}
{e^{2\beta_1}+e^{2\beta_2}+e^{2\beta_3}},
$$
with $a\in[3].$ For $\beta_4\neq 0$, the color assignments across edges are no longer
independent, with the parameter $\beta_4$ controlling the extent to which the model
favors or suppresses rainbow triangles. Specifically, positive values of
$\beta_4$ encourage {\it heterogeneous triadic closure}, that is, two-colored
wedges are more likely to close into rainbow triangles. Conversely, negative
values of $\beta_4$ penalize rainbow triangles, and in the limit the
model concentrates on rainbow-avoiding (Gallai-type) colorings (see
Section~\ref{sec:rt} for details). 
Figure~\ref{fig:rainbowbeta} shows four finite-size samples from the model \eqref{eq:rbn}, illustrating these distinct regimes. For $\beta_4=0.5$, the sampled coloring remains visually close to a homogeneous random coloring, with no discernible macroscopic block structure. For $\beta_4=18$, the coloring exhibits a bipodal (2-block) structure: one color dominates inside two vertex sets, while the other two colors occur mainly between them. For $\beta_4=50$, the sample has a pronounced 4-block structure,  with nearly monochromatic connections between distinct
blocks and random-like connections within the blocks. This aligns with the top level of the iterated $K_4$ blowup structure that appears in the positive zero-temperature limit as a maximizer of the rainbow triangle density (see Theorem~\ref{thm:concentrationrainbow}). On the other hand, for $\beta_4=-10$, one color is strongly suppressed and the sampled coloring is nearly bichromatic and, consequently, nearly rainbow-avoiding, as predicted in the negative zero-temperature regime (see Theorem~\ref{thm:concentration_rainbow_negative}).  Thus, the model \eqref{eq:rbn} can capture and distinguish qualitatively different forms of organization in multilayer networks, from homogeneous mixing to structured community formation and, at the opposite extreme, to the segregation or suppression of particular types of interactions.
It is worth noting that the samples shown in Figure~\ref{fig:rainbowbeta} are illustrations of the predicted structures, not conclusions about the mixing time or the relative stationary weights of different metastable regions.

\begin{figure}[tbp]
    \centering
    \makebox[\textwidth][c]{%
    \includegraphics[width=0.26\textwidth]{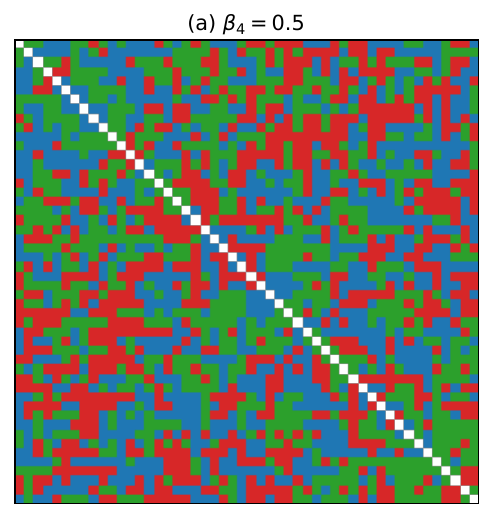}%
    \hspace{0.004\textwidth}%
    \includegraphics[width=0.26\textwidth]{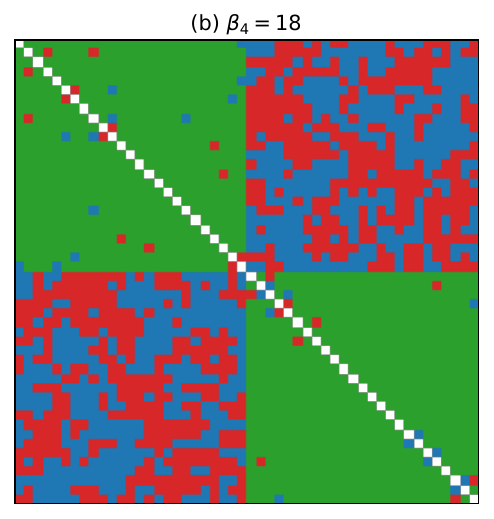}%
    \hspace{0.004\textwidth}%
    \includegraphics[width=0.26\textwidth]{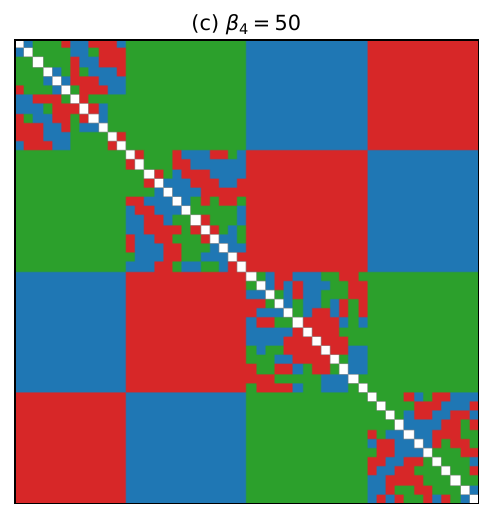}%
    \hspace{0.004\textwidth}%
    \includegraphics[width=0.26\textwidth]{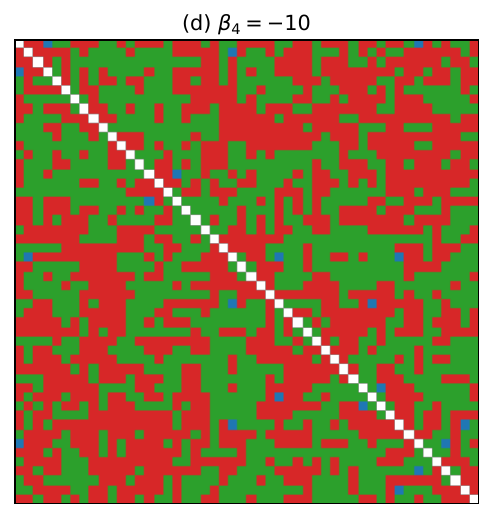}%
    }
    \caption{\small{ Representative configurations from the
    rainbow triangle ERGM
    \eqref{eq:rbn}, with $n=50$, $\beta_1=\beta_2=\beta_3=0$, and
    (a) $\beta_4=0.5$, (b) $\beta_4=18$, (c) $\beta_4=50$, and
    (d) $\beta_4=-10$.  } 
    }
    \label{fig:rainbowbeta}
\end{figure}

\subsection{Summary of Results}

The formal statements of our results are given in Section~\ref{sec:color}.
Here, we provide a high-level overview of the main results and the organization
of the paper.

\begin{itemize} 

\item We begin by showing that the limiting free energy (normalized log-partition function) of colored ERGMs of the form \eqref{eq:Fun} can be expressed as an entropy-tilted variational problem on the space of $k$-color probability graphons (see Theorem \ref{thm:variational_formula}). This is a straightforward consequence of the large deviation principle for finitely decorated probability graphons (recalled in Theorem~\ref{THM:LDP_prob_graphons}), combined with Varadhan's lemma. A further consequence is that a typical coloring sampled from \eqref{eq:Fun} concentrates around the set of maximizers of the limiting variational problem (see Corollary~\ref{cor::expconc}). 
 Thus, understanding the typical structure of a coloring sampled from a colored ERGM reduces to analyzing the maximizers of this variational problem.

\item  A particularly important case in the analysis of the variational problem is the \emph{replica-symmetric} regime, in which the variational problem is solved by constant $k$-color probability graphons. Such graphons correspond to product colorings for which distinct edges are asymptotically independent. Replica symmetry therefore identifies parameter regimes in which the interaction modifies only the limiting color frequencies, without generating nontrivial spatial structure. In contrast with classical uncolored ERGMs, the replica-symmetric regime for colored models is substantially richer, due to the interactions among colors and the competition induced by the simplex constraint. We illustrate this in Section~\ref{sec:starmonotone} by identifying several general families of colored ERGMs for which replica symmetry holds. 

\item Next, in Section~\ref{sec:euler_lagrange}, we study the structure of optimizers for general colored ERGMs, beyond the explicitly solvable replica-symmetric classes, by deriving the Euler--Lagrange equations for the variational problem and establishing a high-temperature uniqueness result. The Euler--Lagrange equations characterize the maximizers through a nonlinear fixed-point map involving the softmax function (see Theorem~\ref{thm:EL-prob-graphon}). Then using a contraction argument in an ``oscillation'' seminorm adapted to the shift invariance of the softmax map, we obtain a general high-temperature uniqueness theorem. In particular, Theorem~\ref{thm:highT_colored} gives an explicit parameter regime in which the variational problem has a unique constant maximizer.

\item In Section \ref{sec:temperatureasymptotics}, we study the complementary zero-temperature regime, in which one of the parameters in \eqref{eq:Fun} diverges to $\pm\infty$. Theorem~\ref{ThmConcentrationBeta} establishes a general two-level selection principle for determining the asymptotic structure of the model in this regime. First, the model concentrates on the graphons that maximize the leading energy term associated with the diverging parameter. Among these ground states, the lower-order energy terms, in competition with the entropy functional, then act as a tie-breaker. Thus, the analysis reduces first to solving an extremal graph problem and then to maximizing the lower-order energy-entropy functional over its solution set. We illustrate this principle with two examples.

\begin{itemize}

\item The first is the \emph{induced wedge ERGM}, in which the sufficient statistic
consists of the induced \(2\)-star density together with a baseline
edge-density term. This model has a natural interpretation in signed networks:
if edges represent negative relations and non-edges represent positive
relations, then an induced wedge corresponds to the structural-balance
principle that ``the enemy of my enemy is my friend'' \cite{fritz2025exponential,feng2022testing} (see
Remark~\ref{remark:wedge} for further discussion). Since the induced wedge
statistic is not monotone with respect to the usual pointwise order on graphons,
nontrivial asymptotic structure can emerge in the zero-temperature limit.
Specifically, our two-step selection principle yields the following. As
\(\beta_2\to+\infty\), the model favors graphons maximizing the induced
\(2\)-star density, and the optimizing graphons converge to the balanced
complete bipartite graphon (see Theorem~\ref{thm:wedge-beta-plus-infty}). In the
opposite regime, as \(\beta_2\to-\infty\), the model favors graphons with zero
induced \(2\)-star density, and the optimizing graphons converge to certain
cluster graphons, that is, disjoint unions of clique graphons (see
Theorem~\ref{thm:wedge-ground-states}).

\item Next, we consider the rainbow triangle ERGM \eqref{eq:rbn}. As \(\beta_4\to+\infty\), the leading-order term favors 3-color probability graphons that maximize the rainbow triangle density. Consequently, invoking  results of \citet{baloghRainbowTrianglesThreecolored2017} on the maximum possible density of rainbow triangles, we show in Theorem~\ref{thm:concentrationrainbow} that the variational maximizers of the rainbow triangle ERGM converge to the limiting iterated \(K_4\)-blowup probability graphon (which is consistent with the simulation in Figure~2(c)). On the other hand, as $\beta_4\to-\infty$, the leading-order constraint is the absence of rainbow triangles, so the ground-state contains graphon limits of Gallai colorings. Once the entropy and edge-density terms are taken into account as the lower-order tie-breaker, the selected optimizers are constant probability graphons supported on only two colors (see Theorem~\ref{thm:concentration_rainbow_negative}). This bichromatic structure is also reflected in Figure~\ref{fig:rainbowbeta} (d). 

\end{itemize}

\item Finally, in Section~\ref{sec:symmetry}, we show that both the induced-wedge and
rainbow triangle ERGMs exhibit finite-temperature symmetry breaking. For each
model, we identify a regime of replica symmetry, obtained either from the
general high-temperature uniqueness theorem described above or from a more
refined model-specific analysis, and a separate regime of symmetry breaking,
established by comparing the optimal constant solution with a suitably chosen
block-structured competitor. In both cases, a nontrivial interval remains
between the rigorous replica-symmetric and symmetry-breaking regimes. Numerical experiments are reported in Section \ref{sec:finite_beta_simulations} to explore the structure of the optimizers in these intermediate regions.

\end{itemize}

In Section~\ref{sec:notation-prev-results}, we review the relevant background
on decorated probability graphons. The remaining sections are devoted to the
proofs of the results described above. Several technical ingredients used
throughout the paper are collected in the appendix.

\section{Background on Finitely Decorated Probability Graphons}\label{sec:notation-prev-results}

To analyze the asymptotic behavior of colored ERGMs, we require the framework of
colored graph convergence, which may be viewed as a special case of the more
general theory of decorated graph convergence. Specifically, given a Polish space $Z$, a $Z$-decoration of a simple graph $F= (V(F), E(F))$ is a map $g: E(F) \rightarrow Z$. This formalism provides a natural language for studying graphs whose edges carry additional structure, such as colors, weights, or other marks. 
The limit theory of decorated graphs was initiated by Lov\'{a}sz and Szegedy 
\cite{lovasz2010limits} (see also \cite[Chapter~17]{LovaszGraphLimits}) and has been subsequently generalized in several directions (see, for example,
\cite{falgasravry2016multicolour,rollin2023dense,KunLovSze22,zucal2024probabilitygraphonsrightconvergence,abraham2023probabilitygraphons}). In particular, the limit of a sequence of $Z$-decorated graphs can be described in terms of a {\it probability graphon} (a term introduced in \cite{abraham2023probabilitygraphons}), which is a measurable and symmetric map: 
\begin{align*}
\bm{W} : [0, 1]^2 \rightarrow  \mathcal P(Z) , 
\end{align*}
where $\mathcal P(Z)$ is the space of probability measures on $Z$. Here, symmetric means $\bm{W}(x, y; A) = \bm{W}(y, x; A)$ for all $x, y \in [0, 1]$ and for every measurable subset $A$ of $Z$. One can interpret a probability graphon as follows: for any two `vertices' $x, y \in [0, 1]$, the weight of the edge between $x$ and $y$ is distributed as the probability measure $\bm{W}(x, y; \mathrm dz)$. 
In the setting of finitely colored graphs, $Z= [k]$, for some integer $k\ge 2$, equipped with the discrete topology, and 
$$\mathcal P([k])\cong \Delta_k := \left\{ \bm \mu =(\mu_1,\dots,\mu_k)\in[0,1]^k: \sum_{a=1}^k \mu_a=1 \right\}.$$ 
Define 
$$
\mathcal W_k:= \left\{ \bm W\colon [0,1]^2\to\Delta_k: \bm W \text{ is measurable and symmetric} \right\}.
$$
The elements of $\mathcal W_k$ will be referred to as {\it $k$-colored probability graphons}. Note that a  $k$-colored probability graphon $\bm W \in \mathcal W_k$ can be equivalently represented as $\bm W(x,y)=(W_1(x,y),\dots,W_k(x,y))$, where 
$W_1, W_2, \ldots, W_k \colon [0,1]^2\to[0,1]$, are a collection of measurable and symmetric functions satisfying $\sum_{a=1}^k W_a(x,y)=1$, for almost every $(x,y)\in[0,1]^2$. Furthermore, given a $k$-coloring $\mathscr X\in\mathcal C_n^k$, we define the empirical $k$-colored probability graphon associated with $\mathscr X$ as $\bm W^{\mathscr X} = \bigl(W^{\mathscr X}_1,W^{\mathscr X}_2,\ldots,W^{\mathscr X}_k\bigr) \in\mathcal W_k$, where, for $1 \leq u \ne v \leq n$ and $a\in[k]$, 
\begin{align}\label{eq:Wa} 
W^{\mathscr X}_a(x,y) := \bm 1\{\mathscr X((u,v))=a\}, \quad \text{ whenever } (x,y)\in \left(\frac{u-1}{n},\frac{u}{n}\right] \times \left(\frac{v-1}{n},\frac{v}{n}\right] . 
\end{align}
Further, when $(x,y)\in \left(\frac{u-1}{n},\frac{u}{n}\right]^2$, for $1\leq u\leq n$, we set $\bm W^{\mathscr X}(x,y) = (\frac1k,\frac1k,\ldots,\frac1k)$.

\begin{remark}\label{remark:graphoncolor}
The usual notion of a graphon is recovered from the colored setting when
$k=2$. Recall that a graphon $W$ is a measurable symmetric function
$W:[0,1]^2\to[0,1]$. A $2$-colored probability graphon
$\bm W\in\mathcal W_2$ can be represented as $\bm W(x,y)=(W(x,y),1-W(x,y))$, 
where $W$ is a graphon. Consequently, $\mathcal W_2$ can be identified with
the standard graphon space of all measurable symmetric functions
$W:[0,1]^2\to[0,1]$. 
\end{remark}

\subsection{Homomorphism Densities and Cut Distance }

Given a finite, simple graph $H=(V(H),E(H))$, a \emph{$k$-coloring} of $H$ is a function $\sigma: E(H) \to [k]:=\{1, 2, \ldots, k\}$, where $\sigma((u,v))$ denotes the color of the edge $(u,v)\in E(H)$. In this case, $H$ is called a \emph{$k$-colored graph}, or simply a
\emph{colored graph} when the value of $k$ is clear from the context, and is
denoted by $\bm H=(H,\sigma)$. We say that a colored graph $\bm H$ is \emph{monochromatic} if all of its edges are assigned the same color. In this case, we will omit the boldface notation and simply write $H$ for $\bm H$, interpreting $H$ either as an
uncolored graph or as a monochromatic colored graph, depending on the context. For each $a\in[k]$, let $E_a(H)=\{(u,v)\in E(H):\sigma((u,v))=a\}$ denote the set of edges of color $a$ in the coloring $\sigma$, and let $H^{(a)}=(V(H),E_a(H))$ denote the subgraph of $H$ consisting of all edges of color $a$. Thus,  $H^{(a)}$ is a monochromatic subgraph of $H$ of color $a$. Clearly, $|E(H)| = \sum_{a=1}^k |E_a(H)|$ is the total number of edges in $H$. 
The homomorphism density of a $k$-colored graph $\bm H = (H, \sigma)$ in a $k$-colored probability graphon $\bm W = (W_1, W_2, \ldots, W_k) \in \mathcal W_k$ is defined as (see \cite[Definition 7.1]{abraham2023probabilitygraphons}): 
\begin{equation}\label{Def:homDens}
 t( \bm H, \bm W)
 := \int_{[0,1]^{V(H)}} \prod_{ (u, v) \in E(H)} W_{\sigma((u, v))} (x_u, x_v) \prod_{v\in V(H)} \mathrm {d} x_v.
\end{equation}

\begin{remark}
Note that if $\bm H = (H, \sigma)$  is a 2-colored graph and $\bm W = (W, 1-W)$ is a 2-colored probability graphon, for some graphon $W: [0, 1]^2\rightarrow [0, 1]$ (recall Remark \ref{remark:graphoncolor}), then
\begin{align}\label{Def:homDensInd}
 t(\bm H, \bm W) & := \int_{[0,1]^{V(H)}} \prod_{\substack{(u, v) \in E(H)\\ \sigma((u, v)) = 1}} W(x_u, x_v)  \prod_{\substack{(u, v) \in E(H)\\ \sigma((u, v)) = 2}} (1- W(x_u, x_v) ) \prod_{v\in V(H)} \mathrm {d} x_v . 
\end{align}
In particular, if we consider the monochromatic coloring $\sigma((u, v)) = 1$ for each $(u, v) \in E(H)$, then 
\begin{align}\label{eq:tHW}
t(\bm H, \bm W) := \int_{[0,1]^{V(H)}} \prod_{(u, v) \in E(H)} W(x_u, x_v) \prod_{v\in V(H)} \mathrm {d} x_v  := t(H, W) , 
\end{align}
is the  homomorphism density of the (uncolored) graph $H$ in the graphon $W$ (see \cite[Section 7.2]{LovaszGraphLimits}). 
Further, if one considers a 2-coloring $\sigma_H$ of the complete graph $K_{R}$, where $R= |V(H)|$, such that $K_{R}^{(1)}$ is isomorphic to $H$, then denoting this colored graph by $\bm K_R$, \eqref{Def:homDensInd} becomes, 
\begin{align}\label{eq:inducedW}
t( \bm K_{R}, \bm W) = \int_{[0,1]^{V(H)}} \prod_{(u, v) \in E(H)} W(x_u, x_v)  \prod_{(u, v) \notin E(H)} (1- W(x_u, x_v) ) \prod_{v\in V(H)} \mathrm {d} x_v =: t_{\mathrm{ind}}(H, W)  , 
\end{align}
which is the induced homomorphism density of the (uncolored) graph $H$ in the graphon $W$ (see \cite[Section 7.2]{LovaszGraphLimits}). Hence, \eqref{Def:homDens} is the natural  analogue of the classical (uncolored) graphon homomorphism density to the colored setting. 
\end{remark}

As in the usual theory of graph convergence, the convergence of a sequence of $k$-colorings is defined in terms of homomorphism densities of colored subgraphs. Specifically, a sequence of $k$-colorings $\{\mathscr{X}_n\}_{n \geq 1}$ is said to be {\it left convergent} if $\lim_{n \rightarrow \infty}t(\bm H, \bm W^{\mathscr{X}_n})$ exists for every colored graph $\bm H$. To metrize this convergence one needs to appropriately define a notion of distance between $k$-colored probability graphons. To this end, recall that the {\it cut norm of a graphon} $W: [0, 1]^2 \rightarrow [0, 1]$ is defined as (see \cite[Section 8.2]{LovaszGraphLimits}): 
$$\|W\|_{\square} :=\sup_{S,T\subset [0,1]}\left|\int_{S\times T} W(x,y) \mathrm dx\mathrm dy\right|.$$
Then the {\it cut norm of a $k$-colored probability graphon} $\bm W = (W_1, W_2, \ldots, W_k) \in \mathcal W_k$ is defined as (see \cite[Section 3.2]{multiplexconvergence}): 
\begin{align*}
\| \bm W \|_{\square} := \sum_{a = 1 }^k \| W_a \|_\square . 
\end{align*}
For $\phi\in\mathcal M$, where $\mathcal M$ denotes the collection of
measure-preserving bijections of $[0,1]$ (understood modulo null sets), define
the relabeling of $\bm W=(W_1,\ldots,W_k)\in\mathcal W_k$ by
\[
\bm W^\phi(x,y)
:=
\bm W\bigl(\phi(x),\phi(y)\bigr),
\qquad\text{that is,}\qquad
W_a^\phi(x,y):=W_a\bigl(\phi(x),\phi(y)\bigr),
\quad \text{ for } a\in[k].
\]
Then the cut distance between $\bm U,\bm W\in\mathcal W_k$ is the
pseudometric
\begin{align*}
\delta_\square(\bm U,\bm W)
:=
\inf_{\phi\in\mathcal M}
\left\|\bm U-\bm W^\phi\right\|_\square .
\end{align*}
We say that $\bm U$ and $\bm W$ are \emph{weakly isomorphic}, and write
$\bm U\sim\bm W$, if $\delta_\square(\bm U,\bm W)=0$. Equivalently, $\bm U\sim\bm W$ if and only if $\bm U$ belongs to the
closure, with respect to the cut norm, of the orbit
$\{\bm W^\phi:\phi\in\mathcal M\}$. (In particular, $\delta_\square(\bm W,\bm W^{\phi})=0.$)
Denote the resulting quotient space by $\tilde{\mathcal W}_k:=\mathcal W_k/\!\sim$ and denote the equivalence class of $\bm W$ by $\tilde{\bm W}$. The cut distance on the quotient is defined by
\[
\delta_\square(\tilde{\bm U},\tilde{\bm W})
:=
\delta_\square(\bm U,\bm W),
\]
where $\bm U$ and $\bm W$ are arbitrary representatives of
$\tilde{\bm U}$ and $\tilde{\bm W}$, respectively. This definition
does not depend on the choice of representatives, and
$\bigl(\tilde{\mathcal W}_k,\delta_\square\bigr)$ is a metric space.  The cut distance allows us to define a natural notion of metric convergence
for $k$-colorings as follows. A sequence of $k$-colorings
$\{\mathscr{X}_n\}_{n\geq 1}$ is said to converge to a $k$-colored probability
graphon $\bm W\in\mathcal W_k$ if $\lim_{n\to\infty}
\delta_\square(\wt{\bm W}^{\mathscr{X}_n},\wt{\bm W})=0$. 
From the general convergence theory of decorated graphons, specialized to the
setting of finite colorings, one has the following results \cite{abraham2023probabilitygraphons} (see also \cite{multiplexconvergence}): (1) The metric space $(\wt{\mathcal W}_k, \delta_\square)$ is compact, (2) the functional $t(\bm H, \cdot)$ is continuous for every colored graph $\bm H$, and (3) left convergence of $k$-colorings and convergence in terms of the cut distance are equivalent.

\begin{remark}[Vertex relabelings and color relabelings]
Note that the space $\wt{\mathcal W}_k$ is obtained by
quotienting only with respect to measure-preserving relabelings of the vertex
space, while keeping the color labels fixed. One could, in principle, further
identify graphons that differ only by a global permutation of the colors. 
For instance, given
$\bm W=(W_1,W_2,\ldots,W_k)\in\mathcal W_k$ and $\tau\in S_k$, where
$S_k$ denotes the group of permutations of $[k]$, define $\tau \bm W = ((\tau \bm W)_1, \ldots, (\tau \bm W)_k) \in \mathcal W_k$ as 
\begin{align*}
(\tau \bm W)_a := W_{\tau^{-1}(a)} \quad \text{ for } a\in[k]. 
\end{align*}
Thus, $\tau\bm W$ is the $k$-colored probability graphon obtained from
$\bm W$ by globally permuting the color labels.
In general, $\tau\bm W$ need not be equivalent to $\bm W$ in $\wt{\mathcal W}_k$, since measure-preserving relabelings of the vertex space do not alter the labels of the colors. For color-symmetric observables, one may further quotient by global color permutations. However, unless explicitly stated otherwise,  all uniqueness statements in this paper are understood in the fixed-color quotient $\wt{\mathcal W}_k$.  
\end{remark}

\subsection{Large Deviation Principle}\label{sec:ldp-probability-graphons}

For ${\bm \nu}\in\Delta_k$, denote by $\mathscr{X}(n, {\bm \nu} ) $ the coloring of $K_n$ where each edge is assigned a color independently with distribution ${\bm \nu}$. The random coloring $\mathscr{X}(n, {\bm \nu})$, 
induces a probability measure $\bm \mu_{n,{\bm \nu}}$ on $\mathcal W_k$ through the map $\mathscr{X} \mapsto \bm W^\mathscr{X}$ (recall \eqref{eq:Wa}). This naturally lifts to a probability measure $\wt{\bm \mu }_{n,{\bm \nu}}$ on $\wt{\mathcal W}_k$ through the map 
\begin{equation}\label{eq:mappingGraphsGraphons}
    \mathscr{X} \mapsto \bm W^\mathscr{X} \mapsto \wt{\bm W}^{\mathscr{X}}.
\end{equation} Recently, \cite{dionigiLargeDeviationsProbability2025} derived a large deviation principle (LDP) for random weighted graphs with general decorations, which, specialized to setting of finite-color decorations gives, an LDP for $\wt{\bm \mu }_{n,{\bm \nu}}$. To state this result, we need to introduce a few definitions: 
For $\bm \mu =(\mu_1,\dots, \mu_k), {\bm \nu}=(\nu_1,\dots,\nu_k)\in\Delta_k$, define the relative entropy as: 
\begin{align}\label{eq:divergence}
D_k(\bm \mu \|{\bm \nu}):=\sum_{a=1}^k \mu_a\log\frac{\mu_a}{\nu_a},
\end{align}
with the conventions $0\log 0=0$ and $\mu_a\log(\mu_a/\nu_a)=+\infty$,  whenever $\mu_a>0$ and $\nu_a=0$.
This naturally extends pointwise to $k$-color probability graphons as follows: For  $\bm W\in \mathcal W_k$ and ${\bm \nu}\in\Delta_k$,  
\begin{align}\label{eq:IcolorW}
\mathcal I_{\bm \nu}(\bm W):=\int_{[0,1]^2} D_k\left( \bm W(x,y)\|{\bm \nu}\right)\,\mathrm{d}x\,\mathrm{d}y = \sum_{a=1}^k \int_{[0,1]^2} W_a(x,y) \log \frac{W_a(x,y)}{\nu_a}  \,\mathrm{d}x \,\mathrm{d}y . 
\end{align}
It is shown in \cite{dionigiLargeDeviationsProbability2025} that $\mathcal I_{\bm \nu}$ is well defined on equivalence classes in $\wt{\mathcal W}_k$, and with a slight abuse of notation we denote the induced functional on the quotient space by the same symbol. In particular, $\mathcal I_{\bm \nu}$ is lower semicontinuous and has compact level sets. The LDP for $\wt{\bm \mu }_{n,{\bm \nu}}$ can now be stated as follows:

\begin{theorem}[{\cite[Theorem 1.3]{dionigiLargeDeviationsProbability2025}}]
\label{THM:LDP_prob_graphons}
For every ${\bm \nu}\in\Delta_k$, the sequence of measures $\{\wt{\bm \mu }_{n,{\bm \nu}}\}_{n\ge 1}$ satisfies a large deviation principle on $\wt{\mathcal W}_k$ with speed $\frac{n^2}{2}$ and good rate function $\mathcal I_{\bm \nu}$. Equivalently, for every closed set $\tilde F\subset\wt{\mathcal W}_k$,
$$
\limsup_{n\to\infty}\frac{2}{n^2}\log \wt{\bm \mu }_{n,{\bm \nu}}(\tilde F)
\le -\inf_{\wt{\bm W}\in \tilde F}\mathcal I_{\bm \nu}(\wt{\bm W}),
$$
and for every open set $\tilde U\subset\wt{\mathcal W}_k$,
\begin{align*}
\liminf_{n\to\infty}\frac{2}{n^2}\log \wt{\bm \mu }_{n,{\bm \nu}}(\tilde U)
\ge -\inf_{\wt{\bm W}\in \tilde U}\mathcal I_{\bm \nu}(\wt{\bm W}).
\end{align*}
\end{theorem}

Using the above result, in the next section we will derive a variational
formula for the limiting log-partition function for colored ERGMs, which will enable us to characterize the asymptotic structure of typical colorings.

\section{Main Results}
\label{sec:color}

Let $T\colon \wt{\cW}_k\to\mathbb R$ be a bounded and continuous function on the metric space $(\wt{\cW}_k, \delta_\square)$.
This induces a probability measure on $\mathcal C_n^k$ as follows: For $\mathscr{X} \in \mathcal C_n^k$, 
\begin{align}\label{eq:TW}
\mathbb P_{n}(\mathscr{X}) := \exp\left\{n^2T(\wt{\bm W}^{\mathscr{X}}) - n^2 \psi_n \right\}, 
\end{align}
where
\begin{align}\label{eq:Zn}
\psi_n:= \frac{1}{n^2}\log \sum_{\mathscr{X}\in\mathcal C_n^k}\exp\left\{n^2T(\wt{\bm W}^\mathscr{X})\right\} , 
\end{align}
is the normalized log-partition function (free energy). Note that the measure $\mathbb{P}_n$ defined above can be directly related to the measure defined \eqref{eq:Fun}. Moreover, we can naturally lift the measure $\P_n$ from the space of graphs to the space $\wt{\mathcal W}_k$ pushing forward the measure through the map \eqref{eq:mappingGraphsGraphons}. We will denote the lifted measure by $\tilde{\mathbb{P}}_n$ in the following.

The function $T$ will be referred to as the sufficient statistic of the model \eqref{eq:TW}.  In this paper, we consider sufficient statistics of the form
\begin{align}\label{eq:sufficientTW}
T(\bm W)=\sum_{s=1}^r \beta_s \, t( \bm H_s, \bm W),
\end{align}
where $\beta_1,\dots,\beta_r \in\mathbb R$ and $\bm H_1,\dots, \bm H_r$ are finite, simple $k$-colored graphs.\footnote{Note that \eqref{eq:sufficientTW} is invariant under measure-preserving relabelings. Hence, it descends to a well-defined functional on $\wt{\mathcal W}_k$ by setting $T(\wt{\bm W}) := T(\bm W)$, for any representative $\bm W$ of the equivalence class $\wt{\bm W}$. Hereafter, with a slight abuse of notation, we will sometimes regard $T$ as a functional on $\mathcal W_k$ and sometimes as a functional on $\wt{\mathcal W}_k$.} 
Since homomorphism densities take values in $[0,1]$ and are continuous on $\wt{\mathcal W}_k$, such a functional $T$ is indeed bounded and continuous. For a concrete example, consider the rainbow triangle ERGM introduced in \eqref{eq:rbn}. In this case, the sufficient statistic $T$ for $\bm W = (W_1, W_2, W_3) \in \mathcal W_3$ takes the form: 
\begin{align}\label{eq:rainbowT}
T(\bm W)= \beta_1 t(K_2^{(1)}, \bm W) + \beta_2 t(K_2^{(2)}, \bm W) + \beta_3 t(K_2^{(3)}, \bm W) + \beta_4 t_{\mathrm{rb}\triangle}(\bm W) . 
\end{align}
Here, $K_2^{(a)}$ denotes a single edge colored $a\in[3]$ and consequently, $t(K_2^{(a)},  \bm W) = \int_{[0,1]^2} W_a (x, y) \mathrm{d} x \mathrm{d} y$. Further, for $a,b,c\in[3]$, let $\bm K_3^{(a,b,c)}$ denote the decorated triangle on vertex set $\{1,2,3\}$ whose edges $(1,2)$, $(2,3)$, and $(3,1)$ have colors $a$, $b$, and $c$, respectively. Set $\bm K_3^{\mathrm{rb}}:=\bm K_3^{(1,2,3)}$. Then,
\begin{align}\label{eq:tWrainbow}
t_{\mathrm{rb}\triangle}(\bm W) & := t(\bm K_3^{\mathrm{rb}},\bm W)= \int_{[0,1]^3} W_1(x,y)W_2(y,z)W_3(z,x)\, \mathrm{d}x\, \mathrm{d}y\, \mathrm{d}z ,  
\end{align}
is  the \emph{fixed decorated} rainbow triangle homomorphism density. With this parametrization, \eqref{eq:rainbowT} is consistent with the finite $n$ model described in  \eqref{eq:rbn}. Specifically, with the diagonal-block convention in \eqref{eq:Wa}, one has the exact identities: 
\begin{equation}\label{eq:empirical_rainbow_exact}
t(K_2^{(a)},\bm W^{\mathscr X})
=
\frac{2|E_a(\mathscr X)|}{n^2}+\frac{1}{3n},
\quad \text{ for } a\in[3],
\qquad\text{ and }\qquad
t_{\mathrm{rb}\triangle}(\bm W^{\mathscr X})
=
\frac{\triangle_{\mathrm{rb}}(\mathscr X)}{n^3}
+
\frac{1}{27n^2}.
\end{equation}
The second identity follows by noting that every unordered rainbow triangle contributes exactly once to the fixed decorated density, configurations with exactly two coincident sampled vertices contribute zero, while the $n$ diagonal cubes contribute the final term in
\eqref{eq:empirical_rainbow_exact}. Hence, 
\begin{align*}
n^2\left[
\sum_{a=1}^3\beta_a t(K_2^{(a)},\bm W^{\mathscr X})
+
\beta_4t_{\mathrm{rb}\triangle}(\bm W^{\mathscr X})
\right]
&=
2\sum_{a=1}^3\beta_a|E_a(\mathscr X)| + \frac{\beta_4}{n}\triangle_{\mathrm{rb}}(\mathscr X) + \frac n3\sum_{a=1}^3\beta_a+
\frac{\beta_4}{27}.
\end{align*} 
The last two terms are independent of $\mathscr X$ and can therefore be absorbed into the partition function, showing that \eqref{eq:TW} with sufficient statistic as in \eqref{eq:rainbowT} induces exactly the probability law \eqref{eq:rbn}.

\begin{remark}[Symmetrized rainbow triangle density]
Another alternative way to define the rainbow triangle density is to use the color-symmetrized statistic that counts a rainbow triangle without fixing which edge receives which color:
\begin{align}
\label{eq:tWrainbow_sym}
t_{\mathrm{rb}\triangle}^{\mathrm{sym}}(\bm W) &:=
\sum_{\substack{a,b,c\in[3]\\ \{a,b,c\}=[3]}}
\int_{[0,1]^3}
W_a(x,y)W_b(y,z)W_c(z,x) \mathrm dx \mathrm dy \mathrm dz =6 t_{\mathrm{rb}\triangle}(\bm W). 
\end{align}  
The last identity above follows by noting that every permutation of the three edges of a
triangle is induced by a permutation of its vertices, together with the
symmetry of the kernels $W_a$, for $a \in [3]$. If one places a coefficient $\bar\beta_4$ in front of the symmetrized statistic \eqref{eq:tWrainbow_sym}, then the same model as in \eqref{eq:rainbowT} is obtained under the reparametrization $\beta_4=6\bar\beta_4$. We will need to use this symmetrized way of counting rainbow triangles to compare our results with the flag algebra literature in Section \ref{sec:rt}.
\end{remark}

\begin{remark}[Relation to weighted ERGMs]
A weighted exponential random graph model is an ERGM in which each edge carries
a real-valued weight, and the sufficient statistics are typically weighted
homomorphism densities
\cite{bhamidiWeightedExponentialRandom2018,demuse2018phase}.
When the base measure $\mu$ has finite support, weighted ERGMs can be naturally
embedded into the colored ERGMs framework through the contraction of probability
graphons; see
\cite[Section~5.1 and Corollary~7.10.1]
{zucal2024probabilitygraphonspvariablesequivalent}.
In terms of decorated homomorphism densities, these statistics use the same
scalar decoration on every edge of the test graph. The colored setting is
more general, since different edges may involve different color coordinates
and the resulting statistics need not factor through a single scalar
contraction, as illustrated by the rainbow-triangle statistic.
A theory of ERGMs with general edge decorations was already raised as a natural direction in \cite[Remark~4.5]{bhamidiWeightedExponentialRandom2018}; the present framework provides such a theory in the finite-decoration setting.
\end{remark} 

\subsection{Limiting Free Energy and Concentration of Measure}
\label{sec:Zvariationalpf}

Define the function $I_k : \Delta_k \rightarrow \R$ as $
I_k(\bm \mu ):=\sum_{a=1}^k \mu_a\log \mu_a$, which is the negative of the usual Shannon entropy. This can be extended to $k$-color probability graphons as: 
\begin{align}\label{eq:IW}
\mathcal I(\bm W):= \int_{[0,1]^2} I_k\left(\bm W(x,y)\right)\,\mathrm{d}x\,\mathrm{d}y = \sum_{a=1}^k \int_{[0,1]^2} W_a(x,y) \log W_a(x,y)  \,\mathrm{d}x\,\mathrm{d}y  ,  
\end{align}
for $\bm W \in \mathcal W_k$. As for the relative entropy \eqref{eq:IcolorW}, the probability graphon entropy $\mathcal I$ is well defined on equivalence classes in $\wt{\mathcal W}_k$ and with a slight abuse of notation we will denote the induced functional on the quotient space by the same symbol. Our first result is the following variational representation of the limiting free energy:

\begin{theorem}\label{thm:variational_formula}
Let $T\colon \wt{\mathcal W}_k\to\mathbb R$ be a bounded and continuous function and $\psi_n$ and $\mathcal I$ be as defined in \eqref{eq:Zn} and \eqref{eq:IW}, respectively. Then
\begin{equation}
\label{eq:varformula}
\lim_{n\to\infty}\psi_n
=
\sup_{\wt{\bm W}\in\wt{\mathcal W}_k}
\left(T(\wt{\bm W})-\frac12\mathcal I(\wt{\bm W})\right).
\end{equation}
\end{theorem}

The proof of Theorem \ref{thm:variational_formula} is given in Section \ref{sec:Znpf}. This is the analogue of \cite[Theorem 3.1]{ChatterjeeDiaconisExponential2013}  for uncolored ERGMs in the colored setting. The proof is a direct application of the large deviation principle from \Cref{THM:LDP_prob_graphons} and Varadhan's Lemma (see Theorem \ref{thm::EVaradhanLemma}). Note that the set of maximizers
\begin{equation}
\label{eq:main_maximizer_set}
\wt{\mathcal S}_T
:=
\argmax_{\wt{\bm W}\in\wt{\mathcal W}_k}
\left\{
T(\wt{\bm W})-\frac12\mathcal I(\wt{\bm W})
\right\}
\end{equation}
is nonempty and compact. This is because $T$ is continuous, $\mathcal I$ is lower semicontinuous, and the functional $\tilde {\bm W }\rightarrow T(\tilde {\bm W})-\frac12\mathcal I(\tilde {\bm W})$ is upper semicontinuous on the compact quotient space
$\wt{\mathcal W}_k$. We can now identify the limiting structure of a coloring sampled from \eqref{eq:TW} as follows:

\begin{corollary}
\label{cor::expconc}
Let $\wt{\cS}_T$ be as defined in \eqref{eq:main_maximizer_set} and for each $n \geq 1$, let $\mathscr{X}_n \sim \P_{n}$ be a random $k$-coloring sampled from \eqref{eq:TW}. Then for every $\delta > 0$, there exists $C : = C(\delta)> 0$ such that for all sufficiently large $n$,
$$\P_n\left(\delta_{\square}(\wt{\bm{W}}^{\mathscr{X}_n} ,\wt{\cS}_T) > \delta\right) \leq e^{-Cn^2}.$$
\end{corollary}

The proof of Corollary \ref{cor::expconc} is given in Section \ref{sec:concentrationpf}. The result shows that the asymptotic structure of a typical coloring sampled from \eqref{eq:TW} must be close to the set $\wt{\mathcal S}_T$ with high probability. Specifically, if $\wt{\cS}_T=\{\wt{\bm{W}} \}$ is a singleton, then Corollary \ref{cor::expconc} shows that $\delta_{\square}(\wt{\bm{W}}^{\mathscr{X}_n} ,\wt{\bm{W}})\rightarrow 0$ in probability, that is, \ $\wt{\P}_{n}$ (the probability measure lifted to the $k$-colored probability graphon space $\wt{\mathcal W}_k$), weakly converges to $\delta_{\wt{\bm{W}}}$ in the metric space $(\wt{\mathcal W}_k, \delta_{\square})$.

\subsection{Replica Symmetric Families }
\label{sec:starmonotone}

The results of the previous section show that understanding the typical structure of a coloring sampled from a colored ERGM reduces to analyzing the maximizers of the limiting variational problem in \eqref{eq:varformula}. A particularly important regime is when the variational problem is solved by constant $k$-colored probability graphons, where $\bm W \in \mathcal W_k$ is of the form 
$$\bm W(x,y)= \bm \mu= (\mu_1,\mu_2,\ldots,\mu_k)\in\Delta_k, \qquad \text{ for } (x,y)\in[0,1]^2. $$ 
For example, if \eqref{eq:varformula} has a unique maximizer among all $k$-colored probability graphons and this maximizer is some $\bm \gamma \in \Delta_k $, then a random coloring sampled from \eqref{eq:TW} is asymptotically close to the product coloring $\mathscr X(n, \bm\gamma)$, in which the colors of different edges are assigned independently according to $\bm \gamma$. More generally, if the set $\Gamma$ of maximizers of \eqref{eq:varformula} consists of several constant $k$-colored probability graphons, then the model concentrates asymptotically near the family of product colorings $\{\mathscr X(n,\bm \gamma):\bm \gamma\in\Gamma\}$. In this case, the limiting behavior is described by a mixture over the constant maximizers, rather than by a single product coloring. This is referred to as the {\it replica symmetric regime}, where the typical large-scale structure of the coloring is spatially homogeneous, determined only by the limiting color distribution on a single edge. 

In this section, we identify the following families of colored ERGMs for which replica symmetry holds. 

\begin{itemize}

\item The first consists of models whose sufficient statistics are linear
    combinations of homomorphism densities of colored stars. These models are
    exactly solvable and replica symmetric for all parameter values (Section~\ref{sec:colorstar}).

    \item The second family consists of models with monotone chromatic interactions. These are obtained by projecting the ERGM sufficient statistic onto finitely many scalar graphons and considering positive linear combinations of homomorphism densities of these scalar graphons. By a generalized H\"older bound, the resulting variational problem under these assumptions has only constant maximizers, giving another replica-symmetric class (see Section~\ref{sec:monochromaticTW}). 

\item The third consists of negative Sidorenko interactions. In this case, Sidorenko-type lower bounds for the projected homomorphism densities yield upper bounds on the ERGM variational problem when the corresponding coefficients are negative. This gives replica symmetry for a class of interactions with negative parameters (see Section~\ref{sec:negativereplicasymmetry}). 

\item The fourth consists of models whose sufficient statistics are non-negative linear combinations of rainbow-common colored graphs. Rainbow commonness implies that the relevant homomorphism densities are maximized by the uniform $k$-color graphon, and the entropy term selects the same constant graphon. This yields a simple sufficient condition for replica symmetry (see Section~\ref{sec:rc}). 

\end{itemize}

In each of the above cases, the main step is to establish an upper-bound on the variational problem \eqref{eq:varformula} of the following form:
\begin{equation}
\label{eq:repsym}
\left(T(\bm W) - \frac{1}{2}\cI(\bm W)\right) \leq \sup_{\bm \mu  \in \Delta_k} P(\bm \mu ) , 
\end{equation}
for all $\bm W \in \cW_k$ and for some function $P:\Delta_k \to \R$. This upper bound is specifically chosen so that it is attained when $W \equiv \bm \mu ^*$, where $\bm \mu ^*$ maximizes $P$. The second step of the argument is to show that the inequality above is strict when $W$ is non-constant.

\subsubsection{Colored Star Models}
\label{sec:colorstar}

In the classical uncolored setting, \cite[Theorem~6.4]{ChatterjeeDiaconisExponential2013}
shows that star models, whose sufficient statistic $T$ is a linear combination
of homomorphism densities of stars, are always replica symmetric. In
this section, we establish an analogous result in the colored setting. To this end, for any integer $s \geq 1$, denote by $K_{1, s}$ the star-graph with $s$ leaves. Given a vector $\bm s = (s_1, s_2, \ldots, s_k)$ of non-negative integers with $\sum_{a=1}^k s_a := |\bm s |$, denote by $\bm K_{1, \bm s}$ a $k$-colored $|\bm s|$-star in which exactly $s_a$ edges are assigned the color $a \in [k]$. The precise assignment of colors to edges is irrelevant for the density. In fact, for any assignment $\sigma$ of colors to the edges of $K_{1, s}$ realizing the specified multiplicities and $\bm W = (W_1, W_2, \dots, W_k) \in \mathcal W_k$ the homomorphism density can be expressed as: 
\begin{align}\label{eq:t_star_general_M}
t(\bm K_{1, \bm s}, \bm W)
&:=\int_{[0,1]^{|\bm s|+1}} \prod_{\ell=2}^{|\bm s|+1} W_{\sigma((1, \ell))}(x_1,x_\ell)\,\mathrm{d} x_1\cdots \mathrm{d} x_{|\bm s|+1} =\int_0^1 \prod_{a=1}^k d_{W_a}(x)^{s_a}\,\mathrm{d} x, 
\end{align} 
where $d_{W_a}(x):=\int_0^1 W_a(x,y)\,\mathrm{d}y$ is the degree function of the graphon $W_a$, for $a \in [k]$. We also denote $\bm d_{\bm W}(x) := (d_{W_1}(x),\dots, d_{W_k}(x))$. Note that $\sum_{a=1}^k  d_{W_a}(x)=1$ for almost every $x$, which means that $\bm d_{\bm W}(x) \in\Delta_k$. Now, consider a $k$-color ERGM model as in  \eqref{eq:TW} with sufficient statistic of the form: 
\begin{equation}
\label{eq:T_general_star}
T(\bm W):=\sum_{\bm s \in\mathscr{S} } \beta_{\bm s} \,t( \bm K_{1, \bm s}, \bm W).
\end{equation}
where $\mathscr{S}$ is a finite collection of vectors with non-negative entries and $(\beta_{\bm s})_{\bm s \in \mathscr{S}}$ be a collection of real parameters. Note, from \eqref{eq:t_star_general_M}, that $T$ depends on $\bm W$ only through $\bm d_{\bm W}$:
\begin{equation}
\label{eq:P_general_star}
T(\bm W)
=\int_0^1 P\left(\bm d_{\bm W} (x) \right)\, \mathrm{d}x, 
\qquad \text{ where } 
P(\bm \mu ):=\sum_{\bm s \in\mathscr S} \beta_{\bm s}\prod_{a=1}^k \mu_a^{s_a},
\end{equation}
for $\bm \mu \in\Delta_k$.  Consequently, in this case, the model \eqref{eq:TW} depends only on the vector of colored degrees. The following result shows that this model is in the replica symmetric regime for all parameter values.

\begin{theorem}\label{thm:colordegree}
 Consider the $k$-colored ERGM model \eqref{eq:TW} with $T$ as in \eqref{eq:T_general_star}. Then, for any collection of real parameters $(\beta_{\bm s})_{\bm s \in \mathscr{S}}$, 
 \begin{equation}
\label{eq:general_star_variational_reduction}
\lim_{n \rightarrow \infty }\psi_n 
=
\sup_{\bm \mu \in\Delta_k}\left\{P(\bm \mu )-\frac12I_k(\bm \mu )\right\},
\end{equation} 
where $I_k(\bm \mu ):=\sum_{a=1}^k \mu_a\log \mu_a$.  Moreover, every maximizer of the variational problem in Theorem \ref{thm:variational_formula} is a constant $k$-color probability graphon of the form $\bm W = \bm \mu ^*$, where $\bm \mu ^* \in \Delta_k$ is a maximizer of  \eqref{eq:general_star_variational_reduction}. 
\end{theorem}

The proof of Theorem~\ref{thm:colordegree} is given in Section~\ref{sec:starpf}, which 
follows from a straightforward application of multivariate Jensen's inequality.

\begin{remark} 
To obtain the result of \cite{ChatterjeeDiaconisExponential2013} for uncolored
star models from Theorem~\ref{thm:colordegree}, take $k=2$ and suppose 
$\bm W=(W,1-W)\in\mathcal W_2$, where $W$ is a graphon. For each integer
$s\geq 1$, color all edges of the $s$-star $K_{1,s}$ with color $1$
(which corresponds to choosing the vector $\bm s=(s,0)$ in the notation above). Then
\eqref{eq:T_general_star} reduces to the usual homomorphism density
$t(K_{1,s},W)$ of the uncolored $s$-star in the graphon $W$. Consequently,
\eqref{eq:P_general_star} becomes
$$
T(W)=\sum_{s=1}^r \beta_s t(K_{1,s},W) , 
$$ 
for $\beta_1,\ldots,\beta_r\in\mathbb R$. Theorem~\ref{thm:colordegree} then implies that the corresponding uncolored
ERGM with sufficient statistic $T(\bm W)$ is in the replica-symmetric regime for
all $\beta_1,\ldots,\beta_r\in\mathbb R$, which recovers the result in 
\cite[Theorem~6.4]{ChatterjeeDiaconisExponential2013}.
\end{remark}

\subsubsection{Monotone Interactions}
\label{sec:monochromaticTW}

For classical uncolored ERGMs, \cite{ChatterjeeDiaconisExponential2013} shows that models with sufficient statistics of the form \eqref{eq:sufficientTW} exhibit replica symmetry whenever the parameters $\beta_1,\beta_2,\ldots,\beta_r$ are non-negative. A result of this type, however, does not hold for general colored models. For example, in Section~\ref{sec:rt}, we show that for the rainbow triangle ERGM (recall  \eqref{eq:rbn} and \eqref{eq:rainbowT}), there exists a critical value $\beta_*$ such that, for every $\beta_4>\beta_*$, the model exhibits symmetry breaking, that is, constant $k$-colored probability graphons no longer optimize the variational problem \eqref{eq:varformula}.  In this section, we isolate a class of interactions for which such competition is absent, and for which the variational problem admits replica-symmetric constant maximizers.  
The key structural assumption is that the sufficient statistics can be expressed
as non-negative linear combinations of uncolored homomorphism densities. To begin with, given 
$\bm\alpha=(\alpha_1,\ldots,\alpha_k)\in \mathbb R_{ \geq 0 }^k$ and
$\bm W=(W_1,\ldots,W_k)\in\mathcal W_k$, define the projection of $\bm W$ in the direction $\bm\alpha$ as the function  $\langle \bm\alpha,\bm W\rangle:[0,1]^2\to\mathbb R_{\geq 0}$, where  
$$
\langle \bm\alpha,\bm W\rangle(x,y)
=
\sum_{a=1}^k \alpha_a W_a(x,y) , 
$$ 
for $(x, y) \in [0, 1]^2$. It is now convenient to define the following preorder: For
$\bm\alpha\in\mathbb R_{\geq0}^k$ and
$\bm W,\bm W'\in\mathcal W_k$, write
$$
\bm W'\succeq_{\bm\alpha}\bm W
\quad\Longleftrightarrow\quad
\langle \bm\alpha,\bm W'\rangle(x,y)
\ge
\langle \bm\alpha,\bm W\rangle(x,y) , 
\quad\text{for almost every }(x,y)\in[0,1]^2.
$$
For an (uncolored) finite simple graph $H$, extending the definition of the homomorphism density in \eqref{eq:tHW} verbatim to general bounded functions (not necessarily positive), one has 
\begin{align}\label{eq:tHWprojection}
t(H, \langle \bm \alpha, \bm W \rangle ) & =\int_{[0,1]^{V(H)}}\prod_{ (u, v) \in E(H)} \langle \bm \alpha, \bm W \rangle (x_u, x_v) \, \prod_{v\in V(H)} \mathrm{d} x_v \nonumber  \\ 
& =\sum_{\sigma:E(H)\to[k]} \left( \prod_{e\in E(H)}\alpha_{\sigma(e)} \right)\, t((H, \sigma), \bm W),
\end{align}
where $(H, \sigma)$ is the graph $H$ with the edges $k$-colored by $\sigma: E(H) \rightarrow [k]$. Note that the second equality in \eqref{eq:tHWprojection} follows by  expanding $\langle \bm \alpha, \bm W \rangle=\sum_{a=1}^k \alpha_a W_a$ for each edge. From the above representation it follows that $W'\succeq_{\bm \alpha} W$ implies $t(H, \langle \bm \alpha, \bm W' \rangle)\ge t(H, \langle \bm \alpha, \bm W \rangle)$ for every finite graph $H$, that is, the homomorphism densities are monotone under the pre-order $\succeq_{\bm \alpha}$. This motivates the following definition:

\begin{definition} 
\label{def:MonotoneFamHomDens} 
A function $T$ of the form \eqref{eq:sufficientTW} is said to be a \emph{monotone interaction}, if there exist uncolored graphs $H_1, \dots, H_R$ (each with at least one edge), non-negative vectors $\bm \alpha_1, \bm \alpha_2, \ldots, \bm \alpha_R \in \mathbb R_{\geq 0 }^k$, coefficients $\gamma_1, \gamma_2, \ldots, \gamma_R \in \R_{\geq 0}$, and another vector $ \bm c \in \mathbb R^k$, such that
\begin{equation}\label{eq:T_gamma_projection}
T(\bm W) = t(K_2, \langle \bm c , \bm W \rangle ) + \sum_{s=1}^R \gamma_s  t(H_s, \langle \bm \alpha_s , \bm W \rangle ) . 
\end{equation}
\end{definition}

Recall that, in the classical uncolored ERGM setting, one often considers
sufficient statistics of the form
\begin{align}\label{eq:TWgraph}
T( W )= \beta_0 t(K_2,W) + \sum_{s=1}^r \beta_s\, t(H_s,W), 
\end{align}
where $H_1, \ldots,H_r$ are finite simple graphs and
$\beta_1,\beta_2,\ldots,\beta_r\in\mathbb R$. Note that the homomorphism density $t(H, W)$, for any graph $H$, is trivially monotone with respect to the usual pointwise partial order on graphons: that is, if $W'(x,y)\geq W(x,y)$ for almost every $(x,y)\in[0,1]^2$,  then $t(H, W')\geq t(H, W)$. In the colored setting, there is no canonical single pointwise order, since the coordinates of a $k$-colored probability graphon are constrained to lie in the simplex, and hence, increasing the density of one color necessarily decreases the density of at least one other color. The purpose of
Definition~\ref{def:MonotoneFamHomDens} is to identify directions
$\bm\alpha_1,\bm\alpha_2,\ldots,\bm\alpha_R$ in the non-negative orthant of
$\mathbb R^k$ along which monotonicity can still be formulated. 
We will discuss several examples after the next
theorem, which shows that an ERGM with a monotone interaction is in the
replica-symmetric regime.

\begin{theorem}\label{thm:explicit_solvable_monotone}
 Suppose $T$ is a monotone interaction as in~\eqref{eq:T_gamma_projection}. For the $k$-colored ERGM model \eqref{eq:TW} with sufficient statistic $T$, the following holds: 
\begin{equation}\label{eq:explicit_solvable_formula}
\lim_{n\to\infty}\psi_n
=
\sup_{\bm \mu \in\Delta_k}
\left( \langle \bm c , \bm \mu \rangle + \sum_{s=1}^R \gamma_s \langle \bm \alpha_s, \bm \mu \rangle^{|E(H_s)|} - \frac12\,I_k( \bm \mu) 
\right).
\end{equation} 
Moreover, every maximizer of the variational problem in Theorem \ref{thm:variational_formula} is a constant $k$-color probability graphon of the form $\bm W = \bm \mu ^*$, where $\bm \mu ^* \in \Delta_k$ is a maximizer of  \eqref{eq:explicit_solvable_formula}.
\end{theorem}

The proof of Theorem~\ref{thm:explicit_solvable_monotone} is given in
Section~\ref{sec:monotonepf}. The key ingredient is a generalized H\"older's inequality
(see Lemma~\ref{lem:holder_equality}), which allows us to bound the relevant homomorphism
densities in terms of suitable norms of the underlying graphons. The monotonicity
of the projection directions then allows these bounds to
be combined into an upper bound for the variational problem in
\eqref{eq:varformula} of the form given in
\eqref{eq:repsym}.

\begin{remark} 
Note that representation in \eqref{eq:T_gamma_projection} may not be unique. For instance, to ensure the projection acts as a standard probability graphon bounded in $[0,1]$, one could define $$\tilde{\bm \alpha}_s = \frac{\bm \alpha_s}{ \max_{a \in [k]}  \alpha_{s ,a} } , $$ and extract the scaling factor $(\max_{a \in [k]}  \alpha_{s ,a})^{|E(H_s)|}$ to multiply the coefficient outside the homomorphism density. Further, as $\gamma_s \ge 0$ for $s \in [R]$, one could absorb $\gamma_s$ into the projection vector $\bm \alpha_s$ by taking $$\tilde{\bm \alpha}_s = \gamma_s^{\frac{1}{|E(H_s)|}} \bm \alpha_s.$$
For this reason, we will often omit these coefficients from the representation
\eqref{eq:T_gamma_projection}. In such cases, they are understood to have been
absorbed into the projection vectors
$\bm\alpha_1,\bm\alpha_2,\ldots,\bm\alpha_R$ as described above.
\end{remark}

We now present several examples of ERGMs with monotone interactions. We begin
with monochromatic patterns, which leads to classical uncolored ERGMs with interactions of the form \eqref{eq:TWgraph}.

\begin{example}[Monochromatic patterns]
Fix $k\geq 2$, and suppose that $H_1,H_2,\ldots,H_r$ is a collection of $k$-colored graphs, each of which is monochromatic. For $s\in[r]$, let $\kappa(s)$ denote the unique color assigned to the edges of $H_s$. Then any function of the form 
\begin{align}\label{eq:color1}
T(\bm W) = t\left(K_2,\langle \bm c,\bm W\rangle\right) + \sum_{s=1}^r \beta_s t(H_s,\bm W), 
\end{align} 
where $\beta_1,\beta_2,\ldots,\beta_r\in\mathbb R_{\geq 0}$, $\bm W\in\mathcal W_k$, and $\bm c = (c_1, c_2, \ldots, c_k)\in\mathbb R^k$, can be expressed as a monotone interaction. Indeed, choosing $\bm\alpha_s=\bm e_{\kappa(s)}$, where $\bm e_a$ denotes the $a$-th standard basis vector of $\mathbb R^k$, gives $$ t(H_s,\bm W) = t\left(H_s,W_{\kappa(s)}\right) = t\left(H_s,\langle \bm\alpha_s,\bm W\rangle\right), $$ where $H_s$ is interpreted as a simple graph in the second and third terms of the equality, which yields an expression of the form \eqref{eq:T_gamma_projection}. In particular, if $k=2$ and $H_1,H_2,\ldots,H_r$ are all monochromatic graphs of color $1$, then, under the identification $\bm W=(W,1-W)\in\mathcal W_2$, where $W$ is an ordinary graphon, \eqref{eq:color1} reduces to (ignoring a constant term):  
$$ T(\bm W)= \beta_0 t(K_2, W) + \sum_{s=1}^r \beta_s t(H_s,W), $$ 
for some $\beta_0 \in \R$. 
Theorem~\ref{thm:explicit_solvable_monotone} shows that such an ERGM lies in the
replica symmetric regime whenever $\beta_1,\beta_2,\ldots,\beta_r$ are non-negative, which is precisely the result of \cite[Theorem~4.1]{ChatterjeeDiaconisExponential2013}. 
\end{example}

\begin{example}[Monotone colored triangles]  
Let $K_3$ be an uncolored triangle with vertices labeled $\{1, 2, 3\}$. A single projection of the homomorphism density of $K_3$ in the direction $\bm \alpha =  (\alpha_1, \alpha_2, \ldots, \alpha_k) \in \mathbb{R}_{\geq 0}^k$ generates a weighted combination of colored triangles as follows: For $\bm W = (W_1, W_2, \ldots, W_k) \in \mathcal W_k$, 
\begin{align}\label{eq:alphamonotone}
t(K_3, \langle \bm \alpha , \bm W \rangle) = \sum_{a,b,c \in [k]} \alpha_a \alpha_b \alpha_c  \, t( \bm K_3^{(a,b,c)}, \bm W) , 
\end{align}
where $\bm K_3^{(a,b,c)}$ is a triangle where edges $(1, 2), (2, 3), (3, 1)$ are assigned colors $a, b, c$, respectively. Hence, any $k$-colored ERGM with sufficient statistic of the form \eqref{eq:alphamonotone} will be in the replica symmetric phase. Conversely, one may ask for which choices of parameters
$(\beta_{abc})_{a,b,c\in[k]}$, that we will assume without loss of generality to be symmetric\footnote{If the tensor $(\beta_{abc})_{a,b,c\in[k]}$ is not symmetric, then  take $\frac{1}{3!} \sum_{\pi \in S_3} \beta_{a_{\pi(1)}a_{\pi(2)}a_{\pi(3)}}$, where $(a_1, a_2, a_3)  = (a, b, c)$ and $S_3$ denotes the set of permutations of $[3]$, and observe that $t (\bm K_3^{(a,b,c)},\bm W )$ is invariant under the permutation of $a, b, c$.}, a $k$-color ERGM sufficient statistic of the form
\begin{align}\label{eq:TWcolortriangle}
T(\bm W) =  \sum_{a,b,c\in [k]}
\beta_{abc}\, t (\bm K_3^{(a,b,c)},\bm W ) 
\end{align}
can be represented as a monotone interaction? 

\begin{itemize}

\item If we consider a projection along a single direction $\bm \alpha \in \R_{\geq0}^k$ as in \eqref{eq:alphamonotone}, then a representation of the form \eqref{eq:TWcolortriangle} is  possible, for example, when the coefficient symmetric tensor
$\bm B=(\beta_{abc})_{a,b,c\in[k]}$ admits a rank-one representation
\begin{align}\label{eq:Babc}
\bm B = \bm\alpha\otimes\bm\alpha\otimes\bm\alpha := (\alpha_a\alpha_b\alpha_c)_{a,b,c\in[k]},
\end{align}  
for some non-negative vector
$\bm\alpha=(\alpha_1,\alpha_2,\ldots,\alpha_k)\in\mathbb R_{\geq0}^k$. In Lemma \ref{lem:monotoneabc} we show that \eqref{eq:Babc} holds if and only if
\begin{align}\label{eq:positiveinteraction}
\beta_{aaa}\geq0  \quad \text{ and } \quad  \beta_{abc} = (\beta_{aaa}\beta_{bbb}\beta_{ccc})^{\frac{1}{3}}, 
\end{align}
for $a,b,c\in[k].$ In this case, the vector $\bm\alpha$ is uniquely determined by 
$\alpha_a=\beta_{aaa}^{\frac{1}{3}}$, for $a\in[k]$.  
Note that the set 
\begin{align*}
\mathcal R_\triangle^+
:=
\left\{
(\beta_{abc})_{a,b,c\in[k]}:
\beta_{abc}=\alpha_a\alpha_b\alpha_c
\text{ for some } \bm\alpha\in\mathbb R_{\geq0}^k
\right\} , 
\end{align*}
is closed under multiplication by non-negative scalars, hence it is a cone inside the space of symmetric order three tensors. Algebraically, $\mathcal R_\triangle^+$ is the non-negative part of the affine cone over the degree-three Veronese variety \cite{landsberg2012tensors}.

\item One can substantially enlarge the class of coefficient (symmetric) tensors $\bm B=(\beta_{abc})_{a,b,c\in[k]}$ for which replica symmetry holds by
considering projections along multiple non-negative directions
$\bm\alpha_1,\bm\alpha_2,\ldots,\bm\alpha_R\in\mathbb R_{\geq0}^k$, for some
$R\geq 1$. In this case, Theorem~\ref{thm:explicit_solvable_monotone} implies replica symmetry whenever
$$
T(\bm W)
=
\sum_{s=1}^R
t\left(K_3,\langle \bm\alpha_s,\bm W\rangle\right).
$$
This, in particular, applies when $\bm{B}$ belongs to the set 
$$
\mathcal R_{\Delta, R}^+ := \left\{ \bm B = (\beta_{abc})_{a,b,c\in[k]} :  \bm B=\sum_{s=1}^R \bm\alpha_s \otimes \bm\alpha_s \otimes \bm\alpha_s \text{ and } \bm\alpha_s\in\mathbb R_{\geq0}^k \right\} , $$
A symmetric non-negative tensor admitting such a decomposition is called a
\emph{completely positive tensor} (of order-three, in our setting), and the minimum possible value of $R$ is called its \emph{CP-rank} \cite{qi2014nonnegative,luo2016completely}. Entrywise non-negativity, that is, $\beta_{abc}\geq 0$, for all
$a,b,c\in[k]$, is clearly necessary for complete positivity, but it is not
sufficient in general.  In fact, deciding whether an arbitrary tensor is completely positive is generally NP-hard \cite{luo2016completely,dickinson2014complexity}. Consequently, the
completely positive cone provides a rich and structured class of coefficient
tensors for which the colored interaction admits a monotone projection
representation. 
\end{itemize}  
\label{example:triangleinteractions}  
\end{example}

\begin{remark}[Extension to general colored graphs]  
The discussion in Example \ref{example:triangleinteractions}  can be extended from triangles to an arbitrary finite graph $H$. The additional point is that different edge-colorings
$\sigma$ of $H$ may define the same decorated homomorphism density. Thus, one
should first clump together the coefficients corresponding to colorings that
are equivalent under the symmetries of $H$. The resulting coefficient data is
naturally indexed by the corresponding equivalence classes of colorings, and
need not form a symmetric tensor, since a general graph $H$ does not realize
all permutations of its edges through vertex automorphisms. For a single projection $t(H,\langle \bm \alpha,\bm W\rangle)$, for $\bm \alpha\in\mathbb R_{\geq0}^k$, 
the induced coefficients are nevertheless generated by a symmetric rank-one
tensor, and a criterion analogous to \eqref{eq:positiveinteraction} can be obtained by asking
whether the clumped coefficient data admits such a rank-one representation.
For several projection directions, the corresponding characterization is more
involved: one must determine whether the clumped coefficients admit a lift that is a finite sum of symmetric rank-one tensors generated by nonnegative vectors, that is, a completely positive symmetric tensor compatible with the symmetries of  $H$.
\end{remark}

\begin{remark}[Non-monotonicity of induced homomorphism densities]While homomorphism densities for ordinary graphons are monotone with
respect to the pointwise partial order, this generally fails for
induced homomorphism densities (recall \eqref{eq:inducedW}). For example, suppose $k=2$, 
$\bm W=(W,1-W)\in\mathcal W_2$, where $W: [0, 1]^2 \rightarrow [0, 1]$ is an ordinary graphon, and 
$\bm K_{3}^{(1, 1, 2)}$ is a triangle with two edges colored 1 and one edge colored 2. Then,
recalling \eqref{Def:homDensInd} and \eqref{eq:inducedW}, $t(\bm K_{3}^{(1, 1, 2)},\bm W) := t_{\mathrm{ind}}(K_{1,2},W)$,
where
\begin{align}\label{eq:wedgeW}
t_{\mathrm{ind}}(K_{1,2},W) = \int_{[0,1]^3} W(x,y)W(x,z) (1-W(y,z) )\, \mathrm dx\,\mathrm dy\,\mathrm dz = t(K_{1,2},W)-t(K_3,W), 
\end{align}
is the induced homomorphism density of the 2-star $K_{1, 2}$ in $W$. 
This functional is clearly not monotone in $W$ with respect to the pointwise partial
order. In particular, consider a classical uncolored ERGM with sufficient statistic
$$
T(W)=\beta_0 t(K_2,W)+\beta_1 t_{\mathrm{ind}}(K_{1,2},W).
$$
We show in Section~\ref{sec:temperaturewedge} that there exists a critical value $\beta_*$ such
that, for $\beta_1>\beta_*$, this model exhibits symmetry breaking. However, as the previous example illustrates, for multiple colors (more than 2) it is possible to weight the various colored triangles appropriately to obtain monotonicity.  
\end{remark}

\subsubsection{Negative Sidorenko interactions} 
\label{sec:negativereplicasymmetry}

A graph $H = (V(H), E(H))$ is said to have the \emph{Sidorenko property}, or to be a \emph{Sidorenko graph}, if  $t(H,W)\geq t(K_2,W)^{|E(H)|}$, for every graphon $W$. Equivalently, among all graphs with a prescribed edge density, the Erd\H{o}s--R\'enyi random graph asymptotically minimizes the homomorphism density of $H$. Sidorenko's conjecture asserts that every bipartite graph has this property \cite{sidorenko1993correlation} (see also \cite[Chapter~16]{LovaszGraphLimits}). The conjecture is known for several important classes of bipartite graphs, including trees, even cycles, complete bipartite graphs. Hatami introduced the framework of graph norms and proved Sidorenko-type inequalities for weakly norming graphs, including hypercubes \cite{hatami2010graph}. Conlon, Fox, and Sudakov \cite{conlon2010approximate} proved the conjecture for bipartite graphs having a vertex complete to the opposite part, and also established an approximate version for arbitrary bipartite graphs. For recent overviews of known cases and methods related to Sidorenko's conjecture, see the introductions of \cite{im2026sidorenko,conlon2024extremal} and the references therein.

For colored ERGMs, using the projection idea from Section \ref{sec:monochromaticTW}, one can identify linear combinations of Sidorenko graphs for which replica symmetry holds. To this end, we need the following definition: 
\begin{definition}
A function $T$ of the form \eqref{eq:sufficientTW} is said to be a \emph{negative Sidorenko interaction}, if there exist Sidorenko graphs $F_1, \dots, F_R$ (each with at least one edge), non-negative vectors  $\bm \eta_1, \bm \eta_2, \ldots, \bm \eta_R \in \mathbb R_{ \geq 0 }^k$, coefficients $\delta_1, \delta_2, \ldots, \delta_R \geq 0$, and another vector $ \bm{c}_1\in \mathbb R^k$, such that 
\begin{equation}\label{eq:T_delta_projection}
T(\bm W) = t(K_2, \langle \bm{c}_1 , \bm W \rangle ) - \sum_{s=1}^R \delta_s  t(F_s, \langle \bm \eta_s , \bm W \rangle ) .  
\end{equation}
\end{definition}

The next result shows that colored ERGMs with negative Sidorenko interactions lie in the replica-symmetric regime. The proof is given in Section~\ref{sec:negativesymmetrypf}.

\begin{theorem}\label{thm:NegativeSidorenko_RS}
Suppose $T$ is a negative Sidorenko interaction as in~\eqref{eq:T_delta_projection}. Then for the $k$-colored ERGM model \eqref{eq:TW} with sufficient statistic $T$, the following holds: 
\begin{equation}\label{eq:NegativeSidorenko_finite}
\lim_{n\to\infty}\psi_n
=
\sup_{\bm \mu \in\Delta_k}
\left( \langle  \bm{c}_1, \bm \mu \rangle  - \sum_{s=1}^R \delta_s \langle {\bm \eta_s}, \bm \mu \rangle^{|E(F_s)|} - \frac12\,I_k( \bm \mu) 
\right).
\end{equation} 
Moreover, every maximizer of the variational problem in Theorem \ref{thm:variational_formula} is a constant $k$-color probability graphon of the form $\bm W = \bm \mu ^*$, where $\bm \mu ^* \in \Delta_k$ is a maximizer of  \eqref{eq:NegativeSidorenko_finite}. 
\end{theorem}

\subsubsection{ Rainbow Common }
\label{sec:rc}

In extremal combinatorics, an uncolored graph $H$ is said to be {\it common} if, among all 2-colorings of $K_n$, the asymptotic minimum number of monochromatic copies of $H$ is attained by the uniform random coloring. In graphon language, this translates the following \cite[Chapter 16.5.4]{LovaszGraphLimits}: $H$ is common if 
$$t(H,W)+t(H,1 - W) \geq 2^{1- |E(H)|},$$ 
for every graphon $W$. Commonness has been extensively studied because of its
connections to other homomorphism-density inequalities, specifically, the Sidorenko conjecture. In particular, Sidorenko graphs (recall discussion in Section \ref{sec:negativereplicasymmetry}) are common \cite[Chapter 16]{LovaszGraphLimits}.  In the colored setting, specifically in anti-Ramsey theory, one often studies the opposite extremal problem: instead of minimizing the number of monochromatic
copies, the goal is to maximize the number of rainbow copies of a graph, that
is, copies whose edges all receive distinct colors
\cite{Silva2018AntiRamsey,baloghRainbowTrianglesThreecolored2017,erdos1972ramsey}.
This leads to the following natural analogue of commonness for rainbow graphs.

\begin{definition}[Rainbow common]
\label{defn:rc}
An uncolored graph $H$ (with at least one edge) is said to be $k$-\emph{rainbow common} (for $k\geq |E(H)|$) if \begin{equation}\label{eq:rainbow_stat}
R_{H,k}(\bm{W}) := \sum_{\substack{\sigma : E(H) \to [k] \\ \text{injective}}} t((H,\sigma), \bm{W}),
\end{equation} where the sum on the right-hand side of \eqref{eq:rainbow_stat} is over the injective maps $\sigma:E(H)\rightarrow [k]$ (i.e.\ the colorings $\sigma\in \mathcal C_k(H)$ that are rainbow) is maximized, over all $k$-colored probability graphons, at the uniform distribution on the simplex. In other words,
\begin{align*}
 \left( \frac{1}{k}, \frac{1}{k}, \ldots, \frac{1}{k} \right)\in \argmax_{\bm W}R_{H,k}(\bm{W}) 
\end{align*}
and, therefore, 
\begin{align}\label{eq:raimbCommMax}
\sup_{\bm W}R_{H,k}(\bm{W}) =\frac{k! }{(k-|E(H)|)!k^{|E(H)|}}  . 
\end{align}
Otherwise, $H$ is said to be $k$-{\it rainbow uncommon}. We say that $H$ is {\it rainbow common} (respectively, {\it rainbow uncommon}) if it is $k$-rainbow common (respectively, $k$-rainbow uncommon), for every $k \geq |E(H)|$.\footnote{Note that $H$ can be $r$-rainbow common for some $r$ and $s$-rainbow uncommon for some $s \ne r$,
in which case it is neither rainbow common nor rainbow uncommon. However, no such example of $H$ is known. } 
\end{definition}

A classical result of \citet{erdos1972ramsey} shows that
triangles are 3-rainbow uncommon, that is, there exists a $3$-coloring of $K_n$
with asymptotically larger rainbow triangle density than that obtained from the
uniform random $3$-coloring. The problem of determining the maximum possible
rainbow triangle density was resolved later by \cite{baloghRainbowTrianglesThreecolored2017}, using techniques from flag algebra. We will return to this result in Section~\ref{sec:rt}. On the positive side, \cite{Silva2018AntiRamsey} showed that disjoint unions of stars are rainbow common and conjectured that paths on $4$ vertices are rainbow common. More
recently, \cite{sun2024rainbowcommongraphsforests} showed that if $ H$
contains a cycle, then it is rainbow uncommon. In particular, any rainbow common
graph must be a forest. A recent paper of \citet{Batetal24} conjectures that
\emph{all} forests are rainbow common, which, if true, together with the result of
\cite{sun2024rainbowcommongraphsforests}, would give a complete
characterization of rainbow common graphs.

For colored ERGMs, rainbow commonness provides a simple sufficient condition for replica symmetry. Specifically, suppose we have a $k$-colored ERGM with sufficient statistic of the form \eqref{eq:sufficientTW} that can be rewritten as \begin{equation}\label{Eq:statRainbowCommonSum}
    T(\bm{W})=\sum^r_{s=1}\beta_sR_{H_s,k}(\bm{W}),
\end{equation}where  $H_1,H_2,\ldots,H_r$ are all $k$-rainbow common and $R_{H_s,k}$ is defined as in \eqref{eq:rainbow_stat} and $\beta_s$ are scalars. Then, for $\beta_1,\beta_2,\ldots,\beta_r\geq 0$, 
\begin{align}\label{eq:colorHW} T(\bm W)-\frac12\mathcal I(\bm W) \leq \sum_{s=1}^r \beta_s\frac{k!}{(k-|E( H_s)|)!k^{|E( H_s)|}} + \frac12\log k .  
\end{align} 
This follows from the rainbow commonness bound 
$R_{H,k}(\bm{W})\leq \frac{k!}{(k-|E( H)|)!k^{|E( H)|}}$ (recall \eqref{eq:raimbCommMax}), together with the entropy bound $\mathcal I(\bm W)\geq \log\left(\frac{1}{k}\right)$. Moreover, equality in \eqref{eq:colorHW} is attained if and only if $\bm W$ is the uniform $k$-color probability graphon $\bm{U}^{(k)}$. Consequently, colored ERGMs whose sufficient statistics are non-negative linear combinations of rainbow-common graphs lie in the replica-symmetric regime,  with the uniform color distribution as a maximizer of the variational problem. This is summarized in the following result:

\begin{proposition}
Consider the $k$-colored ERGM model \eqref{eq:TW} with sufficient statistic of the form \eqref{Eq:statRainbowCommonSum}, where $ H_1, H_2,\ldots, H_r$ are $k$-rainbow common and $\beta_1,\beta_2,\ldots,\beta_r\geq 0$. Then 
$$\lim_{n \rightarrow \infty} \psi_n =\sum_{s=1}^r \beta_s\frac{k!}{(k-|E( H_s)|)!k^{|E( H_s)|}} + \frac12\log k . $$ 
Moreover, the uniform $k$-color probability graphon $\bm{U}^{(k)} = ( \frac{1}{k}, \frac{1}{k}, \ldots, \frac{1}{k} ) $ is the unique maximizer of the variational problem in Theorem \ref{thm:variational_formula}.   
\end{proposition}

\subsection{Euler--Lagrange Equations and High-Temperature Uniqueness }
\label{sec:euler_lagrange}

In the previous sections, we derived several conditions (involving both the
colored graphs appearing in the sufficient statistics and the choices of
parameters) under which colored ERGMs exhibit replica symmetry. For general
sufficient statistics and general parameter values, the variational problem is
notoriously difficult. Nevertheless, one can still obtain information about the
structure of the optimizers by deriving the Euler--Lagrange equations. To this
end, consider the $k$-colored ERGM \eqref{eq:TW} with $T$ as in
\eqref{eq:sufficientTW}, where
$\beta_1,\ldots,\beta_r\in\mathbb R$ and
$\bm H_1,\ldots,\bm H_r$ are finite simple $k$-colored graphs. 
Recall that for each $s \in [r]$, $\bm H_s =  (H_s, \sigma_s)$, where $\sigma_s$ is a $k$-coloring of the $H_s$. For each $s \in [r]$, $a\in[k]$, and $\bm W = (W_1, W_2, \ldots, W_k) \in \mathcal W_k$, define first the oriented edge-rooted kernel as:   
\begin{equation}\label{eq:Delta_H}
\Delta_{\bm H_s}^{(a)}(\bm W)(x,y)
=
\sum_{(u, v)\in E(H_s)}
\bm 1\{ \sigma_s((u, v))=a \}
\int_{[0,1]^{|V(H_s)|-2}}
\prod_{\substack{(u', v')\in E(H_s)\\ (u', v')\neq(u, v)}}
W_{\sigma_s((u', v'))}(x_{u'},x_{v'})
\, \prod_{\substack{ t \in V(H_s) \\ t\notin \{u, v\} }} \mathrm{d} x_t ,
\end{equation} 
where $x_u=x$ and $x_v=y$. Then define the symmetrized edge-rooted kernel as:  
\begin{equation}\label{eq:Delta_H_sym}
\widehat\Delta_{\bm H_s}^{(a)}(\bm W)(x,y)
:=
\frac12\left[
\Delta_{\bm H_s}^{(a)}(\bm W)(x,y)
+
\Delta_{\bm H_s}^{(a)}(\bm W)(y,x)
\right].
\end{equation}
For every collection of bounded symmetric functions
$g_1,\ldots,g_k:[0,1]^2\to\mathbb R$, one has
\begin{align}\label{eq:first_variation_symmetric}
\lim_{\varepsilon \to 0} \frac{t(\bm H_s, \bm W+ \varepsilon  \bm g) - t(\bm H_s, \bm W)}{\varepsilon } = \sum_{a=1}^k \left\langle g_a, \widehat\Delta_{\bm H_s}^{(a)}(\bm W)
\right\rangle_{L^2} = \sum_{a=1}^k \left\langle g_a, \Delta_{\bm H_s}^{(a)}(\bm W)
\right\rangle_{L^2}  ,  
\end{align}
where the second equality follows because each $g_a$ is symmetric. 
Also, denote 
\begin{align}\label{eq:F}
\mathcal F(\bm W) = T(\bm W) - \frac{1}{2} \mathcal I( \bm W),
\end{align}  
where $T(\bm W)$ and $\mathcal I(\bm W)$ are as defined in \eqref{eq:sufficientTW} and \eqref{eq:IW}, respectively. This  descends naturally to the quotient space $\wt{\mathcal W}_k$: For $\wt{\bm W}\in\wt{\mathcal W}_k$, we define $\mathcal F(\wt{\bm W}) = T(\wt{\bm W}) - \frac{1}{2} \mathcal I(\wt{\bm W})$, where $\bm W$ is any representative of the equivalence class $\wt{\bm W}$.

\begin{theorem}\label{thm:EL-prob-graphon}
Let $\wt{\bm W}\in\wt{\mathcal W}_k$ be a maximizer of $\mathcal F$ and $\bm W = (W_1, W_2, \ldots, W_k) \in\mathcal W_k$ be any representative of the equivalence class $\wt{\bm W}$. Then every $a \in [k]$, 
\begin{equation}\label{eq:softmax}
W_a(x,y)
=
\frac{
\exp\left(2\sum_{s=1}^r \beta_s \,\widehat\Delta_{\bm H_s}^{(a)}(\bm W)(x,y)\right)
}{
\sum_{b=1}^k
\exp\left(2\sum_{s=1}^r \beta_s \,\widehat\Delta_{\bm H_s}^{(b)}(\bm W)(x,y)\right)
} , 
\end{equation} 
for almost all $(x,y)\in[0,1]^2$, where
$\widehat\Delta_{\bm H_s}^{(a)}$ is defined in \eqref{eq:Delta_H_sym}. 
Moreover, $\bm W$ is bounded away from the boundary of the simplex, that is, there exists $\varepsilon>0$ such that $\varepsilon\le W_a(x,y)\le 1-\varepsilon$, for  all $a\in[k]$ and almost every $(x,y)\in[0,1]^2$.  
\end{theorem}

The proof of Theorem~\ref{thm:EL-prob-graphon} is given in Section~\ref{thm:EL-prob-graphonpf}. The equations in \eqref{eq:softmax} are the first-order conditions satisfied by
any maximizer of the variational problem in Theorem~\ref{thm:variational_formula}. As an illustration, we consider the rainbow triangle ERGM, for which the equations simplify considerably.

\begin{example}\label{ExampEulerLagrangeRainbowTriang}
Consider the rainbow triangle ERGM with sufficient statistic as in \eqref{eq:rainbowT}.  For simplicity assume $\beta_1 = \beta_2 = \beta_3 = 0 $. Recall, from \eqref{eq:tWrainbow}, that, for $\bm W = (W_1, W_2, W_3) \in \mathcal W_3$, 

\begin{align*}
t_{\mathrm{rb}\triangle}(\bm W)
&=
\int_{[0,1]^3}W_1(x,y)W_2(y,z)W_3(z,x)
\,\mathrm dx\,\mathrm dy\,\mathrm dz
=
t(\bm K_3^{\mathrm{rb}},\bm W).
\end{align*}
Here $\bm K_3^{\mathrm{rb}}$ is the fixed decorated triangle introduced in \eqref{eq:tWrainbow}. In this case, the symmetrized edge-rooted kernels are: 
\begin{align*}
\widehat\Delta_{\bm K_3^{\mathrm{rb}}}^{(1)}(\bm W)(x,y)
&=
\frac12\int_0^1
\left[
W_2(y,z)W_3(x,z)+W_2(x,z)W_3(y,z)
\right] \,\mathrm dz,\\
\widehat\Delta_{\bm K_3^{\mathrm{rb}}}^{(2)}(\bm W)(x,y)
&=
\frac12\int_0^1
\left[
W_1(x,z)W_3(y,z)+W_1(y,z)W_3(x,z)
\right] \,\mathrm dz,\\
\widehat\Delta_{\bm K_3^{\mathrm{rb}}}^{(3)}(\bm W)(x,y)
&=
\frac12\int_0^1
\left[
W_1(x,z)W_2(y,z)+W_1(y,z)W_2(x,z)
\right] \,\mathrm dz.
\end{align*}
Then \eqref{eq:softmax} simplifies to the following: For $a\in[3]$, 
\begin{equation*}
W_a(x,y) = \frac{\exp\left(2\beta_4\widehat\Delta_{\bm K_3^{\mathrm{rb}}}^{(a)}(\bm W)(x,y)\right)}{\sum_{b=1}^3 \exp\left(2\beta_4\widehat\Delta_{\bm K_3^{\mathrm{rb}}}^{(b)}(\bm W)(x,y)\right) }. 
\end{equation*} 
\end{example}

The Euler--Lagrange equations in \eqref{eq:softmax} admit a useful representation as a fixed-point equation.  To see this, define 
\begin{align*}
\widehat\Psi_a(\bm W)(x,y)
:=
\sum_{s=1}^r
\beta_s\,\widehat\Delta_{\bm H_s}^{(a)}(\bm W)(x,y), 
\end{align*}
for $a\in[k]$ and let
$\widehat{\bm\Psi}=(\widehat\Psi_1,\ldots,\widehat\Psi_k)$.
Then the equations in \eqref{eq:softmax} can be expressed as: 
\begin{align}\label{eq:hightemp}
\bm W(x,y) = \SoftMax\left(2\widehat{\bm\Psi}(\bm W)(x,y)\right) , 
\end{align}  
where $\SoftMax\colon\mathbb R^k\to\Delta_k$ is defined as 
\begin{align}\label{eq:softmaxz}
\SoftMax(\bm z) := \left( \frac{e^{z_1}}{\sum_{b=1}^k e^{z_b}}, \frac{e^{z_2}}{\sum_{b=1}^k e^{z_b}}, \ldots , \frac{e^{z_k}}{\sum_{b=1}^k e^{z_b}}\right) , 
\end{align}
for $\bm z=(z_1,z_2,\ldots,z_k)$. This representation allows us to derive the following high-temperature uniqueness criterion for colored ERGMs.

\begin{theorem}\label{thm:highT_colored} 
Consider the $k$-colored ERGM \eqref{eq:TW} with $T$ as in
\eqref{eq:sufficientTW}, where
$\beta_1,\ldots,\beta_r\in\mathbb R$ and
$\bm H_1,\ldots,\bm H_r$ are finite, simple $k$-colored graphs. 
Define
\begin{equation}
\label{eq:Lambda_pair_def}
\Lambda
:=
\sum_{s=1}^r |\beta_s| \left( (|E(H_s)|-1)\,
\max_{ 1 \leq a\neq b \leq k } \left(|E_a(H_s)|+|E_b(H_s)|\right) \right) .
\end{equation}
If $\Lambda<2$,  then the variational problem in Theorem \ref{thm:variational_formula} has a unique (almost everywhere) maximizer. Consequently, the maximizer is replica-symmetric, that is, there exists $\bm \mu ^\star\in\Delta_k$ such that $\bm W^\star(x,y)\equiv \bm \mu ^\star$, almost everywhere in $[0, 1]^2$. 
\end{theorem}

The proof of Theorem~\ref{thm:highT_colored} is given in Section~\ref{sec:highT_coloredpf}. The proof is based on a contraction argument, which requires establishing a Lipschitz property of the fixed-point map in \eqref{eq:hightemp} with respect to an appropriate seminorm.
To identify the appropriate seminorm, note the following shift-invariance property of the $\SoftMax$ function: $$ \SoftMax(\bm z+c\bm 1)=\SoftMax(\bm z), \qquad \text{for all } \bm z\in\mathbb R^k \text{ and } c\in\mathbb R. $$ Thus, $\SoftMax$ depends only on the relative differences between the coordinates of its input, and not on a common additive shift. This suggests that the correct quantity for measuring deviations in the fixed-point equation is not the $L^\infty$ norm itself, but rather the distance to the diagonal direction $\mathbb R\bm 1$. This motivates the definition of the oscillation seminorm (see \eqref{eq:normz}), using which we establish the contraction property of the fixed-point map, leading to the high-temperature uniqueness criterion in \eqref{eq:Lambda_pair_def}. In particular, when $k=2$, this recovers the high-temperature threshold for classical ERGMs, as explained in the next remark.

\begin{remark}[Recovery of the classical binary threshold]
Note that when $k=2$, $|E_1(H_s)|+ |E_2(H_s)|=|E(H_s)|$, for $s \in [r]$. Hence,  \eqref{eq:Lambda_pair_def} becomes $\Lambda=\sum_{s=1}^r |\beta_s|\,(|E(H_s)|-1)\,|E(H_s)|$. Therefore, the condition $\Lambda<2$ reduces exactly to the high-temperature criterion for classical uncolored ERGMs in \cite[Theorem~6.2]{ChatterjeeDiaconisExponential2013}. 
\end{remark}

\subsection{Zero Temperature Asymptotics }
\label{sec:temperatureasymptotics}

In the previous sections, we established various conditions under which the solutions of the variational problem \eqref{eq:varformula} are replica-symmetric constant probability graphons. These results either require the sufficient statistic to admit certain specific representations, or require the total magnitude of the coefficients to be sufficiently small (a `high-temperature' condition). In this section, we investigate what happens when one of the coefficients takes a large positive or negative value. To this end, consider a $k$-colored ERGM with sufficient statistic of the form: 
\begin{equation}\label{eq:T_betaf_infty} 
    T(\bm W) = \beta f(\bm W) + g(\bm W) , 
\end{equation}
where $f, g: \mathcal{W}_k \rightarrow \mathbb{R}$ are continuous functions that are invariant under measure preserving transformations. Our interest is in the asymptotic behavior of the model as $\beta \rightarrow \infty$. (The other parameters, if any, are absorbed in the function $g$.) This entails, by Theorem \ref{thm:variational_formula}, understanding the structure of the  optimizers  of the following variational problem, as $\beta \rightarrow \infty$: 
\begin{align}\label{eq:limitvariational}
\Gamma_\beta := \argmax_{\wt{\bm W}\in\wt{\mathcal W}_k}
\left\{\beta f(\wt{\bm W})+g( \wt{\bm W} )-\frac12\mathcal I( \wt{\bm W} )\right\}. 
\end{align}
For this, we propose the following two-step strategy: 
\begin{enumerate}
\item First, maximize the `main term' $f$. Specifically, let  
\begin{align}\label{eq:largetemperaturemaximizerf} 
\mathcal{N}_f = \argmax_{\wt {\bm W} \in \wt{\mathcal{W}}_k} f(\wt{\bm W})  .  
\end{align} 
Note that, by the compactness of $\wt{\mathcal{W}}_k$ and the continuity of $f$, the set $\mathcal{N}_f$ is compact. 

\item Then among the maximizers of $f$, maximize the `tie-breaker' term $h(\bm W):=g(\bm W)-\frac12\mathcal I(\bm W)$. Formally, define 
\begin{align}\label{eq:largetemperaturemaximizer}  
\mathcal N^*=\argmax_{\wt{\bm W}\in \mathcal{N}_f} \left(g(\tilde{\bm W})-\frac12\mathcal I(\tilde{\bm W})\right). 
\end{align}

\end{enumerate}
The next result shows that the optimizers of \eqref{eq:limitvariational} concentrate around $\mathcal{N}^*$, as $\beta \to \infty$. To emphasize our interest in the asymptotics of $\beta$, we will denote $\P_{n} := \P_{n, \beta}$. Also, let $\tilde{\P}_{n,\beta}$ denote the pushforward of $\P_{n,\beta}$ under the map $\mathscr X\mapsto\wt{\bm W}^{\mathscr X}$. Further, $\wt{\bm W} \in \wt{\mathcal{W}}_k$ and $A \subseteq \wt{\mathcal{W}}_k$, define 
$\delta_\square(\wt{\bm W}, A) :=\inf_{\wt{\bm U} \in A}
\delta_\square( \wt{\bm W}, \wt{\bm U} )$.

\begin{theorem}\label{ThmConcentrationBeta}
Consider a $k$-colored ERGM with sufficient statistic of the form
\eqref{eq:T_betaf_infty}, and let $\mathcal N^*$ be as defined in
\eqref{eq:largetemperaturemaximizer}. For any $\varepsilon>0$, define
\[
\mathcal S_\varepsilon
:=
\left\{
\wt{\bm W}\in\wt{\mathcal W}_k:
\delta_\square(\wt{\bm W},\mathcal N^*)\geq\varepsilon
\right\}.
\]
Then there exist constants $\beta_0=\beta_0(\varepsilon)>0$ and
$K=K(\varepsilon)>0$ such that the following hold: For every fixed $\beta\geq\beta_0$,
there exists $n_0=n_0(\varepsilon,\beta)$ such that for $n\geq n_0$, 
\begin{align}\label{eq:tempertaureset}
\tilde{\P}_{n,\beta}(\mathcal S_\varepsilon)
&=
\P_{n,\beta}\left(
\delta_\square(\wt{\bm W}^{\mathscr X},\mathcal N^*)
\geq\varepsilon
\right)
\leq \exp(-Kn^2) . 
\end{align} 
Further,
\begin{align}\label{eq:tempertauredistance}
\lim_{\beta\rightarrow\infty}
\sup_{\wt{\bm W}\in\Gamma_\beta}
\delta_\square(\wt{\bm W},\mathcal N^*)
=0.
\end{align}
\end{theorem}

The proof of Theorem~\ref{ThmConcentrationBeta} is given in
Section~\ref{sec:ThmConcentrationBetapf}. The result provides a general
framework for understanding ground states of colored ERGMs along any prescribed
parameter direction, through a two-level optimization strategy: \emph{first
maximize the main energy term, and then maximize the second-order energy and entropy functional within the resulting ground-state set}. Later in this section, we will illustrate this general principle
in two examples: the induced wedge model in Section~\ref{sec:temperaturewedge}
and the rainbow triangle model in Section~\ref{sec:rt}.

\begin{remark}[Asymptotics as $\beta \to -\infty$]
\label{remark:temperaturenegativeTW}
The regime $\beta\to-\infty$ can also be obtained from
Theorem~\ref{ThmConcentrationBeta} by setting
$\alpha=-\beta>0$ and $r(\bm W)=-f(\bm W)$. Then the conclusion of
Theorem~\ref{ThmConcentrationBeta} applies to $T(\bm W)=\alpha r(\bm W)+g(\bm W)$,
as $\alpha\to\infty$. 
\end{remark}

Theorem~\ref{ThmConcentrationBeta} connects to several well-known results in the classical uncolored setting and in extremal combinatorics. We discuss some of these connections in the following remarks.

\begin{remark}[Zero temperature regime in the classical uncolored ERGMs]
In the classical two-color/scalar graphon setting, \cite[Theorem 7.1]{ChatterjeeDiaconisExponential2013} derived the asymptotics of ERGMs with sufficient statistic of the form: 
\begin{align}\label{eq:temperatureclassicalTW}
T_{\beta_2}(W)=\beta_1\,t(K_2, W)+\beta_2\,t(H,W ), 
\end{align}
for a finite graph $H$, as $\beta_2 \rightarrow -\infty$. To obtain their result from Theorem \ref{ThmConcentrationBeta}, set 
$f(W):=-t(H,W)$ and $g(W):=\beta_1\,t(K_2,W)$ in \eqref{eq:T_betaf_infty}. Then \eqref{eq:largetemperaturemaximizer} simplifies to (up to the normalization convention used for the entropy functional), 
\begin{align}\label{eq:NW}
\mathcal{N}^* = \argmax_{t(H,W)=0} \left\{\beta_1\,t(K_2,W)-\frac{1}{2}\mathcal I(W)\right\},
\end{align}
where $\mathcal I(W) = \int_{[0, 1]^2} I_2(W(x, y), 1-W(x,y)) \mathrm d x \mathrm dy$. Note that 
$$\beta_1\,t(K_2,W)-\frac{1}{2}\mathcal I(W) = - \frac{1}{2}J_p(W) - \frac{1}{2}\log(1-p),$$ 
where $p=\frac{e^{2\beta_1}}{1+e^{2\beta_1}}$ and $J_p(W) := \int_{[0, 1]^2} J_p(W(x, y)) \mathrm d x \mathrm dy  =\mathcal{I}_{(p,1-p)}((W,1-W))$ (recall \eqref{eq:IcolorW}), with $J_p(u) := u \log \frac{u}{p} + (1 - u) \log \frac{1 - u}{1-p}=D_2((u,1-u), (p,1-p))$ (recall \eqref{eq:divergence}), for $u \in [0, 1]$. Hence, \eqref{eq:NW} simplifies to the constrained entropy minimization problem
$$
\mathcal{N}^* =  \argmin\left\{J_p(W):\ t(H,W)=0\right\}.
$$
This is exactly the variational problem underlying
\cite[Theorem~7.1]{ChatterjeeDiaconisExponential2013}. 
\end{remark}

\begin{remark}[Quantization and extremal combinatorics]
One can also let both parameters $\beta_1$ and $\beta_2$ in \eqref{eq:T_betaf_infty} diverge along straight lines, leading to nontrivial asymptotic extremal behavior. For instance, Yin et al.~\cite{YinRinaldoFadnavis_AoP_quantiz} consider ERGMs with sufficient statistic of the form \eqref{eq:temperatureclassicalTW} in the asymptotic regime where $\beta_2 \to -\infty$ and $\beta_1 = a\beta_2 + b$. In this case, setting $\beta=\beta_2$, the limiting variational problem is governed by $$ \beta\, t(H,W)+(a\beta+b)t(K_2,W)-\frac12\mathcal I(W). $$ In the language of \Cref{ThmConcentrationBeta}, this corresponds to taking the leading-order functional to be $f(W):=t(H,W)+a\,t(K_2,W)$, and the lower-order term to be $g(W):=b\,t(K_2,W)$. The associated tie-breaker functional is therefore $h(W)=b\,t(K_2,W)-\frac12\mathcal I(W)$. Thus, our selection principle shows that the zero-temperature limit is resolved in two steps. First, since $\beta\to-\infty$, the tilted measure concentrates on the ground-state manifold $\mathcal N_f$ of minimizers of $f(W)$. Then, among these ground states, it selects those maximizing the entropy-tilted functional $h(W)$. In the case of \cite[Theorem~3.3]{YinRinaldoFadnavis_AoP_quantiz}, where $H=K_3$ is the triangle, the minimizers of the leading-order problem are described, depending on the scaling parameter $a$, by complete multipartite graphons (Tur\'an graphons). This description relies on known extremal results for the lower envelope of triangle density as a function of edge density (the Razborov curve \cite{razborov2008minimal}). When several such ground states are available, the lower-order parameter $b$ resolves the degeneracy by selecting the entropy-maximizing ground state. This produces the quantization phenomenon observed in \cite{YinRinaldoFadnavis_AoP_quantiz}: the limiting graphon is forced onto a discrete family of Tur\'an graphons, and consequently the limiting edge density takes values in the discrete set, rather than varying continuously.
\end{remark}

\begin{remark}[Maximum entropy graphons]
The asymptotic regimes $\beta\to\pm\infty$ provide a natural bridge between colored ERGMs and the literature on subgraph-constrained maximum-entropy graphons. 

\begin{itemize}

 \item In the regime $\beta\to+\infty$, concentration on $\mathcal N^*$ shows that the model first selects graphons maximizing the target functional $f(W)$, such as a prescribed homomorphism density. Among these energy maximizers, the limiting measure then favors those graphons that maximize the tie-breaker variational functional $h(W)=g(W)-\frac12\mathcal I(W)$.  In the special case $g\equiv 0$, this reduces to maximizing the graphon entropy $-\frac12\mathcal I(W)$ over the set of graphons that maximize $f$. Thus, the zero-temperature limit is closely related to subgraph-constrained maximum-entropy problems, in which one seeks the most typical graphon subject to fixed or extremal subgraph-density constraints (see, for example, \cite{radin2014asymptotics,kenyon2017multipodal,kenyon2017phases, neeman2023typical} and the references therein). 
 
 \item Conversely, the regime $\beta\to-\infty$ forces minimization of $f$. In the canonical case where $f\geq 0$ records the density of a prescribed family of subgraphs, the condition $f(W)=0$ is precisely a forbidden-pattern constraint. The admissible graphons then describe the closure of a hereditary graph property. If $g\equiv 0$, the concentration result shows that the model selects entropy-maximizing graphons among all graphons avoiding the forbidden patterns. This aligns with the graph-limit approach to typical structure in hereditary properties, where the asymptotic enumeration and typical behavior of graphs in a property are governed by entropy-maximizing graphons \cite{EntropyGraphonsSzegediJansonHatamy}. It also parallels the corresponding decorated and multicolored theory developed using graph-limit and container methods \cite{falgasravry2016multicolour, falgasravry2019multicolor, falgas-ravryRectilinearApproximationVolume2023}.  
\end{itemize} 
\end{remark}

\subsubsection{Induced Wedge ERGM}
\label{sec:temperaturewedge}

In classical uncolored ERGMs, sufficient statistics are usually taken to be linear combinations of homomorphism densities (as in \eqref{eq:sufficientTW}). However, as briefly discussed after Example~\ref{example:triangleinteractions}, the situation can be quite different if one considers induced homomorphism densities instead. To investigate this further we consider the induced wedge model, an uncolored ERGM with sufficient statistic 
\begin{align}\label{eq:wedgeTW} 
T_{\wedge}(W) := \beta_1 t(K_2,W) + \beta_2 t_{\mathrm{ind}}(K_{1,2},W), 
\end{align} 
where $\beta_1,\beta_2\in\mathbb R$, $W:[0,1]^2\to[0,1]$ is a graphon, and $t_{\mathrm{ind}}(K_{1,2},W)$ denotes the induced homomorphism density of the $2$-star, (wedge) in $W$ (as defined in \eqref{eq:wedgeW}). An ERGM with sufficient statistic as in \eqref{eq:wedgeTW} will be referred to as the {\it induced wedge ERGM}.

\begin{remark}[Induced wedges and structural balance]
\label{remark:wedge}
The induced wedge statistic admits a natural interpretation in terms of structural balance theory \citep{heider1946attitudes,cartwright1956structural,davis1967clustering, leskovec2010signed}. Suppose, for example, that $G_n$ is a graph on $n$ vertices sampled from the ERGM with sufficient statistic as in \eqref{eq:wedgeTW}. Interpret the edges of $G_n$ as negative $(-)$ relations, such as rivalry or conflict, and the non-edges as positive $(+)$ relations, such as friendship or alliance. Then an induced wedge encodes the balance-theoretic phenomenon `the enemy of my enemy is my friend' (see \cite{fritz2025exponential,feng2022testing} and the references therein). Consequently, this ERGM can be viewed as a binary signed-network model in which $\beta_1$ controls the overall density of negative ties, while $\beta_2$ controls the tendency toward triads of type $\{-,-,+\}$. A positive value of $\beta_2$ favors such balanced configurations, whereas a negative value suppresses them. In particular, when $\beta_2$ is large and positive, the model favors structures in which enemies of a common actor tend to be friends with one another. 
\end{remark}

In this section, our aim is to understand, using the selection principle underlying  Theorem~\ref{ThmConcentrationBeta}, 
the behavior of the induced wedge model as $\beta_2\to\infty$, with $\beta_1$ held fixed. To begin with, note that, by Theorem~\ref{thm:variational_formula}, the limiting free energy in this case is given by 
\begin{align}\label{eq:wedgetemperatureTW}
\lim_{n\to\infty}\psi_n = \psi_{\wedge}(\beta_2) := \sup_{W\in\wt{\mathcal W}} \left\{ T_{\wedge}(W)-\frac12\mathcal I(W) \right\}. 
\end{align}
Although $\psi_{\wedge}$ depends on both $\beta_1$ and $\beta_2$, we keep only $\beta_2$ in the argument to emphasize that $\beta_2$ is the parameter of interest, while $\beta_1$ is held fixed throughout this analysis. Also, define 
\begin{align}\label{eq:wedgemaximizer}
 \mathcal M^*_{\wedge} (\beta_2) := \argmax_{W\in\wt{\mathcal W}} \left\{ T_{\wedge}(W)-\frac12\mathcal I(W) \right\}. 
\end{align}

We begin with the case $\beta_2 \rightarrow + \infty$. In this regime, the model favors graphons that maximize the induced wedge density.  Towards this, we show \Cref{lem:wedge-max-bipartite} that $\max_{W} t_{\mathrm{ind}}(K_{1,2},W)=\frac14$ and equality holds if and only if $W$ is the {\it complete balanced bipartite graphon}: 
\begin{align}\label{eq:bipartiteW}
W_{\mathrm{bip}}(x,y)=\bm 1 \left\{[0, \tfrac{1}{2}] \times [\tfrac{1}{2}, 1]\right\}(x,y)+\bm 1\left\{[\tfrac{1}{2}, 1] \times [0, \tfrac{1}{2}]\right\}(x,y) , 
\end{align}
up to a measure-preserving transformation. 
Applying Theorem \ref{ThmConcentrationBeta} now leads to the following result. The proof is given in Section \ref{sec:wedge-beta-plus-inftypf}.

\begin{theorem}
\label{thm:wedge-beta-plus-infty}
Let $\psi_{\wedge}(\beta_2)$ and $ \mathcal M^*_{\wedge} (\beta_2)$ be defined in \eqref{eq:wedgetemperatureTW} and \eqref{eq:wedgemaximizer}, respectively. 
Then 
\begin{align}\label{eq:wedgeoptimizer}
\lim_{\beta_2\to+\infty} \sup_{ \tilde{ W} \in \mathcal  \mathcal M^*_{\wedge} (\beta_2) }\delta_\square( \tilde{W} ,  \tilde W_{\mathrm{bip}})  = 0, 
\end{align}
where $\tilde W_{\mathrm{bip}}$ is the equivalence class of the complete balanced bipartite graphon 
$W_{\mathrm{bip}}$ in \eqref{eq:bipartiteW}. Moreover, 
\begin{align}\label{eq:temperaturewedgeW}
\lim_{\beta_2\to+\infty}\frac{\psi_{\wedge}(\beta_2)}{\beta_2}=\frac14. 
\end{align}
\end{theorem}

\begin{remark}[Symmetry breaking for induced wedges at positive temperatures]
The above result shows that, in the limit $\beta_2\to\infty$, the induced wedge ERGM concentrates on the complete balanced bipartite graphon. Intuitively, this means that for $\beta_2$ sufficiently large, a graph sampled from the induced wedge ERGM should resemble a complete balanced bipartite graph. On the other hand, it follows from Theorem~\ref{thm:highT_colored} that when $|\beta_2|$ is sufficiently small (the high-temperature regime), the model admits a unique constant (replica symmetric) solution and is therefore asymptotically indistinguishable from an Erd\H{o}s--R\'enyi graph. This suggests the existence of a finite-temperature threshold $\beta_*\in(0,\infty)$ at which the model transitions from a replica-symmetric phase to a symmetry-broken phase. (For explicit bounds on the threshold $\beta_*$, see Proposition~\ref{ppn:finite_beta_wedge_symmetry_breaking}). This contrasts with ERGMs based on linear combinations of homomorphism densities (not necessarily induced), for which \cite[Theorem~4.1]{ChatterjeeDiaconisExponential2013} establishes replica symmetry for all non-negative parameter values.  The induced wedge density, however, falls outside this regime, since it is the difference between the homomorphism densities of the wedge and the triangle (recall \eqref{eq:wedgeW}).
Thus, the induced wedge ERGM can exhibit symmetry breaking, making it a natural
model for clustered structure, particularly in signed networks.  
\end{remark}

Next, we consider the regime $\beta_2\to-\infty$. Here, the leading constraint is $t_{\mathrm{ind}}(K_{1,2},W)=0$. It is well known that this condition is equivalent to $W$ being a \emph{cluster graphon}, that is, a graphon representing a disjoint union of
cliques (see \cite[Theorem 7.1]{janson2013graphlimits}). More precisely,
\begin{align}\label{eq:disjointW}
W(x,y)
=
\sum_{i\geq 1}\bm 1\{(x,y)\in A_i\times A_i\},
\qquad \text{ for } (x,y)\in[0,1]^2,
\end{align}
for some pairwise disjoint measurable sets $\{A_i\}_{i\geq 1}$ in $[0,1]$
satisfying $\sum_{i\geq 1}\lambda(A_i)\leq 1$, where $\lambda$ denotes Lebesgue measure. The remaining set $A_0:=[0,1]\setminus\bigcup_{i\geq 1}A_i$ is the `dust' part  where the graphon is identically zero. We denote by $\mathcal W_{\mathrm{clus}}$ the collection of cluster graphons, and let $\wt{\mathcal W}_{\mathrm{clus}}$  be the corresponding collection of equivalence classes. Applying Theorem~\ref{ThmConcentrationBeta} now yields the following result. Hereafter, for $p\in[0,1]$, we write $W\equiv p$ for the graphon that is equal to the constant $p$ on $[0, 1]^2$.

\begin{theorem} 
\label{thm:wedge-ground-states} 
Let $\psi_{\wedge}(\beta_2)$ and $ \mathcal M^*_{\wedge} (\beta_2)$ be defined in \eqref{eq:wedgetemperatureTW} and \eqref{eq:wedgemaximizer}, respectively. 
Then 
\begin{align}\label{eq:wedgenegativeoptimizer} 
\lim_{\beta_2\to - \infty} \sup_{ \tilde{ W} \in  \mathcal M^*_{\wedge} (\beta_2) }\delta_\square( \tilde{W} ,  \mathcal N^*_{\wedge})  = 0, 
\end{align}
where 
\begin{align}\label{eq:wedgenegativeN} 
\mathcal N^*_{\wedge} := 
\begin{cases}
\{W \equiv 0\} & \text{ if }  \beta_1 < 0, \\ 
\tilde{\mathcal W}_{\mathrm{clus}} & \text{ if }  \beta_1 =0 , \\ 
\{W \equiv 1\} & \text{ if }  \beta_1 >0 . 
\end{cases} 
\end{align}
Moreover, 
\begin{align}\label{eq:wedgenegativeZ} 
\lim_{\beta_2\to -\infty} \psi_{\wedge}(\beta_2) 
= \begin{cases}
0  & \text{ if } \beta_1\le 0, \\
\beta_1 & \text{ if } \beta_1>0.
\end{cases}
\end{align}
\end{theorem}

The proof of Theorem~\ref{thm:wedge-ground-states} is given in Section~\ref{sec:wedge-ground-statespf}. The result shows that in the $\beta_2\to-\infty$ limit, the induced wedge ERGM undergoes a sharp selection phenomenon governed by the edge-density parameter $\beta_1$. In this regime, the leading-order effect of the negative induced wedge parameter is to force the constraint $t_{\mathrm{ind}}(K_{1,2},W)=0$,  so that the admissible limiting graphons are precisely the cluster graphons. Within this ground-state set, the lower-order term $\beta_1 t(K_2,W)$ determines which cluster graphons are selected. More precisely, if $\beta_1<0$, the edge-density term favors the smallest possible edge density among cluster graphons, and the model concentrates around the empty graphon $W\equiv 0$. If $\beta_1>0$, it favors the largest possible edge density, and the model concentrates around the complete graphon $W\equiv 1$. At the critical value $\beta_1=0$, the edge-density term no longer breaks the degeneracy among cluster graphons, and every feasible cluster graphon is optimal.

\subsubsection{Rainbow Triangles ERGM}
\label{sec:rt}

In this section, we derive the zero-temperature asymptotics of the rainbow triangle ERGM, the $3$-colored ERGM with sufficient statistic given in \eqref{eq:rainbowT}. Specifically, we study the regime in which the coefficient $\beta_4$ of the rainbow triangle density diverges to $\pm\infty$, while the edge-density parameters $\beta_1,\beta_2,\beta_3$ are held fixed. 
As in the previous section, by Theorem~\ref{thm:variational_formula}, the limiting free energy in this case is given by 
\begin{align}\label{eq:rbtemperatureTW}
\lim_{n\to\infty}\psi_n = \psi_{\mathrm{rb}\triangle}(\beta_4) := \sup_{\bm W\in\wt{\mathcal W}_3} \left\{ T(\bm W)-\frac12\mathcal I(\bm W) \right\}. 
\end{align}
for $T(\bm W)$ as defined in \eqref{eq:rainbowT}. Also, define 
\begin{align}\label{eq:rbmaximizer}
\mathcal M^*_{\mathrm{rb}\triangle}(\beta_4) := \argmax_{\bm W\in\wt{\mathcal W}_3} \left\{ T(\bm W)-\frac12\mathcal I(\bm W) \right\}. 
\end{align}

We begin with the regime $\beta_4\to+\infty$. In this case, the leading-order term favors $3$-colored probability graphons that maximize the rainbow triangle density.
A first natural candidate is the uniform 3-color probability graphon $\bm{U}^{(3)}(x,y)=\left(\frac13,\frac13,\frac13\right)$, for all $x,y\in[0,1]$. For this, recalling \eqref{eq:tWrainbow}, one has  
\begin{equation*}
t_{\mathrm{rb}\triangle}(\bm U^{(3)}) = \frac{1}{3^3}  = \frac{1}{27} \quad \text{ and } \quad t_{\mathrm{rb}\triangle}^{\mathrm{sym}}(\bm U^{(3)}) = \frac{2}{9}.
\end{equation*}  
However, as mentioned in Section~\ref{sec:rc}, triangles are $3$-rainbow
uncommon \cite{erdos1972ramsey}. In graphon language, this means that there
exists a $3$-colored probability graphon with rainbow triangle density larger
than that attained by $\bm U^{(3)}$ (recall Definition~\ref{defn:rc}). In fact,
determining the maximum possible density of rainbow triangles over all
$3$-colorings of $K_n$ is a delicate question, going back to
Erd\H{o}s and S\'os \cite{rodl1990mathematics}. Using flag algebras,
Balogh et al.~\cite{baloghRainbowTrianglesThreecolored2017} settled this
problem by showing that the maximum density is asymptotically $\frac{2}{5}$ (see
Theorem~\ref{thm:rainbow} for the formal statement). Their normalization is the density of an unordered rainbow triangle,
equivalently the color-symmetrized graphon statistic
$t_{\mathrm{rb}\triangle}^{\mathrm{sym}}=6t_{\mathrm{rb}\triangle}$ (recall \eqref{eq:tWrainbow_sym}).
Therefore, their extremal value $\frac{2}{5}$ corresponds to
\begin{equation}\label{eq:fixed_conversion}
\sup_{\bm W\in\mathcal W_3}
t_{\mathrm{rb}\triangle}(\bm W)
=
\frac16\cdot\frac25
=
\frac1{15}
\end{equation}
under the fixed decorated Hamiltonian used in this paper. The extremizers are
unchanged by this factor-six conversion. The extremal value is attained in the limit by the following iterative construction.

\begin{definition}[Iterated blowups of properly 3-colored $K_4$] 
\label{defn:rainbowmaximumgraph}
Consider the following $3$-coloring of $K_4$ with vertex set $\{1,2,3,4\}$: 
\begin{align}\label{eq:color1234}
 \mathscr X(12)=\mathscr X(34)=1,\qquad \mathscr X(13)=\mathscr X(24)=2,\qquad \mathscr X(14)=\mathscr X(23)=3. 
\end{align}  
This is a proper $3$-coloring of $K_4$, and every triangle is rainbow (see Figure~\ref{fig:iterated-k4-blowup-graphon} (a)). We refer to it as the \emph{properly $3$-colored $K_4$ template}. Given this template, its {\it balanced $1$-step blowup} is the following $3$-coloring of $K_n$: 
\begin{itemize} 

\item Partition $[n]$ into four parts $V_1,V_2,V_3,V_4$ with sizes as equal as possible. 

\item Color every edge between $V_i$ and $V_j$ according to the color of the edge $(i ,j)$ in the template $K_4$, for $1 \leq i \ne j \leq 4$. 
\end{itemize}
Next, to obtain the {\it balanced $2$-step blowup}, repeat  the above 1-step construction inside each of the four diagonal blocks $V_i\times V_i$, for $i \in [4]$. Continuing in this way, the {\it balanced $t$-step blowup}, for $t \geq 3$, is obtained by recursively applying the same $K_4$-template blowup inside every diagonal block produced at the previous step.
\end{definition}

The continuum analogue of the above construction is a sequence of $3$-colored probability graphons, defined iteratively as follows. 

\begin{itemize} 

\item Define $\bm W^{(1)} := (W^{(1)}_1,W^{(1)}_2,W^{(1)}_3) \in \mathcal W_3$ as the empirical $3$-colored probability graphon (recall \eqref{eq:Wa}) associated with the properly $3$-colored $K_4$ template (the coloring $\mathscr X$ in \eqref{eq:color1234}).

\item For $t \geq 2$, given $\bm W^{(t-1)}$, construct $\bm W^{(t)} = (W^{(t)}_1,W^{(t)}_2,W^{(t)}_3) \in \mathcal W_3$ by first partitioning $[0,1]$ into four intervals $A_1, A_2, A_3, A_4$ of equal length. On each off-diagonal block $A_i\times A_j$, with $i\neq j$, assign the constant color prescribed by the properly $3$-colored $K_4$ template: 
$$W^{(t)}_a(x,y)=1 \quad\text{if and only if } \mathscr X(i,j)=a,$$ 
for $(x,y)\in A_i\times A_j$. On each diagonal block $A_i\times A_i$ place a rescaled copy of $\bm W^{(t-1)}$. More precisely, if $\phi_i:A_i\to[0,1]$ denotes the affine bijection, for $i \in [4]$, then $$W^{(t)}_a(x,y) = W^{(t-1)}_a(\phi_i(x),\phi_i(y)),$$ for $(x,y)\in A_i\times A_i$ and $a\in[3]$. Thus, $\bm W^{(2)}$ is obtained by repeating the 1-step construction inside each of the four diagonal blocks of $\bm W^{(1)}$, and the graphon $\bm W^{(t)}$ is obtained by iterating this procedure $t$ times (see Figure~\ref{fig:iterated-k4-blowup-graphon} (b)).
\end{itemize}
Denote by $\wt{\bm W}^{(t)}$ the equivalence class of $\bm W^{(t)}$. Note that $\{\wt{\bm W}^{(t)}\}_{t \geq 1}$ is a Cauchy sequence in the compact space $(\tilde{\mathcal W}_3, \delta_\square)$, hence, $\{\wt{\bm W}^{(t)}\}_{t \geq 1}$ has a limit $\tilde{\bm W}^{\mathrm{rb}}$, which we call the {\it blowup $3$-colored probability graphon of the $K_4$ template.}

\begin{figure}[htbp]
\centering
\resizebox{0.90\textwidth}{!}{%
\begin{tikzpicture}[
    every node/.style={font=\small},
    vertex/.style={
        circle,
        draw=black,
        fill=white,
        thick,
        minimum size=7mm,
        inner sep=0pt
    },
    arr/.style={-{Latex[length=3mm]}, thick}
]

\definecolor{colone}{RGB}{215,39,40}
\definecolor{coltwo}{RGB}{31,119,180}
\definecolor{colthree}{RGB}{44,160,44}

\draw[colone, line width=1.5pt]
    (-2.05,3.00) -- (-1.55,3.00);
\node[anchor=west] at (-1.45,3.00) {color $1$};

\draw[coltwo, line width=1.5pt]
    (-0.25,3.00) -- (0.25,3.00);
\node[anchor=west] at (0.35,3.00) {color $2$};

\draw[colthree, line width=1.5pt]
    (1.55,3.00) -- (2.05,3.00);
\node[anchor=west] at (2.15,3.00) {color $3$};

\coordinate (v1) at (-1.45, 1.65);
\coordinate (v2) at ( 1.45, 1.65);
\coordinate (v3) at (-1.45,-1.25);
\coordinate (v4) at ( 1.45,-1.25);

\draw[colone, line width=2.2pt] (v1)--(v2);
\draw[colone, line width=2.2pt] (v3)--(v4);

\draw[coltwo, line width=2.2pt] (v1)--(v3);
\draw[coltwo, line width=2.2pt] (v2)--(v4);

\draw[colthree, line width=2.2pt] (v1)--(v4);
\draw[colthree, line width=2.2pt] (v2)--(v3);

\node[vertex] at (v1) {$1$};
\node[vertex] at (v2) {$2$};
\node[vertex] at (v3) {$3$};
\node[vertex] at (v4) {$4$};

\draw[arr] (2.85,0.20) -- (4.05,0.20);
\node at (3.45, 0.67) {blow up};
\node at (3.45,-0.27) {and iterate};

\def\xG{5.55}
\def\yG{-2.05}
\def\LG{5.00}

\def\KIVlist{%
  0/1/colone,   1/0/colone,
  2/3/colone,   3/2/colone,
  0/2/coltwo,   2/0/coltwo,
  1/3/coltwo,   3/1/coltwo,
  0/3/colthree, 3/0/colthree,
  1/2/colthree, 2/1/colthree%
}

\foreach \bi/\bj/\bcolor in \KIVlist {
  \fill[\bcolor]
    ({\xG + \bj*\LG/4},
     {\yG + \LG - (\bi+1)*\LG/4})
    rectangle ++({\LG/4},{\LG/4});
}

\foreach \dA in {0,1,2,3} {
  \foreach \bi/\bj/\bcolor in \KIVlist {
    \fill[\bcolor]
      ({\xG + \dA*\LG/4 + \bj*\LG/16},
       {\yG + \LG - \dA*\LG/4 - (\bi+1)*\LG/16})
      rectangle ++({\LG/16},{\LG/16});
  }
}

\foreach \dA in {0,1,2,3} {
  \foreach \dB in {0,1,2,3} {
    \foreach \bi/\bj/\bcolor in \KIVlist {
      \fill[\bcolor]
        ({\xG + \dA*\LG/4 + \dB*\LG/16
               + \bj*\LG/64},
         {\yG + \LG - \dA*\LG/4 - \dB*\LG/16
               - (\bi+1)*\LG/64})
        rectangle ++({\LG/64},{\LG/64});
    }
  }
}

\foreach \dA in {0,1,2,3} {
  \foreach \dB in {0,1,2,3} {
    \foreach \dC in {0,1,2,3} {
      \foreach \bi/\bj/\bcolor in \KIVlist {
        \fill[\bcolor]
          ({\xG + \dA*\LG/4 + \dB*\LG/16
                 + \dC*\LG/64 + \bj*\LG/256},
           {\yG + \LG - \dA*\LG/4 - \dB*\LG/16
                 - \dC*\LG/64 - (\bi+1)*\LG/256})
          rectangle ++({\LG/256},{\LG/256});
      }
    }
  }
}

\foreach \k in {0,1,2,3,4} {
  \pgfmathsetmacro{\xx}{\xG+\LG*\k/4}
  \pgfmathsetmacro{\yy}{\yG+\LG*\k/4}

  \draw[black, line width=0.45pt]
      (\xx,\yG) -- (\xx,{\yG+\LG});

  \draw[black, line width=0.45pt]
      (\xG,\yy) -- ({\xG+\LG},\yy);
}

\draw[black, line width=0.65pt]
    (\xG,\yG) rectangle ++(\LG,\LG);

\foreach \i/\lab in {0/A_1,1/A_2,2/A_3,3/A_4} {
  \pgfmathsetmacro{\xx}{\xG+\LG*(\i+0.5)/4}
  \node at (\xx,{\yG+\LG+0.30}) {$\lab$};
}

\foreach \i/\lab in {0/A_1,1/A_2,2/A_3,3/A_4} {
  \pgfmathsetmacro{\yy}{
      \yG+\LG-\LG*(\i+0.5)/4
  }
  \node[anchor=east] at ({\xG-0.18},\yy) {$\lab$};
}

\node at ({\xG+\LG/2},{\yG-0.45}) {$y$};
\node[rotate=90] at
    ({\xG-0.58},{\yG+\LG/2}) {$x$};

\node[anchor=east] at
    ({\xG-0.14},{\yG+\LG}) {$0$};

\node[anchor=west] at
    ({\xG+\LG+0.14},{\yG+\LG}) {$1$};

\node[anchor=east] at
    ({\xG-0.14},\yG) {$1$};

\node[anchor=west] at
    ({\xG+\LG+0.14},\yG) {$1$};

\def\yPanelCaption{-3.00}

\node[
    anchor=north,
    align=center,
    text width=4.6cm,
    inner sep=0pt
] at (0,\yPanelCaption) {
    \textup{ (a) }
};

\node[
    anchor=north,
    align=center,
    text width=5.8cm,
    inner sep=0pt
] at ({\xG+\LG/2},\yPanelCaption) {
    \textup{ (b)} } ; 
\end{tikzpicture}%
}
\caption{ \small{ (a) The properly $3$-colored $K_4$ template and (b) its iterated
blowup $3$-colored probability graphon. In panel (b), the off-diagonal blocks follow the $K_4$ color template, while every diagonal block contains
a scaled copy of the whole pattern. } }
\label{fig:iterated-k4-blowup-graphon}
\end{figure}
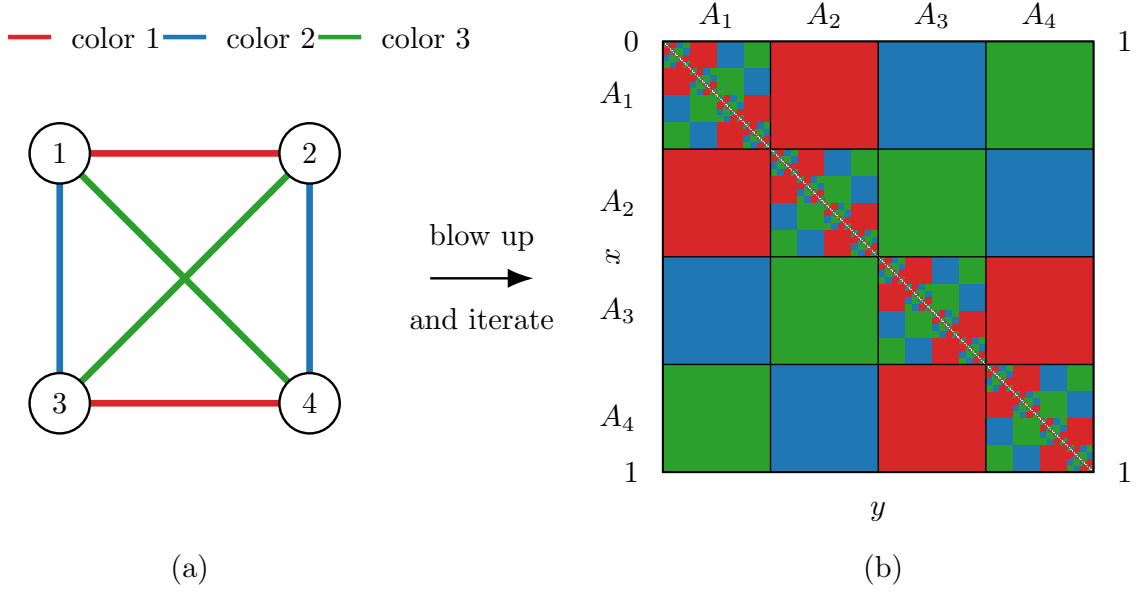

\begin{remark}[Rainbow triangle density in $\tilde{\bm W}^{\mathrm{rb}}$]  
Recall the definition of the fixed decorated rainbow triangle homomorphism density $t_{\mathrm{rb}\triangle}(\cdot)$ from \eqref{eq:tWrainbow}. 
Using the iterated construction above, we can derive a recursion for $t_{\mathrm{rb}\triangle}(\bm W^{(t)})$. For this, sample an ordered triple $(x,y,z)$ uniformly from $[0,1]^3$ and consider the following cases relative to the top-level four-part partition in the construction of $\bm W^{(t+1)}$:
\begin{itemize}
\item First, suppose that $x,y,z$ fall in three distinct top-level parts. This event has probability $\frac{4\cdot3\cdot2}{4^3}=\frac38$. The properly $3$-colored $K_4$ template guarantees that the corresponding
triangle is rainbow. However, among the six orderings of the vertices of a
rainbow triangle, exactly one has the prescribed color pattern $\operatorname{col}(x,y)=1$, $\operatorname{col}(y,z)=2$, and $\operatorname{col}(z,x)=3$. Consequently, the contribution of triples falling in three distinct parts to the fixed decorated density is $\frac16\cdot\frac38=\frac1{16}$. 
\item Second, suppose that all three points fall in the same top-level part. The probability of this event is $4\left(\frac14\right)^3=\frac1{16}$. Inside each such diagonal block, the construction is a rescaled copy of $\bm W^{(t)}$, and therefore this case contributes $\frac1{16}\, t_{\mathrm{rb}\triangle}(\bm W^{(t)})$. 
\item Finally, if exactly two of the three points lie in the same top-level part,
then the two cross-edges joining these points to the third part have the same
deterministic color. Such a triangle cannot be rainbow, hence, this case
contributes zero. 
\end{itemize}
Combining the three cases yields 
\begin{equation}\label{eq:fixed_blowup_recursion}
t_{\mathrm{rb}\triangle}(\bm W^{(t+1)})
=
\frac1{16}
+
\frac1{16}
t_{\mathrm{rb}\triangle}(\bm W^{(t)}).
\end{equation}
It remains to determine the initial value. Note that, by definition, $\bm W^{(1)}$ is the
empirical probability graphon associated with the properly colored $K_4$
template, with the uniform distribution
$(\frac13,\frac13,\frac13)$ on each of its four diagonal blocks. The
off-diagonal contribution is again $\frac1{16}$. Moreover, the probability
that all three sampled points fall in the same diagonal block is $\frac1{16}$,
and, conditional on this event, the fixed decorated density is $\frac13\cdot\frac13\cdot\frac13=\frac1{27}$. Therefore, $t_{\mathrm{rb}\triangle}(\bm W^{(1)}) = \frac1{16} + \frac1{16}\cdot\frac1{27} = \frac7{108}$. Now, solving the recursion \eqref{eq:fixed_blowup_recursion} with the initial value gives 
\begin{equation}\label{eq:fixed_blowup_solution}
t_{\mathrm{rb}\triangle}(\bm W^{(t)})
=
\frac1{15}
-
\frac{1}{540\,16^{t-1}},
\qquad t\ge1.
\end{equation}
This implies, $t_{\mathrm{rb}\triangle}(\bm W^{\mathrm{rb}}) = \frac1{15}$, which is consistent with the result in \cite{baloghRainbowTrianglesThreecolored2017} (recall \eqref{eq:fixed_conversion}). 
\label{remark:trbW}
\end{remark}

Using the above extremal construction combined with Theorem \ref{ThmConcentrationBeta}, we now derive in the next theorem the asymptotic behavior of the rainbow triangle ERGM as $\beta_4 \rightarrow \infty$. 
The proof of the result is given in Section \ref{sec:concentrationrainbowpf}.

\begin{theorem}\label{thm:concentrationrainbow} 
Let $\psi_{\mathrm{rb}\triangle}(\beta_4)$ and $\mathcal M^*_{\mathrm{rb}\triangle}(\beta_4)$ be defined in \eqref{eq:rbtemperatureTW} and \eqref{eq:rbmaximizer},  respectively.  Then 
\begin{align}\label{eq:rainbowdistanceN}
\lim_{\beta_4 \to+\infty} \sup_{ \tilde{ \bm W} \in \mathcal M^*_{\mathrm{rb}\triangle}(\beta_4) }\delta_\square( \tilde{\bm{W}} ,  \wt{\bm W}^{\mathrm{rb}})  = 0. 
\end{align} 
Further,  
\begin{align}\label{eq:rainbowZ}
\lim_{\beta_4\to+\infty}\frac{\psi_{\mathrm{rb}\triangle}(\beta_4)}{\beta_4}= \frac{1}{15} . 
\end{align} 
\end{theorem}

Next, we consider the regime $\beta_4\to-\infty$. In this regime, the leading-order constraint is $t_{\mathrm{rb}\triangle}(\bm W)=0$. Finite edge-colorings of complete graphs with no rainbow triangle have been studied extensively in extremal and structural combinatorics. Their structure goes back to Gallai's work on transitively orientable graphs \cite{gallaiTransitivOrientierbareGraphen1967}; in the modern graph-coloring language, such colorings are commonly called \emph{Gallai colorings}. The terminology and several basic consequences are discussed in \cite{gyarfasEdgeColoringsComplete2004}. Subsequent work has studied Ramsey-type properties, structural refinements, and enumerative aspects of Gallai colorings (see, for example, \cite{gyarfasRamseytypeResultsGallai2010, baloghTypicalStructureGallai2019,bastosCountingGallai3colorings2019}). For the rainbow triangle ERGM, Theorem~\ref{ThmConcentrationBeta} implies that the asymptotics as $\beta_4\to-\infty$ are governed by a secondary variational problem on the no-rainbow triangle ground-state set. More precisely, after the leading term forces $t_{\mathrm{rb}\triangle}(\bm W)=0$, the limiting behavior is determined by maximizing the tie-breaker functional $\sum_{s=1}^3 \beta_s t(K_2^{(s)},\bm W)-\frac12\mathcal I(\bm W)$ over this constrained class. This leads to the optimization problem:  \begin{align}\label{eq:rt_negative_tiebreaker} 
\sup_{\tilde{\bm W} \in \tilde{\mathcal W}_3} \left\{ \sum_{s=1}^3 \beta_s t(K_2^{(s)},\bm W)-\frac12\mathcal I(\bm W): t_{\mathrm{rb}\triangle}(\bm W)=0 \right\}. 
\end{align}  
We show that the optimizers for this problem are constant bichromatic probability graphons. Specifically, for $1\leq a<b\leq 3$, define $\bm W^{(ab)} = (W^{(ab)}_1,W^{(ab)}_2,W^{(ab)}_3)$ by 
\begin{align}\label{eq:colorW}
W^{(ab)}_a(x,y) = \frac{e^{2\beta_a}}{e^{2\beta_a}+e^{2\beta_b}}, \quad W^{(ab)}_b(x,y) = \frac{e^{2\beta_b}}{e^{2\beta_a}+e^{2\beta_b}}, \quad \text{ and } \quad W^{(ab)}_c(x,y)=0 , 
\end{align} 
for almost every $(x,y)\in[0,1]^2$ and for $c\in[3]\setminus\{a,b\}$. Equivalently, $\bm W^{(ab)}$ is the constant probability graphon supported only on two colors $a$ and $b$, with edge-color distribution given by the two-color Gibbs weights determined by $\beta_a$ and $\beta_b$. The following theorem shows that, as $\beta_4\to-\infty$, the rainbow triangle ERGM concentrates around the equivalence classes of these bichromatic probability graphons corresponding to the maximizing pair of edge parameters.

\begin{theorem}\label{thm:concentration_rainbow_negative} 
Let $\psi_{\mathrm{rb}\triangle}(\beta_4)$ and $\mathcal M^*_{\mathrm{rb}\triangle}(\beta_4)$ be defined in \eqref{eq:rbtemperatureTW} and \eqref{eq:rbmaximizer},  respectively.  Define 
$$\mathcal N^*_{\mathrm{rb}\triangle} = \left\{ \tilde{\bm W}^{(ab)}:(a,b)\in\mathcal P_* \right\} , $$ 
where $ \mathcal P_* := \argmax_{1\leq a<b\leq 3} (e^{2\beta_a}+e^{2\beta_b})$. Then 
\begin{align}\label{eq:rainbowfreedistanceN}
\lim_{\beta_4 \to-\infty} \sup_{ \tilde{ \bm W} \in \mathcal M^*_{\mathrm{rb}\triangle}(\beta_4) }\delta_\square( \tilde{\bm{W}} , \mathcal N^*_{\mathrm{rb}\triangle} )  = 0. 
\end{align} 
Further,  
\begin{align}\label{eq:rainbowfreeZ}
\lim_{\beta_4\to-\infty} \psi_{\mathrm{rb}\triangle}(\beta_4) = \frac{1}{2} \log\left( \max_{1\leq a<b\leq3} \left(e^{2\beta_a}+e^{2\beta_b}\right) \right) . 
\end{align}

\end{theorem}

The proof of Theorem \ref{thm:concentration_rainbow_negative} is given in Section \ref{sec:concentration_rainbow_negativepf}. Note that in the symmetric case 
$\beta_1=\beta_2=\beta_3$,  $ \mathcal P_* =\{(1,2),(1,3),(2,3)\}$, and hence,  
$\mathcal N^*_{\mathrm{rb}\triangle} = \{ \tilde{\bm W}^{(12)}, \tilde{\bm W}^{(23)}, \tilde{\bm W}^{(13)} \} $,  that is, all three balanced bichromatic probability graphons are selected.

\begin{remark}[Gallai constraints and bichromatic phases]
Although the class of no-rainbow graphons contains the graphon limits of all Gallai colorings and can have a rich recursive structure, Theorem \ref{thm:concentration_rainbow_negative} shows that the entropy-tilted optimization problem in \eqref{eq:rt_negative_tiebreaker} has a much simpler set of optimizers: they are probability graphons supported only on two colors. In other words, the Gibbs measure at $\beta_4 \to-\infty$ creates at most coexistence among bichromatic phases, but does not produce genuinely three-color Gallai-type graphons. 
\end{remark}

\subsection{Finite Temperature Symmetry Breaking}
\label{sec:symmetry}

The results in the preceding section show that, in the positive zero-temperature limit, the variational problems for both the induced wedge ERGM and the rainbow triangle ERGM converge to extremal nonconstant probability graphons. On the other hand, Theorem~\ref{thm:highT_colored} implies that, in the high-temperature regime, both models admit unique constant, replica-symmetric solutions. This suggests the existence of a finite-temperature threshold for each model, at which the model transitions from a replica-symmetric phase to a symmetry-broken phase. In this section, we derive explicit bounds on the corresponding thresholds for the induced wedge and rainbow triangle ERGMs. We also complement these results with numerical simulations in Section~\ref{sec:finite_beta_simulations}.

\subsubsection{Induced Wedge ERGM}

Recall the induced wedge ERGM variational problem from \eqref{eq:wedgetemperatureTW}. For simplicity, throughout we will assume the coefficient of the edge density $\beta_1 = 0$, in which case 
\eqref{eq:wedgetemperatureTW} simplifies to 
\begin{align}\label{eq:symmetrywedgeTW}
\psi_{\wedge}(\beta_2) = \sup_{W} \left\{ \beta_2 t_{\mathrm{ind}}(K_{1, 2}, W) - \frac12\mathcal I(W) \right\} . 
\end{align}
Then $\beta_2 > 0 $ is in the replica symmetry regime, whenever $\psi_{\wedge}(\beta_2) = \psi_{\wedge}^{\mathrm{const}}(\beta_2)$, and in the symmetry-breaking regime, whenever $\psi_{\wedge}(\beta_2) > \psi_{\wedge}^{\mathrm{const}}(\beta_2)$, where 
\begin{align}\label{eq:wedgeconstant}
\psi_{\wedge}^{\mathrm{const}}(\beta_2) :=
\sup_{c\in[0,1]}
\left\{ \beta_2 c^2(1-c)-\frac12 I(c) \right\} ,   
\end{align}
with $I(c):= I_2(c,(1-c))= c\log c+(1-c)\log(1-c)$, 
is the value of the variational problem \eqref{eq:wedgetemperatureTW} for the optimal constant graphon.

\begin{proposition}\label{ppn:finite_beta_wedge_symmetry_breaking}
Consider the induced wedge ERGM variational problem in \eqref{eq:symmetrywedgeTW}. Then the following hold: 

\begin{itemize} 

\item[$(1)$] For $|\beta_2|<\frac{1}{3}$, $\psi_{\wedge}(\beta_2) = \psi_{\wedge}^{\mathrm{const}}(\beta_2)$, that is, the model is in the replica symmetric regime. 

\item[$(2)$] For $\beta_2\ge 3.054$, $\psi_{\wedge}(\beta_2) > \psi_{\wedge}^{\mathrm{const}}(\beta_2)$, that is, the model is in the symmetry breaking regime. 
 
\end{itemize} 
\end{proposition}

The proof of the result is given in Section~\ref{sec:finite_beta_wedge_symmetry_breakingpf}. It proceeds in two steps. 
\begin{itemize} 
\item The replica-symmetry statement follows from a direct application of Theorem~\ref{thm:highT_colored}. It is worth noting that applying \cite[Theorem~6.2]{ChatterjeeDiaconisExponential2013} gives replica symmetry only for $|\beta_2|<\frac14$ in this model. This is weaker than the region obtained in Proposition~\ref{ppn:finite_beta_wedge_symmetry_breaking} from the colored high-temperature criterion of Theorem~\ref{thm:highT_colored}. 

\item   For the symmetry-breaking statement, we use a tangent-line upper bound on the entropy to obtain an explicit upper bound on $\psi_{\wedge}^{\mathrm{const}}(\beta_2)$. We then compare this upper bound with the value attained by a nonconstant balanced 2-block graphon, with value $p$ on the two diagonal blocks and value $q$ on the two off-diagonal blocks. The values of $p$ and $q$ are chosen based on the results of the numerical experiments in Section~\ref{sec:finite_beta_simulations}, leading to the rigorous lower bound of $3.054$ for the symmetry-breaking phase. 
\end{itemize} 
Combining these two bounds, shows that the induced wedge model is replica symmetric for $|\beta_2|<\frac13$ and symmetry breaking for $\beta_2\geq 3.054$. The intermediate positive-temperature window remains open. The numerical results in Section~\ref{sec:finite_beta_simulations} suggest that the phase transition may occur close to $3.054$. They also indicate that, as $\beta_2$ increases, the optimizing graphon appears to transition from a constant graphon to a 2-block graphon and eventually converges to the complete balanced bipartite graphon (see Figures \ref{fig:wg-section36-centered-competitors} and \ref{fig:wg-section36-centered-competitors-zoom}).

\subsubsection{Rainbow Triangle ERGM}

In this section, we consider the rainbow triangle ERGM variational problem 
\eqref{eq:rbtemperatureTW} with the coefficients corresponding to the monochromatic edge densities set to zero, that is, $\beta_1= \beta_2= \beta_3 =0$. Then \eqref{eq:rbtemperatureTW} simplifies to 
\begin{align}\label{eq:rbsymmetryTW}
\psi_{\mathrm{rb}\triangle}(\beta_4) := \sup_{\bm W\in\wt{\mathcal W}_3} \left\{  \beta_4 t_{\mathrm{rb}\triangle}(\bm W) -\frac12\mathcal I(\bm W) \right\}. 
\end{align}
In this case, $\beta_4 > 0 $ is in the replica symmetry regime, whenever \eqref{eq:rbsymmetryTW} is optimized at a constant 3-color probability graphon $\bm W(x, y) = (p_1, p_2, p_3) \in \Delta_3$, for almost every $(x, y) \in [0, 1]^2$.  
The variational problem \eqref{eq:rbsymmetryTW} restriction to the space of 3-color probability graphons is
\begin{equation*}
\psi_{\mathrm{rb}\triangle}^{\mathrm{const}}(\beta_4)
:=
\sup_{p\in\Delta_3}
\left\{
\beta_4p_1p_2p_3
-\frac12\sum_{i=1}^3p_i\log p_i
\right\}.
\end{equation*}
For $\beta_4\ge0$, both terms are maximized at
$p_1=p_2=p_3=1/3$. Indeed, by the AM--GM inequality, $p_1p_2p_3\le\frac{1}{27}$, with equality if and only if $p_1=p_2=p_3=\frac{1}{3}$, and $-\sum_{i=1}^3 p_i\log p_i\le \log 3$,  
again with equality if and only if $p_1=p_2=p_3=\frac{1}{3}$. Therefore, we obtain, for $
\beta_4\ge0$
\begin{align}\label{eq:optimizerrainbowconstant}
\psi_{\mathrm{rb}\triangle}^{\mathrm{const}}(\beta_4)
=
\frac{\beta_4}{27}+\frac12\log3 . 
\end{align}

Having identified the constant branch, we first identify a rigorous regime in
which it is the unique global maximizer. Note that a direct application of
Theorem~\ref{thm:highT_colored} in this case gives $|\beta_4|< \frac{1}{2}$. However, for the fixed rainbow triangle interaction, the simplex constraint produces
cancellations between the color derivatives and yields the following improved
criterion. The proof is given in Section~\ref{sec:rainbow_highT_sharp_pf}.

\begin{corollary} \label{cor:rainbow_highT_sharp}
Consider the rainbow triangle ERGM with sufficient statistic
\eqref{eq:rainbowT}, where $t_{\mathrm{rb}\triangle}$ is the fixed decorated
homomorphism density in \eqref{eq:tWrainbow}. Suppose $\beta_1,\beta_2,\beta_3\in\mathbb R$ are arbitrary. Then, for $|\beta_4|<1$, the variational problem in Theorem~\ref{thm:variational_formula} has a
unique maximizer. Consequently, this maximizer is a constant $3$-colored
probability graphon.
\end{corollary}

\begin{remark}[Why the model-specific contraction is useful?]
Although the numerical improvement from $\frac{1}{2}$ to $1$ is modest, the argument
illustrates a reusable point: for a chosen colored interaction, estimating its
derivatives before applying general motif-count bounds can exploit cancellations
forced by the simplex constraint $W_1+W_2+W_3=1$ and thereby sharpen the
replica-symmetry criterion.
\end{remark}

To obtain a rigorous symmetry-breaking bound on the other side, we need an
explicit nonconstant family of competitors. Since the positive zero-temperature
limit of the rainbow-triangle ERGM is given by the iterated blowup $3$-colored
probability graphon of the $K_4$ template, a natural initial candidate for
finite-temperature symmetry breaking is the 1-step blowup of the properly
$3$-colored $K_4$ template (as defined in Section~\ref{sec:rt}). However, interestingly the numerical exploration in
Section~\ref{sec:finite_beta_simulations} identifies balanced bipodal graphons
as the first family among those tested to overtake the uniform constant branch
as $\beta_4$ increases. This motivates the following definition: 
Fix a measurable partition $[0,1]=A\sqcup B$ with
$|A|=|B|=\frac{1}{2}$. For $q,s\in[0,1]$, define the balanced bipodal
probability graphon
\begin{equation}
\bm W^{q,s}(x,y)
=
\begin{cases}
\left(q,\frac{1-q}{2},\frac{1-q}{2}\right),
& (x,y)\in(A\times A)\cup(B\times B),\\[1mm]
\left(s,\frac{1-s}{2},\frac{1-s}{2}\right),
& (x,y)\in(A\times B)\cup(B\times A).
\end{cases}
\label{eq:rainbow_balanced_bipodal}
\end{equation}
Permuting the three colors produces analogous bipodal families with the same free energy values. Now, defining $h(u):=-u\log u-(1-u)\log\frac{1-u}{2}$
and
$$\tau(q,s) :=\frac{q^3-2q^2+3qs^2-4qs+2q-2s^2+2s}{16},$$
a direct block computation gives 
\begin{equation}
\mathcal F_{\beta_4}^{\mathrm{rb}\triangle, \mathrm{bi}}(q,s)
:=
\beta_4t_{\mathrm{rb}\triangle}(\bm W^{q,s})
-\frac12\mathcal I(\bm W^{q,s})
=\beta_4\tau(q,s)+\frac14\bigl(h(q)+h(s)\bigr).
\label{eq:rainbow_bipodal_reduction}
\end{equation}
Note that the uniform probability graphon corresponds to $q=s=\frac{1}{3}$. Now, based on the simulations in Section~\ref{sec:finite_beta_simulations} we choose the following  explicit rational parameters $q_0:=\frac{4491}{5000}$, $s_0:=\frac{19}{10000}$, $\Delta_0:=\tau(q_0,s_0)-\frac1{27}$ and set
\begin{equation}
B_0
:=
\frac{\log3-\frac12\bigl(h(q_0)+h(s_0)\bigr)}{2\Delta_0}
=14.0049634433\ldots .
\label{eq:rainbow_bipodal_B0}
\end{equation}
Then exact rational arithmetic gives $\Delta_0
=\frac{4210540070599}{216000000000000}>0$. 
Using \eqref{eq:rainbow_bipodal_reduction},
\eqref{eq:rainbow_bipodal_B0}, and
\eqref{eq:optimizerrainbowconstant}, we therefore obtain, 
\begin{align*}
\psi_{\mathrm{rb}\triangle}(\beta_4)
-\psi_{\mathrm{rb}\triangle}^{\mathrm{const}}(\beta_4) \ge
\beta_4\tau(q_0,s_0)
+\frac14\bigl(h(q_0)+h(s_0)\bigr)
-\left(\frac{\beta_4}{27}+\frac12\log3\right) =\Delta_0(\beta_4-B_0).
\end{align*}
This is strictly positive whenever $\beta_4>B_0$, proving item $(2)$. Combining this with Corollary \ref{cor:rainbow_highT_sharp} leads to the following result:

\begin{proposition}
\label{ppn:rainbow_highT} 
Consider the rainbow triangle ERGM variational problem in \eqref{eq:rbsymmetryTW}. Then the following hold:

\begin{itemize}
\item[$(1)$] For $|\beta_4|<1$,
$\psi_{\mathrm{rb}\triangle}(\beta_4)
=
\psi_{\mathrm{rb}\triangle}^{\mathrm{const}}(\beta_4)$; in fact, the
maximizer is unique and replica symmetric.
\item[$(2)$] For $\beta_4>B_0$, where
$B_0=14.0049634433\ldots$ is the explicit constant defined in
\eqref{eq:rainbow_bipodal_B0},
$\psi_{\mathrm{rb}\triangle}(\beta_4)
>
\psi_{\mathrm{rb}\triangle}^{\mathrm{const}}(\beta_4)$, so the model is in
the symmetry-breaking regime. 
\end{itemize}
\end{proposition}

The above result shows that the rainbow-triangle ERGM is replica symmetric for $|\beta_4|<1$ and is in the symmetry-breaking regime for $\beta_4>B_0=14.0049634433\ldots$. The intermediate positive-temperature window remains open. As mentioned above, the symmetry-breaking bound obtained from the 1-step blowup of the properly $3$-colored $K_4$ template is slightly weaker than the bound obtained from the bipodal family described above, as shown in the next remark. To close the gap between the replica-symmetric and symmetry-breaking bounds, it may be interesting to investigate how bipodal and more general multipodal structures compete with finite-step blowups of the properly $3$-colored $K_4$ template, as $\beta_4$ varies.

\begin{remark}
(Symmetry breaking window from the 1-step blowup) Consider the empirical $3$-colored probability graphon
$\bm W^{(1)}$ associated with the properly $3$-colored $K_4$ template, as
defined in Section~\ref{sec:rt}. From Remark \ref{remark:trbW} we know that $t_{\mathrm{rb}\triangle}(\bm W^{(1)}) = \frac1{16} + \frac1{16}\cdot\frac1{27} = \frac7{108}$. Further, it can be easily checked that $\mathcal I(\bm W^{(1)}) = -\frac14\log3$. Hence, evaluating the variational functional at $\bm W^{(1)}$ gives 
\[
\psi_{\mathrm{rb}\triangle}(\beta_4)
\ge
\beta_4
t_{\mathrm{rb}\triangle}(\bm W^{(1)})
-
\frac12\mathcal I(\bm W^{(1)})
=
\frac{7\beta_4}{108}
+
\frac18\log3.
\]
Then recalling \eqref{eq:optimizerrainbowconstant} one has, 
\begin{align*}
\psi_{\mathrm{rb}\triangle}(\beta_4) - \psi_{\mathrm{rb}\triangle}^{\mathrm{const}}(\beta_4) \ge \left(\frac{7\beta_4}{108} + \frac18\log3 \right) - \left( \frac{\beta_4}{27} + \frac12\log3
\right) =
\frac{\beta_4}{36} - \frac38\log3 > 0  , 
\end{align*} 
when $\beta_4>\frac{27}{2}\log3 \approx 14.8313$, which is slightly worse than the bound obtained in Proposition \ref{ppn:rainbow_highT} (2). 
\end{remark}

\subsection{Numerical Phase Diagram and Simulations}
\label{sec:finite_beta_simulations}

In this section, we provide a numerical analysis of the variational problems for the induced wedge and rainbow triangle ERGMs and compare the results with the rigorous bounds obtained above.

\subsubsection{Induced Wedge ERGM}

For the induced-wedge model, we set $\beta_1=0$. Then the limiting free energy functional is given by: 
\[
 \F^{\wedge}_{\beta_2}(W)
 :=\beta_2t_{\mathrm{ind}}(K_{1,2},W)
   -\frac12\mathcal I(W).
\]
To understand the replica symmetry/symmetry breaking phases we consider the following three candidates. Throughout, we fix $\beta_2 > 0$. 

\begin{itemize}

\item {\it The best constant graphon}: This is the replica symmetric solution that is obtained by maximizing $\F^{\wedge}_{\beta_2}(W)$ over constant graphons. We denote this optimized value by $\psi_{\wedge}^{\mathrm{const}}(\beta_2)$, as in \eqref{eq:wedgeconstant}.

\item {\it The best non-constant balanced 2-block graphon}: A natural competitor for the replica symmetric solution above, is a graphon with 2-blocks that assigns different weights within and between blocks. To obtain the optimal graphon in this class, partition $[0,1]=A\sqcup B$ with $|A|=|B|=1/2$ and define
\[
 W^{p,q}(x,y)=
 \begin{cases}
 p,&(x,y)\in(A\times A)\cup(B\times B),\\
 q,&(x,y)\in(A\times B)\cup(B\times A).
 \end{cases}
\]
Thus, $p$ is the within-block edge probability and $q$ is the between-block edge probability. A direct computation shows that the free energy functional is given by: 
\[
 F^{\wedge}_{\beta_2}(p,q)
 :=\frac{\beta_2}{4}
 \bigl(p^2(1-p)+q^2(1-p)+2pq(1-q)\bigr) 
 -\frac14\bigl(I(p)+I(q)\bigr), 
\]
where $I(c)$ is as defined  in \eqref{eq:wedgeconstant}. This family contains every constant graphon when $p=q$ and approaches the
balanced complete bipartite graphon $W_{\mathrm{bip}}$ (recall \eqref{eq:bipartiteW}) as
$(p,q)\to(0,1)$. The optimized value within this class is denoted by $\psi_{\wedge}^{\mathrm{2-block}}(\beta_2) = \max_{p, q \in [0, 1]} F^{\wedge}_{\beta_2}(p,q)$. 

\item {\it Complete bipartite graphon}: We also consider the free energy functional for the balanced complete bipartite graphon $W_{\mathrm{bip}}$, which is the positive zero-temperature limit in the induced-wedge ERGM. In this case, $\F^{\wedge}_{\beta_2}(W_{\mathrm{bip}})=\frac{\beta_2}{4}$.

\end{itemize}

Figure \ref{fig:wg-section36-centered-competitors} shows plots of the free energy functionals for the above three candidates in the window $\beta_2 \in [0, 10]$. To better visualize the small free energy advantages, we center the plot relative to  
\begin{align}\label{eq:wedgemaximum}
\psi_{\wedge}^{\mathrm{2-opt}}(\beta_2) : =\max\{ \psi_{\wedge}^{\mathrm{const}}(\beta_2), \psi_{\wedge}^{\mathrm{2-block}}(\beta_2)\},
\end{align} 
which is the maximum of the free energies for the best constant graphon and the best non-constant balanced 2-block graphon obtained by the numerical search.\footnote{
Note that the balanced two-block family contains every constant graphon by taking $p=q$, $\psi_{\wedge}^{2\text{-block}}(\beta_2)
\geq \psi_{\wedge}^{\mathrm{const}}(\beta_2)$. Hence, $\psi_{\wedge}^{2\text{-opt}}(\beta_2) =
\psi_{\wedge}^{2\text{-block}}(\beta_2)$. However, a numerical optimization procedure for
$\psi_{\wedge}^{2\text{-block}}(\beta_2)$ may converge to a nonconstant local
optimum whose value is smaller than
$\psi_{\wedge}^{\mathrm{const}}(\beta_2)$. Thus, with a slight abuse of
notation, in the numerical discussion we use
$\psi_{\wedge}^{2\text{-block}}(\beta_2)$ to denote the value of the
nonconstant local optimum identified by the numerical search. In particular, the short orange
segment in Figure~\ref{fig:wg-section36-centered-competitors} to the left of the vertical line
$\beta_2\approx 3.0528$ corresponds to points where the numerical search finds nonconstant local optimum, whose
free-energy value is smaller than that of the best constant graphon. (The orange curve is not shown for smaller values of $\beta_2$, because here the numerical search finds only constant stationary solutions.) Consequently, although the maximum in \eqref{eq:wedgemaximum} is mathematically
redundant, it is useful numerically for comparing the constant solution with
the nonconstant branch, and we interpret it in this sense throughout the
numerical discussion. }
From  the plot we observe that the best non-constant balanced 2-block beats the best constant value for the first time at 
$\beta_{\mathrm{2-block}}=3.0527543267\ldots$ (marked by the dotted vertical line). This is compatible with the rigorous symmetry-breaking statement for $\beta_2\geq3.054$ in
Proposition~\ref{ppn:finite_beta_wedge_symmetry_breaking}. In fact, this experiment enabled us to identify a specific 2-block graphon $W^{p, q}$ that we used in the proof of Proposition \ref{ppn:finite_beta_wedge_symmetry_breaking} to establish symmetry breaking. Another point to note is that the free energy of the balanced complete bipartite graphon approaches the horizontal reference line as $\beta_2$ increases, which also aligns with positive zero-temperature limit obtained in Theorem \ref{thm:wedge-beta-plus-infty}. A zoomed in view of 
Figure \ref{fig:wg-section36-centered-competitors} around $\beta_{\mathrm{2-block}}=3.0528$ is shown in Figure \ref{fig:wg-section36-centered-competitors-zoom}, where the differences between the different curves become more visually apparent.

\begin{figure}[h!]
  \centering
  \includegraphics[width=0.75\textwidth]
  {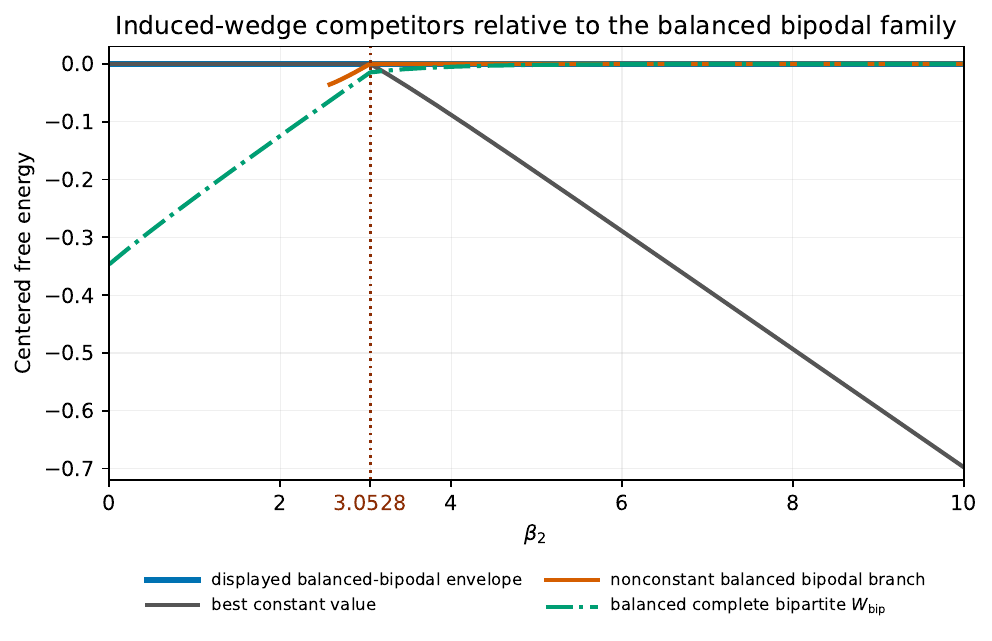}
  \caption{ \small{ Centered free energies for the induced-wedge ERGM on
  $0\leq\beta_2\leq10$, relative to the displayed balanced-bipodal envelope
  $\psi_{\wedge}^{\mathrm{2-opt}}(\beta_2)$ (the blue zero line). The black curve is
  $\psi_{\wedge}^{\mathrm{const}}(\beta_2)-\psi_{\wedge}^{\mathrm{2-opt}}(\beta_2)$,
  and the orange curve is
  $\psi_{\wedge}^{\mathrm{2-block}}(\beta_2)-\psi_{\wedge}^{\mathrm{2-opt}}(\beta_2)$;
  the latter is shown only when the numerical search finds a genuinely nonconstant
  balanced-bipodal local maximum. The dash-dotted green curve is the balanced
  complete bipartite value
  $\frac{\beta_2}{4}-\psi_{\wedge}^{\mathrm{2-opt}}(\beta_2)$.
  The dotted vertical line marks
  $\beta_{\mathrm{2-block}}\approx3.0528$, where the nonconstant branch first
  exceeds the best constant value. } }  
  \label{fig:wg-section36-centered-competitors}
\end{figure}

\begin{figure}[p]
  \centering
  \includegraphics[width=0.75\textwidth]
  {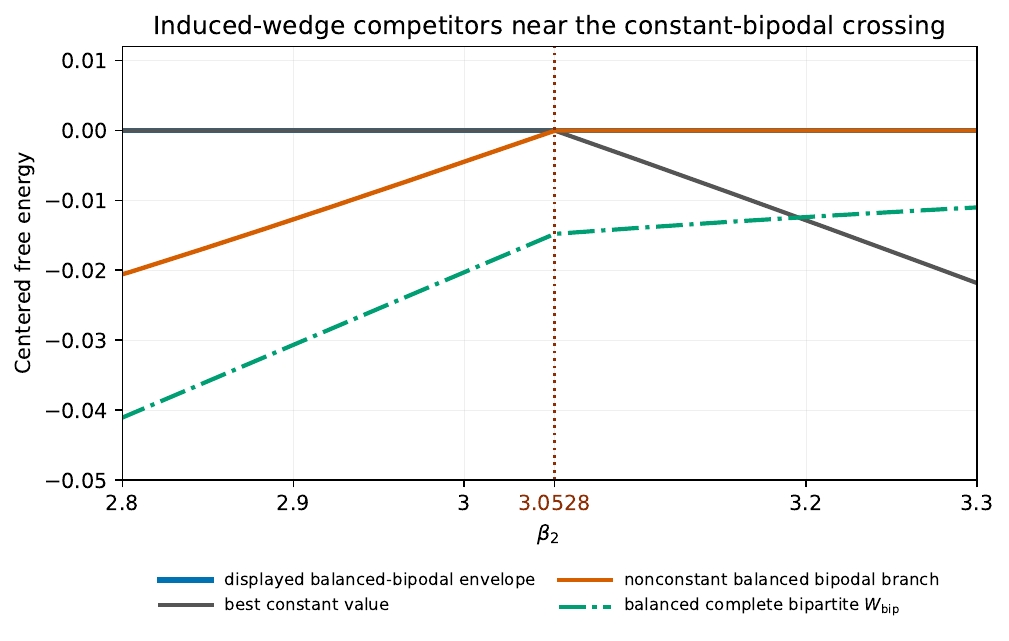}
  \caption{\small{ A zoomed-in view of
  Figure~\ref{fig:wg-section36-centered-competitors} near the restricted
  constant--balanced-bipodal equal-value point
  $\beta_{\mathrm{2-block}}\approx3.0528$. The curves and centering use the
  same conventions as in the full-range comparison. }}
  \label{fig:wg-section36-centered-competitors-zoom}
\end{figure}

\subsubsection{Rainbow Triangle ERGM}

\begin{figure}[p]
  \centering
  \includegraphics[width=0.75\textwidth]
  {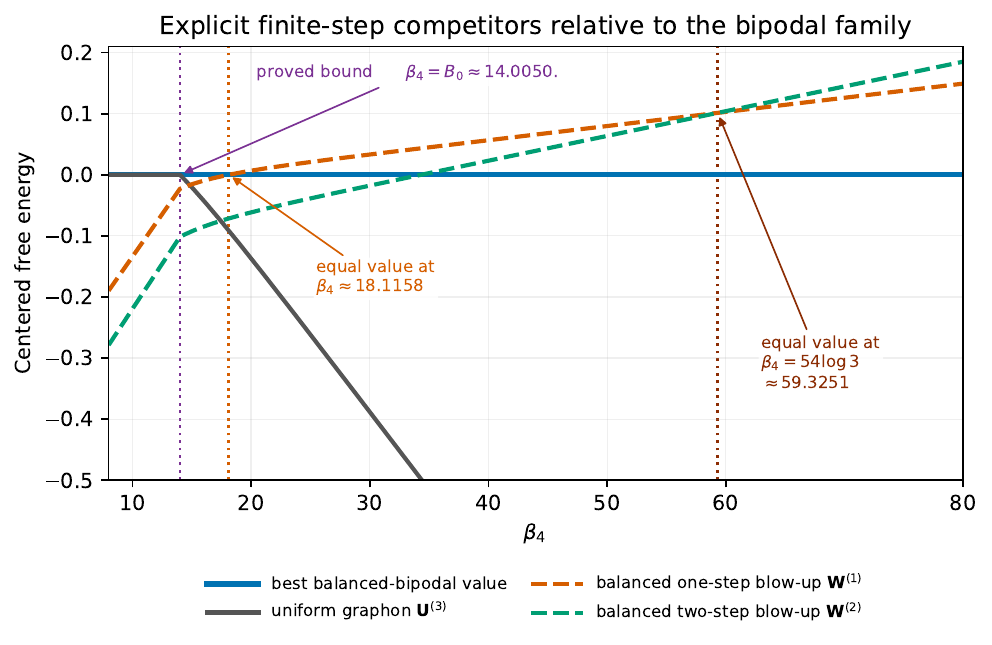}
  \caption{\small{ Centered free energies for the rainbow-triangle ERGM relative
  to the best value in the balanced-bipodal family on $\beta_4\in[8,80]$.
  The blue zero line is $\psi^{\mathrm{rb}\triangle,\mathrm{bi}}_{\beta_4}$.
  The black curve is
  $\F_{\beta_4}^{\mathrm{rb}\triangle}(\bm U^{(3)})-
  \psi^{\mathrm{rb}\triangle,\mathrm{bi}}_{\beta_4}$, the dashed orange curve is
  $\F_{\beta_4}^{\mathrm{rb}\triangle}(\bm W^{(1)})-
  \psi^{\mathrm{rb}\triangle,\mathrm{bi}}_{\beta_4}$, and the dashed green curve is
  $\F_{\beta_4}^{\mathrm{rb}\triangle}(\bm W^{(2)})-
  \psi^{\mathrm{rb}\triangle,\mathrm{bi}}_{\beta_4}$.
  From left to right, the dotted vertical lines mark the proved
  symmetry-breaking bound $B_0\approx14.0050$, the numerical equal-value point
  between $\bm W^{(1)}$ and the best balanced-bipodal value at
  $\beta_4\approx18.1158$, and the exact equal-value point between $\bm W^{(2)}$
  and $\bm W^{(1)}$ at $\beta_4=54\log 3\approx59.3251$. } }
  \label{fig:rb-section36-centered-competitors}
\end{figure}

For the rainbow triangle ERGM, we set $\beta_1=\beta_2=\beta_3=0$ and work with the variational problem in \eqref{eq:rbsymmetryTW}, whose free energy functional is given by 
\begin{align}\label{eq:trianglesimulation}
 \F_{\beta_4}^{\mathrm{rb}\triangle}(\bm W)
 :=\beta_4t_{\mathrm{rb}\triangle}(\bm W)
   -\frac12\mathcal I(\bm W). 
   \end{align}
In this case we consider the following candidates: 
First, we maximize the functional \eqref{eq:trianglesimulation} over the balanced bipodal family $\bm W^{q,s}$ from
\eqref{eq:rainbow_balanced_bipodal}, using the reduced functional in
\eqref{eq:rainbow_bipodal_reduction}. We denote the optimum value in this case by 
$$\psi^{\mathrm{rb}\triangle, \mathrm{bi}}_{\beta_4}:=\max_{(q,s)\in[0,1]^2}\F_{\beta_4}^{\mathrm{rb}\triangle, \mathrm{bi}}(q,s),$$
We then compare this benchmark with the free energy of the uniform graphon $\bm U^{(3)}$, and the free energies of the fixed balanced 1-step and 2-step blowups $\bm W^{(1)}$ and $\bm W^{(2)}$,  
defined in Section~\ref{sec:rt}. From the plot we observe that the best balanced bipodal first exceeds the free energy of the
uniform graphon $\bm U^{(3)}$ at $\beta_{\mathrm{bi}}=14.0049633052\ldots$. This should be
distinguished from the rigorous constant
$B_0=14.0049634433\ldots$ in \eqref{eq:rainbow_bipodal_B0}: the explicit
bipodal competitor used in Proposition~\ref{ppn:rainbow_highT} proves
symmetry breaking for every $\beta_4>B_0$, but does not prove that the
unrestricted maximizer is bipodal or that the global transition occurs at
$B_0$. Figure~\ref{fig:rb-section36-centered-competitors} also shows the point where the free energy of the 1-step blowup $\bm W^{(1)}$ first exceeds the best bipodal value, at $\beta_4\approx18.1158$, and the point where the 2-step blowup $\bm W^{(2)}$ first exceeds $\bm W^{(1)}$, at $\beta_4=54\log 3\approx59.3251$. These comparisons are further illustrated in Figure~\ref{fig:rb-section36-transition-profiles}, where we zoom in on the curves near the respective thresholds.

\begin{figure}[ht!] 
  \centering
  \includegraphics[width=0.88\textwidth]
  {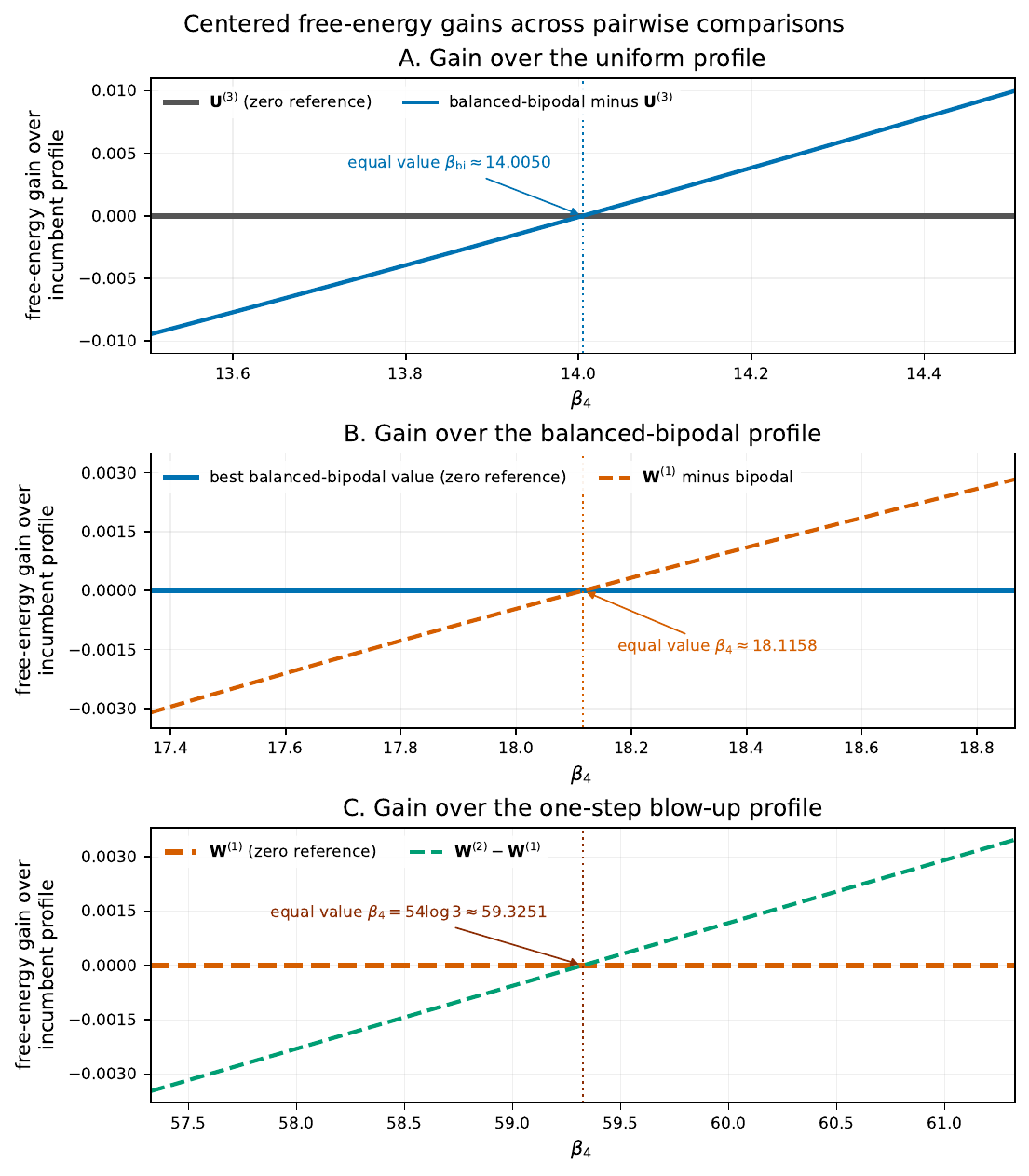}
  \caption{\small{Pairwise comparisons of the centered free energy profiles for the rainbow triangle ERGM. In each panel the
  incumbent profile is the highlighted zero line. In the first panel, the blue curve is   $\psi^{\mathrm{rb}\triangle, \mathrm{bi}}_{\beta_4}-\F_{\beta_4}^{\mathrm{rb}\triangle}(\bm U^{(3)})$ zoomed in near the numerical value $\beta_{\mathrm{bi}}\approx14.0050$. In the second panel, 
the dashed orange curve is $\F_{\beta_4}^{\mathrm{rb}\triangle}(\bm W^{(1)}) - \psi^{\mathrm{rb}\triangle, \mathrm{bi}}_{\beta_4}$ zoomed in near $\beta_4\approx18.1158$. In the third panel, the dashed green curve is 
  $\F_{\beta_4}^{\mathrm{rb}\triangle}(\bm W^{(2)}) - \F_{\beta_4}^{\mathrm{rb}\triangle}(\bm W^{(1)}) $ zoomed in near $\beta_4 = 54\log 3\approx59.3251$. }}
  \label{fig:rb-section36-transition-profiles}
\end{figure}

\clearpage
\section{Proofs from Section \ref{sec:Zvariationalpf} }

\subsection{Proof of Theorem \ref{thm:variational_formula}} 
\label{sec:Znpf}

Denote by ${\bm \nu} =  \bm{U}^{(k)}:=(\frac1k,\dots,\frac1k)$ the uniform probability vector. Then $\mu_{n,\bm{u}}$ is the law of the random coloring of $K_n$ whose edges are colored independently with the uniform distribution $\bm{u}$, that is, 
$$\mu_{n,\bm{u}}(\mathscr{X}) = \frac{1}{k^{\binom{n}{2}}}, $$
for each $\mathscr{X} \in \mathcal C_n^k$. As in Section \ref{sec:ldp-probability-graphons}, denote by $\wt{\bm \mu }_{n,\bm{u}}$ the lift of $\mu_{n,\bm{u}}$ to the space of probability measures on $\wt{\cW}_k$. Then 
\begin{align*}
Z_n = \sum_{\mathscr{X} \in \mathcal C^k_n} \exp\left(n^2T(\wt{\bm W}^{\mathscr{X}})\right) & = k^{\binom{n}{2}}\int_{\mathcal C^k_n} \exp\left(n^2T(\wt{\bm W}^{\mathscr{X}})\right)\,\mu_{n,\bm{u}}(\mathrm{d} \mathscr{X})  \nonumber \\ 
& = k^{\binom{n}{2}}\int_{\wt{\cW}_k} \exp\left(n^2T(\wt{\bm{W}})\right)\,\wt{\bm \mu }_{n,\bm{u}}( \mathrm{d} \wt{\bm{W}}).
\end{align*}
By Theorem \ref{THM:LDP_prob_graphons}, $\wt{\bm \mu }_{n,\bm{u}}$ satisfies a LDP with speed $a_n = \frac{n^2}{2}$ and good rate function $\cI_{\bm{u}}$. From \eqref{eq:IcolorW} and \eqref{eq:IW} note that 
\begin{equation}
\label{eq:IutoI}
\mathcal I_{\bm{u}}(\bm W)=\log k+\mathcal I(\bm W),
\end{equation}
for $W\in\mathcal W_k$. Therefore, applying Theorem \ref{thm::EVaradhanLemma}, with $\cX = \wt{\cW}_k$, $F = T$, and $Q_n = \wt{\bm \mu }_{n,\bm{u}}$, and using \eqref{eq:IutoI}, 
\begin{align*}
\lim_{n\to\infty} \psi_n &= \lim_{n\to\infty}\frac{1}{n^2}\log Z_n\\
&=\lim_{n\to\infty} \frac{1}{n^2}\log\left(k^{\binom{n}{2}}\int_{\wt{\cW}_k}\exp\left(n^2T(\wt{\bm{W}})\right)\,\wt{\bm \mu }_{n,\bm{u}}(d\wt{\bm{W}})\right)\\
&=\frac{1}{2}\lim_{n\to\infty} \frac{1}{a_n}\left(\log \left(k^{\binom{n}{2}}\right) + \log\left(\int_{\cX} \exp(a_n(2F(x)))\,Q_n(\mathrm{d}x)\right)\right)\\
&= \lim_{n\to\infty} \frac{1}{n^2}\binom{n}{2}\log(k)+\frac{1}{2}\sup_{x \in \cX} \left(2F(x) - \cI_{\bm{u}}(x)\right)\\
&= \frac{1}{2}\log(k) + \sup_{\wt{\bm{W}} \in \wt{\cW}_k} \left(T(\wt{\bm{W}}) - \frac12\cI_{\bm{u}}(\wt{\bm{W}})\right)\\
&= \sup_{\wt{\bm{W}} \in \wt{\cW}_k} \left(T(\wt{\bm{W}}) - \frac12\cI(\wt{\bm{W}})\right)  ,   
\end{align*}
as desired. \hfill $\Box$

\subsection{Proof of Corollary \ref{cor::expconc}}  
\label{sec:concentrationpf}

Let $\tilde A \subset \wt{\cW}_k$ be any Borel measurable set. Define $A_n = \{ \mathscr{X} \in \cC^k_n: \wt{\bm W}^{\mathscr{X}} \in \tilde A\}$. Recall that $\mu_{n,\bm{u}}$ is the uniform probability distribution on $\cC^k_n$. Also, recall the definition of $\P_{n}$ from \eqref{eq:TW} and the definitions of probability measures  $\tilde\mu_{n,\bm{u}}$ and $\tilde{\mathbb{P}}_n$, which are the lifts of   $\mu_{n,\bm{u}}$ and $\P_{n}$ to the space of probability graphons $\wt{\cW}_k$,  using the map \eqref{eq:mappingGraphsGraphons}, respectively. Then, 
\begin{align*}
\wt{\P}_{n}(\tilde A) = \P_{n}(A_n) & = \frac{\sum_{\mathscr{X}: \wt{\bm W}^{\mathscr{X}} \in \tilde A}\exp\left(n^2T(\wt{\bm W}^{\mathscr{X}})\right)}{\sum_{\mathscr{X}_0\in \cC^k_n}\exp\left(n^2T(\wt{\bm{W}}^{\mathscr{X}_0})\right)} \\
&= \frac{|\cC^k_n|\int_{A_n}\exp\left(n^2T(\wt{\bm W}^{\mathscr{X}})\right)\mu_{n,\bm{u}}( \mathrm{d} \mathscr{X})}{|\cC^k_n|\int_{\cC^k_n}\exp\left(n^2T(\wt{\bm W}^{\mathscr{X}})\right)\mu_{n,\bm{u}}(\mathrm{d} \mathscr{X})}   \\
&= \frac{\int_{\tilde A} \exp\left(n^2T(\wt{\bm{W}})\right)\wt{\bm \mu }_{n,\bm{u}}(\mathrm{d}\wt{\bm{W}})}{\int_{\wt{\cW}_k} \exp\left(n^2T(\wt{\bm{W}})\right)\wt{\bm \mu }_{n,\bm{u}}(\mathrm{d}\wt{\bm{W}})}.
\end{align*}
By Theorem \ref{THM:LDP_prob_graphons}, $(\wt{\bm \mu }_{n,\bm{u}})_{n \in \N}$ satisfies a large deviations principle with speed $\frac{n^2}{2}$ and good rate function $\cI_{\bm{u}}$. Recall that, by assumption, $T$ is bounded and continuous. Hence, applying Theorem \ref{thm:TiltedLDP}, with $\cX = \wt{\cW}_k$, $Q_n = \wt{\bm \mu }_{n,\bm{u}}$, $a_n = \frac{n^2}{2}$, $\cI = \cI_{\bm{u}}$, $F = 2T$, gives that $\wt{\P}_{n} := Q_{n, F}$ satisfies a LDP with speed $\frac{n^2}{2}$ and rate function 
\begin{align*}
\mathcal I_T(\wt{\bm{W}} ) &= \cI_{\bm{u}}(\wt{\bm{W}} ) - 2T(\wt{\bm{W}}) - \inf_{\wt{\bm{W}} \in \wt{\cW}_k}\left(\cI_{\bm{u}}(\wt{\bm{W}} ) - 2T(\wt{\bm{W}})\right) . 
\end{align*}
By \eqref{eq:IutoI},
\begin{align*}
\argmin_{\wt{\bm{W}} \in \wt{\cW}_k} \mathcal I_T(\wt{\bm{W}} ) &= \argmin_{\wt{\bm{W}} \in \wt{\cW}_k}\left( \cI_{\bm{u}}(\wt{\bm{W}} ) - 2T(\wt{\bm{W}}) - \inf_{\wt{\bm{W}} \in \wt{\cW}_k}\left(\cI_{\bm{u}}(\wt{\bm{W}} ) - 2T(\wt{\bm{W}})\right)\right)\\
&= \argmin_{\wt{\bm{W}} \in \wt{\cW}_k}\left( \cI_{\bm{u}}(\wt{\bm{W}} ) - \log k - 2T(\wt{\bm{W}})\right)\\
&= \argmax_{\wt{\bm{W}} \in \wt{\cW}_k}\left(T(\wt{\bm{W}}) - \frac{1}{2}\cI(\wt{\bm{W}} )\right)\\
&= \wt{\cS}_T.
\end{align*}
Therefore, the set $\{\wt{\bm{W}}  \in \wt{\cW}_k: \delta_{\square}(\wt{\bm{W}} ,\wt{\cS}_T)\geq \delta\}$, which is a closed subset of $\wt{\cW}_k$, does not contain any minimizers of $\mathcal I_T$.
Then by Theorem \ref{thm:TiltedLDP}, there exists a $C := C(\delta) > 0$ and an integer $N : = N(\delta)$ such that for $n \geq N$,
$$\wt{\P}_{n}(\delta_{\square}(\wt{\bm{W}} ,\wt{\cS}_T)\geq \delta) = \P_n\left(\delta_{\square}(\wt{\bm{W}}^{\mathscr{X}},\wt{\cS}_T) \geq \delta\right) \leq e^{-Cn^2} , $$
as desired.  \hfill $\Box$

\section{Proofs from Section \ref{sec:starmonotone}}

\subsection{Proof of Theorem \ref{thm:colordegree}}
\label{sec:starpf}

Note that $I_k$ is convex on $\Delta_k$. Hence, applying multivariate Jensen's inequality \cite[Theorem 10.2.6]{dudleyRealAnalysisProbability2002} gives: 
For each fixed $x \in [0, 1]$,
$$ \int_0^1 I_k(\bm W(x,y))\, \mathrm{d} y \ge I_k\left(\int_0^1 \bm W(x,y)\, \mathrm{d}y \right)
= I_k( \bm d_{\bm W} (x)), $$
with equality only when $y \rightarrow \bm W(x,y)$ is almost everywhere constant. Now, integrating over $x$ and recalling \eqref{eq:IW} yields, 
\begin{equation}
\label{eq:entropy_jensen_general_star}
\mathcal I(\bm W)\ge \int_0^1 I_k( \bm d_{\bm W}(x))\, \mathrm{d}x . 
\end{equation}
Combining \eqref{eq:entropy_jensen_general_star} and \eqref{eq:P_general_star} gives, 
\begin{align}\label{eq:starconvex}
T(\bm W)-\frac12\mathcal I(\bm W) \le \int_0^1 \left(  P(\bm d_{\bm W} (x))\, - \frac12 I_k( \bm d_{\bm W} (x)) \right) \, \mathrm{d} x \le \sup_{\bm \mu \in\Delta_k} \left \{P(\bm \mu )-\frac12 I_k(\bm \mu ) \right \} ,  
\end{align}
since $\bm d_{\bm W} (x) \in \Delta_k$. Taking the supremum over $\bm W\in\wt{\mathcal W}_k$ gives, 
$$\sup_{  \bm W \in \mathcal W_k  }
\left(T( \bm W )- \frac12\mathcal I( \bm W )\right) \leq \sup_{\bm \mu \in\Delta_k} \left \{P(\bm \mu )-\frac12 I_k(\bm \mu ) \right \} . $$ 
To show the reverse direction, let $\bm \mu ^\star$ be a maximizer of $\bm \mu \mapsto P(\bm \mu )-\frac12 I_k(\bm \mu )$ on $\Delta_k$. Observe that for the constant $k$-color probability graphon $\bm W^\star\equiv \bm \mu ^\star$, we have $T(\bm W^\star)=P(\bm \mu ^\star)$ and $\mathcal I(\bm W^\star)=I_k(\bm \mu ^\star)$. Hence, 
$$\sup_{  \bm W \in \mathcal W_k  } \left(T( \bm W )-\frac12\mathcal I( \bm W )\right) \geq 
T( \bm W^\star )-\frac12\mathcal I( \bm W^\star ) = \sup_{\bm \mu \in\Delta_k} \left \{P(\bm \mu )-\frac12 I_k(\bm \mu ) \right \} . $$ 
Combining the above inequalities and recalling Theorem \ref{thm:variational_formula} proves the result in \eqref{eq:general_star_variational_reduction}.

Now, suppose $\bm W$ is a maximizer of the graphon variational problem in Theorem \ref{thm:variational_formula}.  
Then equality must hold in \eqref{eq:starconvex} throughout. This implies, from the equality condition of Jensen's inequality, that the function $y \rightarrow \bm W(x,y)$ must be almost everywhere constant, for almost every $x \in [0, 1]$. Thus, $\bm W(x,y)= \bm d_{\bm W}(x)$ almost everywhere. Since $\bm W$ is symmetric,
\[
\bm d_{\bm W}(x) = \bm W(x,y) = \bm W(y,x) = \bm d_{\bm W}(y)
\]
for almost every $(x,y)$. Hence $\bm d_{\bm W}$ is almost everywhere constant, and so is $\bm W$. Also, equality in the final ``$\le$'' of \eqref{eq:starconvex} implies that $\bm d_{\bm W}(x)$ is almost everywhere equal to a maximizer of $P-\frac12 I_k$ on $\Delta_k$. This shows that $\bm W = \bm \mu ^*$, where $\bm \mu ^* \in \Delta_k$ is a maximizer of  \eqref{eq:general_star_variational_reduction}. \hfill $\Box$

\subsection{ Proof of Theorem \ref{thm:explicit_solvable_monotone} } 
\label{sec:monotonepf}

Note that, for each $\bm \alpha\in\mathbb R_{\geq 0}^k$, the function $\langle\bm\alpha,\bm W\rangle$ is non-negative,
symmetric, and bounded on $[0,1]^2$. For $s\in[R]$, set $h_s:=\langle\bm\alpha_s,\bm W\rangle$. Then  applying H\"older's inequality (Lemma \ref{lem:holder_equality}) on the probability space
$[0,1]^{V(H_s)}$ to the $|E(H_s)|$ functions $(x_v)_{v\in V(H_s)} \longmapsto
h_s(x_u,x_v)$, for $(u,v)\in E(H_s)$, with all exponents equal to $|E(H_s)|$, gives
\begin{align}\label{eq:holder_correct}
t(H_s,h_s) & \leq \prod_{(u,v) \in E(H_s)} \left( \int_{[0,1]^{V(H_s)}}h_s(x_u,x_v)^{|E(H_s)|} \prod_{w\in V(H_s)}\mathrm{d}x_w \right)^{\frac{1}{|E(H_s)|}}  \nonumber \\ 
& = \int_{[0,1]^2}h_s(x,y)^{|E(H_s)|}\,\mathrm{d}x\,\mathrm{d}y  .  
\end{align}
We now derive the corresponding equality condition. If $|E(H_s)|\geq2$, then
equality in \eqref{eq:holder_correct} holds if and only if $h_s$ is
constant almost everywhere. Indeed, if $\int_{[0,1]^2}h_s(x,y)^{|E(H_s)|} = 0$, then $h_s=0$ almost everywhere. Otherwise, the equality condition in H\"older's inequality (Lemma \ref{lem:holder_equality}) implies that, for any
two distinct edges $e, e'\in E(H_s)$, the corresponding functions satisfy, 
\[
h_s(x_e)^{|E(H_s)|}=h_s(x_{e'})^{|E(H_s)|}  ,  \quad\text{ almost everywhere on }[0,1]^{V(H_s)}.
\]
Since $h_s$ is non-negative, this gives $h_s(x_e)=h_s(x_{e'})$, almost everywhere. If $e$ and $e'$ share a vertex, relabeling their endpoints
gives $h_s(x_1,x_2)=h_s(x_1,x_3)$ for almost every $(x_1,x_2,x_3)$, which means $h_s(x,y)=g(x)$ almost everywhere for some measurable function $g$.
By symmetry, $g(x)=h_s(x,y)=h_s(y,x)=g(y)$, for almost every $(x, y)$, and hence, $g$, and therefore $h_s$, is constant almost everywhere. If
$e$ and $e'$ are disjoint, then $h_s(x_1,x_2)=h_s(x_3,x_4)$ for almost every $(x_1,x_2,x_3,x_4)$, which directly implies that $h_s$ is constant almost everywhere. The converse is immediate.

Summing \eqref{eq:holder_correct} over $s\in[R]$ with the non-negative
coefficients $\gamma_1,\ldots,\gamma_R$ and adding the linear term for
$K_2$, we obtain
\begin{align}
T(\bm W)
&=
t(K_2,\langle\bm c,\bm W\rangle)
+
\sum_{s=1}^R
\gamma_s\,t(H_s,\langle\bm\alpha_s,\bm W\rangle)
\notag\\
&\leq
\int_{[0,1]^2}
\left[
\langle\bm c,\bm W\rangle(x,y)
+
\sum_{s=1}^R
\gamma_s
\bigl(\langle\bm\alpha_s,\bm W\rangle(x,y)\bigr)^{|E(H_s)|}
\right]
\mathrm{d}x\,\mathrm{d}y.
\label{eq:T_upper_bound_pointwise}
\end{align}
Combining \eqref{eq:T_upper_bound_pointwise} with \eqref{eq:IW} gives 
\begin{equation*}
\mathcal F(\bm W)
:=
T(\bm W)-\frac12\mathcal I(\bm W)
\leq
\int_{[0,1]^2}\Phi(\bm W(x,y))\,\mathrm{d}x\,\mathrm{d}y,
\end{equation*}
where $\Phi:\Delta_k\to\mathbb R$ is defined by
\begin{equation*}
\Phi(\bm\mu)
:=
\langle\bm c,\bm\mu\rangle
+
\sum_{s=1}^R
\gamma_s\langle\bm\alpha_s,\bm\mu\rangle^{|E(H_s)|}
-\frac12 I_k(\bm\mu).
\end{equation*}
Since $\Phi(\bm W(x,y)) \leq \sup_{\bm\mu\in\Delta_k}\Phi(\bm\mu)$, for almost every $(x,y)\in[0,1]^2$, Theorem~\ref{thm:variational_formula}
gives
\[
\lim_{n\to\infty}\psi_n
\leq
\sup_{\bm\mu\in\Delta_k}\Phi(\bm\mu).
\]
For the reverse inequality, fix $\bm\mu\in\Delta_k$ and consider the
constant $k$-color probability graphon $\bm W\equiv\bm\mu$. Then
$\langle\bm\alpha_s,\bm W\rangle
\equiv\langle\bm\alpha_s,\bm\mu\rangle$, so equality holds in
\eqref{eq:holder_correct} for every $s$. Taking the supremum over
$\bm\mu\in\Delta_k$ proves \eqref{eq:explicit_solvable_formula}.

It remains to prove that every maximizer of the graphon variational
problem is constant. We first eliminate trivial terms. Terms for which
$\gamma_s=0$ may be deleted. If $|E(H_s)|=1$, then
the term can be absorbed into the linear term. If no nonlinear terms remain,
the result follows from Theorem~\ref{thm:colordegree}. Thus, without
loss of generality, we may assume that $R\geq1$, $\gamma_s>0$ and $|E(H_s)|\geq2$, for every $s\in[R]$. Now, suppose $\bm W$ is a maximizer of $\mathcal F$. By
\eqref{eq:explicit_solvable_formula},
\[
\mathcal F(\bm W)
=
\sup_{\bm\mu\in\Delta_k}\Phi(\bm\mu).
\]
On the other hand,
\[
\mathcal F(\bm W)
\leq
\int_{[0,1]^2}\Phi(\bm W(x,y))\,\mathrm{d}x\,\mathrm{d}y
\leq
\sup_{\bm\mu\in\Delta_k}\Phi(\bm\mu).
\]
Consequently, equality holds in both inequalities. Moreover,
\[
\int_{[0,1]^2}\Phi(\bm W(x,y))\,\mathrm{d}x\,\mathrm{d}y
-\mathcal F(\bm W)
=
\sum_{s=1}^R\gamma_s
\left(
\int_{[0,1]^2}h_s(x,y)^{|E(H_s)|}\,\mathrm{d}x\,\mathrm{d}y
-t(H_s,h_s)
\right).
\]
Every summand on the right-hand side is non-negative, and
$\gamma_s>0$. Hence equality must hold in \eqref{eq:holder_correct} for
every $s\in[R]$. By the equality condition established above, there
exist constants $q_1,\ldots,q_R\geq0$ such that $\langle\bm\alpha_s,\bm W\rangle(x,y)=q_s$, for almost every $(x,y)\in[0,1]^2$ and every $s\in[R]$. Equality in the second inequality also implies $\Phi(\bm W(x,y)) = \sup_{\bm\mu\in\Delta_k}\Phi(\bm\mu)$, for almost every $(x,y)\in[0,1]^2$. Define
\[
S
:=
\left\{
\bm\mu\in\Delta_k:
\langle\bm\alpha_s,\bm\mu\rangle=q_s
\text{ for every }s\in[R]
\right\}.
\]
Then $S$ is a compact convex subset of $\Delta_k$, and
$\bm W(x,y)\in S$ almost everywhere. On $S$, the function $\Phi$
reduces to
\[
\Phi(\bm\mu)
=
\langle\bm c,\bm\mu\rangle
+
\sum_{s=1}^R\gamma_s q_s^{|E(H_s)|}
-\frac12 I_k(\bm\mu).
\]
Since $I_k$ is strictly convex on $\Delta_k$, the restriction of
$\Phi$ to $S$ is strictly concave. It therefore has a unique maximizer,
say $\bm\mu^*\in S$. Since $\bm W(x,y)$ belongs to $S$ and maximizes
$\Phi$ almost everywhere, it follows that
\[
\bm W(x,y)=\bm\mu^*
\]
for almost every $(x,y)\in[0,1]^2$. Thus $\bm W$ is a constant
$k$-color probability graphon. Finally, because
$\Phi(\bm\mu^*)=\sup_{\bm\mu\in\Delta_k}\Phi(\bm\mu)$, the vector
$\bm\mu^*$ is a maximizer of \eqref{eq:explicit_solvable_formula}.
\hfill $\Box$

\subsection{Proof of Theorem \ref{thm:NegativeSidorenko_RS}}
\label{sec:negativesymmetrypf}

For a $k$-color probability graphon $\bm W \in \mathcal{W}_k$, define its global spatial average vector $\overline{\bm{W}} \in \Delta_k$ by:
$$
\overline{\bm{W}} := \int_{[0,1]^2} W(x,y) \, \mathrm{d}x\, \mathrm{d}y.
$$
Note that, by the linearity of the integral, $t(K_2, \langle \bm c_1 , \bm{W} \rangle) = \langle \bm c_1 , \overline{\bm{W}} \rangle$ and $\int_{[0,1]^2}  \langle \bm \eta_s , \bm{W} \rangle (x,y)\, \mathrm{d}x\, \mathrm{d}y = \langle \bm \eta_s , \overline{\bm{W}} \rangle$, for any vector $\bm \eta_s \in \R^k_{\geq 0}$. Also, since each graph $F_s$ satisfies the Sidorenko property,\footnote{Note that, although $\langle\eta_s,\bm W\rangle$ need not take values in $[0,1]$,
the Sidorenko inequality extends by scaling to arbitrary bounded
non-negative kernels. }
$$
t(F_s, \langle \bm \eta_s , \bm{W} \rangle ) \ge \left( \int_{[0,1]^2} \langle \bm \eta_s , \bm{W} \rangle (x,y) \, \mathrm{d}x\, \mathrm{d}y \right)^{|E(F_s)| } = (\langle \bm \eta_s , \overline{\bm{W}} \rangle)^{|E(F_s)| }.
$$
Since the coefficients $\delta_1, \delta_2, \ldots, \delta_R$ are non-negative, 
\begin{equation}\label{eq:sidorenko_pointwise_flip}
T(\bm W) = \langle \bm c_1 , \overline{\bm{W}} \rangle- \sum_{s=1}^R \delta_s \, t(F_s, \langle \bm \eta_s , \bm{W} \rangle ) \le \langle \bm c_1 , \overline{\bm{W}} \rangle - \sum_{s=1}^R \delta_s (\langle \bm \eta_s , \overline{\bm{W}} \rangle)^{|E(F_s)| }.
\end{equation}
Next, we bound the entropy functional. Recall that the function $I_k(\bm \mu) = \sum_{a=1}^k \mu_a \log \mu_a$ is strictly convex on the compact simplex $\Delta_k$. Applying Jensen's inequality with respect to the uniform probability measure on $[0,1]^2$ yields:
\begin{align}\label{eq:entropy_jensen_bound}
\mathcal{I}(\bm W) = \int_{[0,1]^2} I_k(\bm W(x,y)) \, \mathrm{d}x\, \mathrm{d}y \ge I_k\left( \int_{[0,1]^2} \bm W(x,y) \, \mathrm{d}x\, \mathrm{d}y \right) = I_k(\overline{\bm{W}}) . 
\end{align}
Now, defining $\mathcal{F}(\bm W) := T(\bm W) - \frac{1}{2}\mathcal{I}(\bm W)$ and combining~\eqref{eq:sidorenko_pointwise_flip} and~\eqref{eq:entropy_jensen_bound} gives: 
$$
\mathcal{F}(\bm W) \le \langle \bm c_1 , \overline{\bm{W}} \rangle - \sum_{s=1}^R \delta_s (\langle \bm \eta_s , \overline{\bm{W}} \rangle)^{|E(F_s)| } - \frac{1}{2} I_k(\overline{\bm{W}}) = \mathcal{F}(\overline{\bm{W}}),
$$
where $\mathcal{F}(\overline{\bm{W}})$ is the functional evaluated at the constant $k$-color probability graphon $\overline{\bm{W}}$. Taking the supremum over both sides yields the reduction to the finite-dimensional optimization problem~\eqref{eq:NegativeSidorenko_finite}.

To prove the necessity of constant maximizers, let $\bm W^*$ be a global maximizer of $\mathcal{F}(\bm W)$, over $\bm W \in \mathcal W_k$. This means equality must hold in the entropy bound~\eqref{eq:entropy_jensen_bound}:
$$
\int_{[0,1]^2} I_k(\bm W^*(x,y)) \, \mathrm{d}x\, \mathrm{d}y = I_k\left( \int_{[0,1]^2} \bm W^*(x,y) \, \mathrm{d}x\, \mathrm{d}y \right).
$$
Since the multivariate negative entropy $-I_k(\bm \mu)$ is strictly concave, equality holds in Jensen's inequality if and only if the integrand is constant almost everywhere. Thus, $\bm W^*(x,y) \equiv \overline{\bm{W}}^*$, for almost every $(x,y) \in [0,1]^2$, completing the proof.   \hfill $\Box$  

\section{Proofs from Section \ref{sec:euler_lagrange}}

\subsection{Proof of Theorem \ref{thm:EL-prob-graphon}} 
\label{thm:EL-prob-graphonpf}

Let $\wt{\bm W}\in\wt{\mathcal W}_k$ be a maximizer of $\mathcal F$ and $\bm W = (W_1, W_2, \ldots, W_k) \in\mathcal W_k$ be any representative of the equivalence class. For $a\in[k]$, define the symmetric potential
\begin{align}\label{eq:sumTW}
\widehat\Phi_a(x,y) := \sum_{s=1}^r \beta_s\,\widehat\Delta_{\bm H_s}^{(a)}(\bm W)(x,y).
\end{align}  
We first show that, for all $a \in [k]$, $W_a$ is strictly positive almost everywhere. To this end, fix $a\in[k]$, and suppose that the set
$$
B_a:=\{(x,y)\in[0,1]^2:\ W_a(x,y)=0\}
$$
has positive measure. Define $\bm U=(U_1,\dots,U_k)$ by
$$
U_a(x, y):=\bm 1\{B_a\}(x,y)(1-W_a (x, y))=\bm 1\{B_a\}(x,y), \quad  \text{ and }
\quad
U_b(x, y):=-\bm 1\{B_a\}(x,y)W_b(x ,y)
$$
for $b\neq a$, where, for a measurable set $S$, we denote with $1\{S\}$ the indicator function of $S$.
Note that $\sum_{b=1}^k U_b(x, y)=0$, for almost every $x, y \in [0, 1]$.\newline
For every $t\in[0,1]$, let $\bm{W}^t := (W^t_1, W^t_2, \ldots, W^t_k)$, where 
$$
W^t_b(x, y):=W_b(x, y) + t U_b (x, y), 
$$
for $b\in[k]$. Note that $\bm{W}^t \in \mathcal W_k$. In particular, on $B_a$, the perturbation is explicitly $W^t_{a} = t$ and $W^t_{b} = (1-t)W_b$, for $b \neq a$. Further, outside $B_a$ it equals $\bm{W}$. Since $B_a$ is symmetric, $\bm{W}^t$ is measurable and symmetric, its coordinates are non-negative, and they still sum to 1.  Hence, recalling the definition of $T$ from~\eqref{eq:sufficientTW},  observe that
\[
T(\bm W^t)=T(\bm W+t\bm U)
\]
is a polynomial in $t$. Indeed, each occurrence of $W_b$ in a homomorphism
density is replaced by the affine function $W_b+tU_b$, and $T$ is a linear
combination of homomorphism densities of colored graphs. Therefore,
\begin{equation}\label{eq:TWU}
T(\bm W+t\bm U)
=
T(\bm W)
+
t\,DT(\bm W)[\bm U]
+
R(t),
\end{equation}
where $DT(\bm W)[\bm U]$ denotes the directional derivative of $T$ at
$\bm W$ in the direction $\bm U$ and $R(t)$ is the remainder term. Using again the fact that $T$ is a linear combination of homomorphism densities, expanding
$T(\bm W+t\bm U)$ yields $|DT(\bm W)[\bm U]| \le C\|\bm U\|_\infty$,
where the constant $C$ depends only on $T$. As
$\|\bm U\|_\infty\le1$, it follows that $|DT(\bm W)[\bm U]| \le
C$. Moreover, we obtain that the remainder in~\eqref{eq:TWU} is uniform. More precisely, there exists a constant $K>0$, depending only on $T$, such that
$R(t) = O(t^2)$ and $|R(t)|\le Kt^2$, for all sufficiently small $t$. 
Consequently,
\[ T(\bm W^t)-T(\bm W) = t\,DT(\bm W)[\bm U]+R(t) \ge -Ct-Kt^2. \]
Hence, there exists $C_1>0$ such that $T(\bm W^t)-T(\bm W) \ge
-C_1t$, for all sufficiently small $t>0$.

Recall that for $(x, y) \in B_a$, we have $W_a(x, y)=0$ and $\sum_{b\neq a}W_b(x, y)=1$, while $\bm{W}^t_a(x, y)=t$ and $\bm{W}^t_b(x, y)=(1-t)W_b(x, y)$, for $b\neq a$. Hence, for $(x,y)\in B_a$,
\begin{align*}
&\sum_{b=1}^k W_b^t(x,y)\log W_b^t(x,y)-\sum_{b=1}^k W_b(x,y)\log W_b(x,y) \\
&\qquad = t\log t +(1-t)\log(1-t) -t\sum_{b\neq a}W_b(x,y)\log W_b(x,y).
\end{align*}
Since the last sum is bounded below by $-\log k$, there exists a constant $C_2 : = C_2(k) <\infty$ such that, 
$$
\mathcal I(\bm{W}^t)-\mathcal I(\bm W)
\le |B_a|\,t \log t + C_2 t , 
\qquad\text{for all sufficiently small } t>0.
$$
Therefore, recalling \eqref{eq:F} and the bound in \eqref{eq:TWU} gives, 
$$
\mathcal F(\bm{W}^t)-\mathcal F(\bm W) \ge -C_1 t - \frac12|B_a|\,t \log t - \frac12 C_2 t , 
$$
for all sufficiently small $t>0$. Note that the right-hand side above is strictly positive for all sufficiently small $t>0$, which contradicts the maximality of $\bm W$. This implies that $
W_a(x,y)>0$ for almost every $(x,y)\in[0,1]^2$. Since $a\in[k]$ was arbitrary, every coordinate of $\bm W = (W_1, W_2, \ldots, W_k)$ is positive almost everywhere. 

We now derive the Euler--Lagrange equation on interior truncations. For $\eta>0$, let
$$
A_\eta:= \left\{(x,y)\in[0,1]^2:\ \min_{1\le a\le k} W_a(x,y)\ge \eta \right\}.
$$
Since $W_a>0$ almost everywhere, for every $a \in [k]$, the sets $\{A_\eta\}_{\eta > 0}$ increase to a full-measure subset of $[0,1]^2$ as $\eta\downarrow0$. Let $\bm U=(\bm U_1, \bm U_2 \dots, \bm U_k)$ be bounded symmetric functions, supported in $A_\eta$, such that $\sum_{a=1}^k U_a(x,y)=0$, for all $(x,y)\in[0,1]^2$. For $t \in\mathbb R$ and $a \in [k]$, set
$$
W_a^t(x,y):=W_a(x,y)+t U_a(x,y).
$$
Since $W_a(x, y) \ge \eta$, for $(x, y) \in A_\eta$, for $|t|$ small enough we still have $\bm{W}^t\in\mathcal W_k$. By maximality of $\bm W$ we have, 
\begin{align}\label{eq:FWt}
\frac{\mathrm{d}}{\mathrm{d}t}\mathcal F(\bm{W}^t)\Big|_{t=0}=0.
\end{align}
For the entropy term,
\begin{align}\label{eq:IWt}
\frac12\frac{\mathrm{d}}{\mathrm{d}t}\mathcal I(\bm{W}^t)\Big|_{t=0}
&=\frac12\int_{[0,1]^2}\sum_{a=1}^k U_a(x,y)\left(\log W_a(x,y)+1\right)\, \mathrm{d}x\, \mathrm{d}y \\
&=\frac12\int_{[0,1]^2}\sum_{a=1}^k U_a(x,y)\log W_a(x,y)\, \mathrm{d}x\, \mathrm{d}y,  \nonumber  
\end{align}
where the constant term vanishes because $\sum_{a=1}^k U_a=0$. 
For the energy term, applying
\eqref{eq:first_variation_symmetric} to the symmetric perturbation
$\bm U$ gives, for each $s\in[r]$,
\[
\frac{\mathrm d}{\mathrm dt}
t(\bm H_s,\bm W^t)\Big|_{t=0}
=
\int_{[0,1]^2}
\sum_{a=1}^k
U_a(x,y)\,
\widehat\Delta_{\bm H_s}^{(a)}(\bm W)(x,y)
\,\mathrm dx\,\mathrm dy,
\]
where $\widehat\Delta_{\bm H_s}^{(a)}$ is defined in
\eqref{eq:Delta_H_sym}. Hence, recalling
\eqref{eq:sufficientTW} and the definition of $\widehat\Phi_a$ in
\eqref{eq:sumTW},
\begin{align}\label{eq:TWt}
\frac{\mathrm d}{\mathrm dt}
T(\bm W^t)\Big|_{t=0} =
\sum_{s=1}^r
\beta_s
\frac{\mathrm d}{\mathrm dt}
t(\bm H_s,\bm W^t)\Big|_{t=0} =
\int_{[0,1]^2}
\sum_{a=1}^k
U_a(x,y)\widehat\Phi_a(x,y)
\,\mathrm dx\,\mathrm dy.
\end{align}
Combining \eqref{eq:FWt}, \eqref{eq:IWt}, and \eqref{eq:TWt}
therefore gives
\begin{equation}\label{eq:EL_symmetric_test}
0
=
\int_{[0,1]^2}
\sum_{a=1}^k
U_a(x,y)
\left[
\widehat\Phi_a(x,y)
-\frac12\log W_a(x,y)
\right]
\,\mathrm dx\,\mathrm dy.
\end{equation}
For $a\in[k]$, set
\[ \Xi_a(x,y) := \widehat\Phi_a(x,y)-\frac12\log W_a(x,y). \]
Each $\Xi_a$ is symmetric, and its restriction to $A_\eta$ is bounded.
Fix two colors $a,b\in[k]$. In
\eqref{eq:EL_symmetric_test}, choose $U_a=h$, $U_b=-h$, and $U_c=0$,  for $c\notin\{a,b\}$,  where $h$ is any bounded symmetric function supported in $A_\eta$. Then, 
\[ \int_{A_\eta} h(x,y)\bigl[\Xi_a(x,y)-\Xi_b(x,y)\bigr] \,\mathrm dx\,\mathrm dy =
0. \]
Since $\Xi_a-\Xi_b$ is bounded and symmetric on $A_\eta$, we may take
$h = \bm 1\{A_\eta\}\bigl(\Xi_a-\Xi_b\bigr)$, which yields 
$$\int_{A_\eta} \left|\Xi_a(x,y)-\Xi_b(x,y)\right|^2 \,\mathrm dx\,\mathrm dy = 0.$$
Consequently, $\Xi_a(x,y)=\Xi_b(x,y)$, for every $a,b\in[k]$ and almost every $(x,y)\in A_\eta$. Since the sets $A_\eta$ increase to a full-measure subset of
$[0,1]^2$, as $\eta\downarrow0$, there exists a symmetric measurable
function $\lambda:[0,1]^2\to\mathbb R$ such that
\[
\widehat\Phi_a(x,y)-\frac12\log W_a(x,y)
=
\lambda(x,y)
\]
for every $a\in[k]$ and almost every $(x,y)\in[0,1]^2$.
Rearranging gives
\[ W_a(x,y) = \exp\left( 2\widehat\Phi_a(x,y)-2\lambda(x,y) \right). \]
Using $\sum_{a=1}^kW_a(x,y)=1$, we obtain $\exp (2\lambda(x,y) ) =
\sum_{b=1}^k \exp( 2\widehat\Phi_b(x,y) )$. Then substitution gives the result in 
\eqref{eq:softmax}.

To complete the proof of Theorem~\ref{thm:EL-prob-graphon}, it remains to
show that the maximizer is uniformly separated from the boundary of the
simplex. By \eqref{eq:Delta_H}, for every $s\in[r]$ and $a\in[k]$,
\[
\left|
\Delta_{\bm H_s}^{(a)}(\bm W)(x,y)
\right|
\le
|E_a(H_s)|
\le
|E(H_s)|
\]
for almost every $(x,y)\in[0,1]^2$. Indeed,
$\Delta_{\bm H_s}^{(a)}$ is a sum of $|E_a(H_s)|$ terms, and every term is
an integral of a product of functions taking values in $[0,1]$. Hence, 
recalling the definition of $\widehat\Delta_{\bm H_s}^{(a)}(\bm W)(x,y)$ from \eqref{eq:Delta_H_sym} gives, 
\[
\left|
\widehat\Delta_{\bm H_s}^{(a)}(\bm W)(x,y)
\right|
\le
|E(H_s)|
\]
for every $s\in[r]$, every $a\in[k]$, and almost every
$(x,y)\in[0,1]^2$.
Recalling the definition of $\widehat\Phi_a$ in \eqref{eq:sumTW} and applying
the triangle inequality, we obtain
\begin{align*} 
\left| \widehat\Phi_a(x,y) \right| = \left| \sum_{s=1}^r \beta_s \widehat\Delta_{\bm H_s}^{(a)}(\bm W)(x,y) \right| \le \sum_{s=1}^r |\beta_s| \left| \widehat\Delta_{\bm H_s}^{(a)}(\bm W)(x,y) \right| \le \sum_{s=1}^r |\beta_s|\,|E(H_s)| =:C.
\end{align*}
Thus, $-C \le \widehat\Phi_a(x,y) \le C$, for every $a\in[k]$ and almost every $(x,y)\in[0,1]^2$. 
Applying these bounds to \eqref{eq:softmax} gives $W_a(x,y) \ge \frac1k e^{-4C}
=:\varepsilon$. Hence, every coordinate is uniformly bounded away from zero. Since
$\sum_{b=1}^kW_b(x,y)=1$, we also have $$W_a(x,y) = 1-\sum_{b\ne a}W_b(x,y) \le 1-(k-1)\varepsilon \le 1-\varepsilon, $$
where the final inequality uses $k\ge2$. Therefore, $\varepsilon \le W_a(x,y) \le
1-\varepsilon$, for every $a\in[k]$ and almost every $(x,y)\in[0,1]^2$. This completes the proof of Theorem~\ref{thm:EL-prob-graphon}.  \hfill $\Box$

\subsection{ Proof of Theorem \ref{thm:highT_colored} } 
\label{sec:highT_coloredpf}

For a vector $\bm z = (z_1, z_2, \ldots, z_k)\in\mathbb R^k$, define its oscillation as: 
\begin{align}\label{eq:normz}
\osc(\bm z):=\max_{a\in[k]} z_a-\min_{a\in[k]} z_a.
\end{align} 
We begin establishing the following key property of the oscillation. 

\begin{lemma} 
\label{lem:osc_distance}
For every $\bm z\in\mathbb R^k$, 
$$\osc(\bm z) = 2 \inf_{c\in\mathbb R}\|\bm z-c\bm 1\|_\infty.$$
\end{lemma}

\begin{proof}
Let $M:=\max_{a\in[k]} z_a$, $m:=\min_{a\in[k]} z_a$, and $c:=\frac{M+m}{2}$. Then for every $a\in[k]$, 
$$
|z_a-c|\le \frac{M-m}{2}=\frac12\,\osc(\bm z).
$$
Therefore, $\| z-c\bm 1 \|_\infty\le \frac12\,\osc(\bm z)$.

Conversely, for any $c\in\mathbb R$, $\| \bm z-c\bm 1\|_\infty\ge \max\{|M-c|,|m-c|\}\ge \frac{M-m}{2} =\frac12\,\osc(\bm z)$. This proves the claim.
\end{proof}

Now, recall the definition of the softmax function from \eqref{eq:softmaxz}. Note that the $\SoftMax$ function satisfies the following uniform $L^\infty$ Lipschitz bound. For $\bm{x}, \bm{y} \in \R^k$,
\begin{equation}
\label{eq:softmax_infty_half}
\|\SoftMax(\bm{x})-\SoftMax(\bm{y})\|_\infty
\le
\frac12\,\|\bm{x} - \bm{y}\|_\infty. 
\end{equation}
This follows from the Jacobian formula $\nabla\SoftMax(\bm{x})=\diag(\bm p)- \bm p \bm p^\top$, where $\bm p=\SoftMax(\bm x )$ and the estimate
$$
\max_{a\in[k]}\sum_{b=1}^k |(\diag(\bm p)- \bm p \bm p^\top)_{ab}|
=
2\max_{a\in[k]} p_a(1-p_a)
\le \frac12.
$$
The Lipschitz bound in \eqref{eq:softmax_infty_half}, however, yields a coarser high-temperature condition than the one required in Theorem~\ref{thm:highT_colored}.
To obtain the sharper criterion, we need the following oscillation-based Lipschitz bound, which follows from the representation in Lemma~\ref{lem:osc_distance}.

\begin{lemma}
\label{lem:softmax_osc_lip}
For all $\bm{x}, \bm{y} \in\mathbb R^k$,
\begin{equation}
\label{eq:softmax_osc_bound}
\|\SoftMax(\bm{x})-\SoftMax(\bm{y})\|_\infty
\le
\frac14\,\osc(\bm{x} - \bm{y}).
\end{equation}
\end{lemma}

\begin{proof} 
Let $\bm{\delta}:=\bm{x} - \bm{y}$. Then using $\SoftMax(\bm{x})=\SoftMax(\bm x -c\bm 1)$, for every $c\in\mathbb R$, and \eqref{eq:softmax_infty_half} gives, 
$$
\|\SoftMax(\bm{x})-\SoftMax(\bm{y})\|_\infty
=
\|\SoftMax(\bm{x}-c\bm 1)-\SoftMax(\bm{y})\|_\infty
\le
\frac12\,\|\bm{\delta}-c\bm 1\|_\infty.
$$
Optimizing over $c\in\mathbb R$ and applying \Cref{lem:osc_distance}, we obtain
$$
\|\SoftMax(\bm{x})-\SoftMax(\bm{y})\|_\infty
\le
\frac12\inf_{c\in\mathbb R}\|\bm{\delta}-c\bm 1\|_\infty
=
\frac{1}{4}\,\osc(\bm{\delta}),
$$
which is exactly \eqref{eq:softmax_osc_bound}.
\end{proof}

\begin{remark}
Note that in the classical uncolored setting, which corresponds to $k=2$, one does not need to introduce oscillation as a separate object. 
This is because, for $\bm x = (x_1, x_2)\in\mathbb
R^2$, $\osc(\bm{x})=|x_1-x_2|$. Hence, quotienting out the common additive shift simply reduces the problem to a single scalar variable. For $k\ge 3$, by contrast, there are genuinely several independent
pairwise differences, and this is precisely why the oscillation-based formulation
becomes the natural one (recall the discussion following Theorem \ref{thm:highT_colored}). 
\end{remark}

With the above preparations, we now proceed with the proof of Theorem \ref{thm:highT_colored}. To establish a contradiction, suppose $\bm W= (W_1, W_2, \ldots, W_k)$ and $\bm V = (V_1, V_2, \ldots, V_k)$ be two maximizers for the function $\mathcal F$ (recall \eqref{eq:F}). Hence, both of them satisfy \eqref{eq:hightemp}. This means, for almost every $(x,y)\in[0,1]^2$,
\begin{align}\label{eq:WV}
\bm W(x,y)
&=
\SoftMax\left(
2\widehat{\bm\Psi}(\bm W)(x,y)
\right),
&
\bm V(x,y)
&=
\SoftMax\left(
2\widehat{\bm\Psi}(\bm V)(x,y)
\right).
\end{align} 
Now, for a vector-valued function $\bm F=(F_1,\dots,F_k)$ on $[0,1]^2$, define 
$$
\|\bm F\|_{\mathrm{osc},\infty}
:=
\operatorname*{ess\,sup}_{(x,y)\in[0,1]^2}\osc(\bm F(x,y))
=
\max_{ 1 \leq a \ne b \leq k }\|F_a-F_b\|_\infty.
$$
Applying Lemma~\ref{lem:softmax_osc_lip} pointwise to the two
fixed-point equations in \eqref{eq:WV} gives
\begin{align*}
\|\bm W(x,y)-\bm V(x,y)\|_\infty
&=
\left\|
\SoftMax\left(2\widehat{\bm\Psi}(\bm W)(x,y)\right)
-
\SoftMax\left(2\widehat{\bm\Psi}(\bm V)(x,y)\right)
\right\|_\infty
\\
&\le
\frac14
\osc\left(
2\widehat{\bm\Psi}(\bm W)(x,y)
-
2\widehat{\bm\Psi}(\bm V)(x,y)
\right)
\\
&=
\frac12
\osc\left(
\widehat{\bm\Psi}(\bm W)(x,y)
-
\widehat{\bm\Psi}(\bm V)(x,y)
\right).
\end{align*}
Taking the essential supremum over $(x,y)$ yields
\begin{align}\label{eq:maximumWV}
\|\bm W-\bm V\|_\infty :=
\operatorname*{ess\,sup}_{(x,y)\in[0,1]^2}
\|\bm W(x,y)-\bm V(x,y)\|_\infty \le \frac12
\left\| \widehat{\bm\Psi}(\bm W) - \widehat{\bm\Psi}(\bm V) \right\|_{\mathrm{osc},\infty}.
\end{align} 
Next, for every pair $a\ne b$, define
\[
\widehat\Theta_{ab}(\bm W)
:=
\widehat\Psi_a(\bm W)-\widehat\Psi_b(\bm W)
=
\sum_{s=1}^r
\beta_s
\left[
\widehat\Delta_{\bm H_s}^{(a)}(\bm W)
-
\widehat\Delta_{\bm H_s}^{(b)}(\bm W)
\right].
\] 
We first record the corresponding estimate for the oriented kernels. By
\eqref{eq:Delta_H},
$\Delta_{\bm H_s}^{(a)}(\bm W)(x,y)$ is a sum of
$|E_a(H_s)|$ terms. Each term is an integral of a product of exactly
$|E(H_s)|-1$ factors, every factor being one coordinate of $\bm W$ and
taking values in $[0,1]$. Using the elementary inequality
\[
\left|
\prod_{j=1}^{K}u_j-\prod_{j=1}^{K}v_j
\right|
\le
\sum_{j=1}^{K}|u_j-v_j|,
\qquad
u_j,v_j\in[0,1],
\]
we obtain, for every $s\in[r]$ and $a\in[k]$,
\begin{equation}\label{eq:oriented_derivative_lipschitz}
\left\|
\Delta_{\bm H_s}^{(a)}(\bm W)
-
\Delta_{\bm H_s}^{(a)}(\bm V)
\right\|_\infty
\le
|E_a(H_s)|
\bigl(|E(H_s)|-1\bigr)
\|\bm W-\bm V\|_\infty.
\end{equation}
Now, let
\[
(\mathsf Sf)(x,y)
:=
\frac12\bigl(f(x,y)+f(y,x)\bigr)
\]
denote the symmetrization operator. By
\eqref{eq:Delta_H_sym}, $\widehat\Delta_{\bm H_s}^{(a)} = \mathsf S\Delta_{\bm H_s}^{(a)}$. Further, $\mathsf S$ is an $L^\infty$ contraction, that is, $\|\mathsf Sf\|_\infty\le\|f\|_\infty$. Therefore,
\begin{align*}
&
\left\|
\left[
\widehat\Delta_{\bm H_s}^{(a)}(\bm W)
-
\widehat\Delta_{\bm H_s}^{(b)}(\bm W)
\right]
-
\left[
\widehat\Delta_{\bm H_s}^{(a)}(\bm V)
-
\widehat\Delta_{\bm H_s}^{(b)}(\bm V)
\right]
\right\|_\infty
\\
&\qquad=
\left\|
\mathsf S
\left(
\left[
\Delta_{\bm H_s}^{(a)}(\bm W)
-
\Delta_{\bm H_s}^{(a)}(\bm V)
\right]
-
\left[
\Delta_{\bm H_s}^{(b)}(\bm W)
-
\Delta_{\bm H_s}^{(b)}(\bm V)
\right]
\right)
\right\|_\infty
\\
&\qquad\le
\left\|
\Delta_{\bm H_s}^{(a)}(\bm W)
-
\Delta_{\bm H_s}^{(a)}(\bm V)
\right\|_\infty
+
\left\|
\Delta_{\bm H_s}^{(b)}(\bm W)
-
\Delta_{\bm H_s}^{(b)}(\bm V)
\right\|_\infty
\\
&\qquad\le
\bigl(|E(H_s)|-1\bigr)
\bigl(|E_a(H_s)|+|E_b(H_s)|\bigr)
\|\bm W-\bm V\|_\infty,
\end{align*}
where the final inequality follows from
\eqref{eq:oriented_derivative_lipschitz}. Consequently, for every $a\ne b$,
\begin{align*}
\left\|
\widehat\Theta_{ab}(\bm W)
-
\widehat\Theta_{ab}(\bm V)
\right\|_\infty
&\le  
\sum_{s=1}^r
|\beta_s|
\bigl(|E(H_s)|-1\bigr)
\bigl(|E_a(H_s)|+|E_b(H_s)|\bigr)
\|\bm W-\bm V\|_\infty
\nonumber\\
&\le
\Lambda\|\bm W-\bm V\|_\infty.
\end{align*}  
Taking the maximum over $a\ne b$ and using the definition of the oscillation
seminorm gives,  
\begin{equation*}
\left\|
\widehat{\bm\Psi}(\bm W)
-
\widehat{\bm\Psi}(\bm V)
\right\|_{\mathrm{osc},\infty}
\le
\Lambda\|\bm W-\bm V\|_\infty.
\end{equation*}
Combining this with \eqref{eq:maximumWV}, we conclude that $\|\bm{W}-\bm{V}\|_\infty \le \frac12\,\Lambda\,\|\bm{W}-\bm{V}\|_\infty$.
Since $\Lambda<2$, this means $\|\bm{W}-\bm{V}\|_\infty=0$, hence
$W=V$ almost everywhere. This proves the uniqueness of the maximizer. 

Finally, if $\phi\colon[0,1]\to[0,1]$ is any measure-preserving bijection, then
the relabeled graphon
$$
\bm W^\phi(x,y):= \bm W(\phi(x),\phi(y))
$$
satisfies $\mathcal F(\bm W^\phi)=\mathcal F(\bm W)$. Hence, $\bm W^\phi$ is again a maximizer. By uniqueness, $\bm W^\phi=\bm W$, almost everywhere, for every such $\phi$. Therefore, the unique maximizer must be invariant under all measure-preserving relabelings, and thus it is constant almost everywhere.  \hfill  $\Box$

\section{Proofs from Section \ref{sec:temperatureasymptotics}} 

\subsection{Proof of Theorem \ref{ThmConcentrationBeta} } 
\label{sec:ThmConcentrationBetapf}

Set
\[
h(\wt{\bm W})
:=
g(\wt{\bm W})-\frac12\mathcal I(\wt{\bm W}).
\]
Since $g$ is continuous and $\mathcal I$ is lower semicontinuous
\cite[Corollary 3.12.1]{dionigiLargeDeviationsProbability2025}, the
function $h$ is upper semicontinuous. For every fixed $\beta>0$,
Theorem~\ref{thm:TiltedLDP} implies that
$\tilde{\P}_{n,\beta}$ satisfies an LDP with speed $n^2/2$ and good
rate function
\begin{equation*}
I_\beta(\wt{\bm W})
=
2\beta\left[
\sup_{\wt{\bm V}\in\wt{\mathcal W}_k}
\left(f(\wt{\bm V})+\frac1\beta h(\wt{\bm V})\right)
-
\left(f(\wt{\bm W})+\frac1\beta h(\wt{\bm W})\right)
\right].
\end{equation*}  
Let $\mathcal N_f$ and $\mathcal N^*$ be as defined in
\eqref{eq:largetemperaturemaximizerf} and
\eqref{eq:largetemperaturemaximizer}, and let
\[
M_f:=\max_{\wt{\bm W}\in\wt{\mathcal W}_k}f(\wt{\bm W}),
\qquad
M^*:=\max_{\wt{\bm W}\in\mathcal N_f}h(\wt{\bm W}).
\]
These maxima are attained because $f$ is continuous, $h$ is upper
semicontinuous, and the relevant sets are compact. In particular,
$\mathcal N^*$ is nonempty and compact. Fix
$\wt{\bm W}^*\in\mathcal N^*$. Then, for every $\beta>0$,
\begin{equation}\label{eq:temperature-lower-bound}
\sup_{\wt{\bm V}\in\wt{\mathcal W}_k}
\left\{\beta f(\wt{\bm V})+h(\wt{\bm V})\right\}
\geq \beta M_f+M^*.
\end{equation} 
We claim that there exist $\delta>0$ and $\beta_0>0$ such that
\begin{equation}\label{eq:uniform-temperature-gap}
\sup_{\wt{\bm W}\in\mathcal S_\varepsilon}
\left\{\beta f(\wt{\bm W})+h(\wt{\bm W})\right\}
\leq \beta M_f+M^*-\delta
\end{equation}
for every $\beta\geq\beta_0$. If $\mathcal S_\varepsilon$ is empty,
the conclusions are immediate, so suppose that it is nonempty. If the
claim were false, there would exist a sequence $\beta_L\to\infty$ and
$\wt{\bm W}_L\in\mathcal S_\varepsilon$ such that
\begin{equation}\label{eq:perturb_ineq}
\beta_L\bigl(f(\wt{\bm W}_L)-M_f\bigr)
+h(\wt{\bm W}_L)
\geq M^*-\frac{1}{L}.
\end{equation}

Since $\mathcal S_\varepsilon$ is a closed subset of the compact space
$\wt{\mathcal W}_k$, after passing to a subsequence we may assume that $\wt{\bm W}_L \rightarrow \wt{\bm W}_\infty\in\mathcal S_\varepsilon$. Moreover, $h$ attains a finite maximum $H$ on
$\wt{\mathcal W}_k$. Hence, it follows from \eqref{eq:perturb_ineq} that
\[
0\leq M_f-f(\wt{\bm W}_L)
\leq\frac{H-M^*+\frac{1}{L}}{\beta_L}.
\]
Therefore $f(\wt{\bm W}_L)\to M_f$, and the continuity of $f$ gives $\wt{\bm W}_\infty\in\mathcal N_f$. Since $f(\wt{\bm W}_L)-M_f\leq0$, \eqref{eq:perturb_ineq} also gives $h(\wt{\bm W}_L)\geq M^*-\frac{1}{L}$. By the upper semicontinuity of $h$, 
$$h(\wt{\bm W}_\infty) \geq \limsup_{L \to\infty}h(\wt{\bm W}_L)
\geq M^*.$$ Because $\wt{\bm W}_\infty\in\mathcal N_f$, the definition of $M^*$
implies that $\wt{\bm W}_\infty\in\mathcal N^*$. This contradicts
$\wt{\bm W}_\infty\in\mathcal S_\varepsilon$ and proves
\eqref{eq:uniform-temperature-gap}.

We now prove \eqref{eq:tempertauredistance}. If
$\wt{\bm W}\in\Gamma_\beta$, then its optimality and
\eqref{eq:temperature-lower-bound} give $\beta f(\wt{\bm W})+h(\wt{\bm W})
\geq\beta M_f+M^*$. For $\beta\geq\beta_0$, this is incompatible with
\eqref{eq:uniform-temperature-gap} whenever $\wt{\bm W}\in\mathcal S_\varepsilon$. Hence, $\sup_{\wt{\bm W}\in\Gamma_\beta}
\delta_\square(\wt{\bm W},\mathcal N^*)
<\varepsilon$, for every sufficiently large $\beta$. Since $\varepsilon>0$ is
arbitrary, \eqref{eq:tempertauredistance} follows.

Finally, for every $\beta\geq\beta_0$ and every $\wt{\bm W}\in\mathcal S_\varepsilon$, applying \eqref{eq:temperature-lower-bound} and
\eqref{eq:uniform-temperature-gap} gives, $I_\beta(\wt{\bm W})\geq2\delta$. Since $\mathcal S_\varepsilon$ is closed, the LDP upper bound in Theorem \ref{thm:TiltedLDP} gives, for every fixed $\beta\geq\beta_0$, 
$$\limsup_{n\to\infty}\frac{2}{n^2}
\log\tilde{\P}_{n,\beta}(\mathcal S_\varepsilon)
\leq
-\inf_{\wt{\bm W}\in\mathcal S_\varepsilon}
I_\beta(\wt{\bm W})
\leq-2\delta.
$$ Consequently, there exists $n_0=n_0(\varepsilon,\beta)$ such that $
\tilde{\P}_{n,\beta}(\mathcal S_\varepsilon) \leq \exp(-\frac{\delta}{2}n^2)$, for every $n\geq n_0$. Taking $K=\delta/2$ proves \eqref{eq:tempertaureset}. \hfill $\Box$

\subsection{Proof of Theorem \ref{thm:wedge-beta-plus-infty}  } 
\label{sec:wedge-beta-plus-inftypf}

In the regime, $\beta_2 \rightarrow + \infty$ the model favors graphons that maximize the induced wedge density. Specifically, the proof of Theorem \ref{thm:wedge-beta-plus-infty} relies on the following key lemma about the structure of graphons that maximize the induced wedge density. The proof of the lemma is given in Appendix \ref{sec:wedge-max-bipartitepf}.

\begin{lemma}
\label{lem:wedge-max-bipartite}
For a graphon $W\colon[0,1]^2\to[0,1]$, let $t_{\mathrm{ind}}(K_{1,2},W)$ be the homomorphism density of the induced 2-star in $W$. 
Then
\begin{align}\label{eq:maximumwedge}
\max_{W} t_{\mathrm{ind}}(K_{1,2},W)=\frac14.
\end{align}
Moreover, equality holds if and only if, $W$ is, up to a measure-preserving relabeling, the complete  balanced bipartite graphon $W_{\mathrm{bip}}$ in \eqref{eq:bipartiteW}. 
\end{lemma}

\begin{remark} 
Lemma~\ref{lem:wedge-max-bipartite} follows from classical results on the inducibility of graphs. The \emph{inducibility} of a finite graph $H$, denoted $\mathrm{ind}(H)$, is the maximum asymptotic density of induced copies of $H$ in a large graph. This notion was introduced systematically by Pippenger and Golumbic~\cite{pippenger1975inducibility}, with earlier related questions, especially for triangles and independent sets, going back to Goodman~\cite{goodman1959sets}. In particular, Goodman's result (see also \cite[Theorem~10]{pippenger1975inducibility}), implies that $\mathrm{ind}(K_{1,2})=\frac34$. With the labeled normalization used in \eqref{eq:wedgeW}, only one of the three possible choices of the center of the induced path is counted. Therefore, $\sup_W t_{\mathrm{ind}}(K_{1,2},W)=\frac14$,  which is precisely the statement in \eqref{eq:maximumwedge}. For completeness, in Appendix~\ref{sec:wedge-max-bipartitepf} we include a proof of this fact, together with a characterization of the extremal configuration.  
\end{remark}

Returning now to the proof of Theorem \ref{thm:wedge-beta-plus-infty}, recall the setup of Theorem \ref{ThmConcentrationBeta} from Section \ref{sec:temperatureasymptotics}. Specifically, in the notation of \eqref{eq:T_betaf_infty}, in this case, $f(W)= t_{\mathrm{ind}}(K_{1, 2}, W)$ and $g(W)= \beta_1 t(K_2, W)$. Hence, by Lemma~\ref{lem:wedge-max-bipartite}, $\mathcal N_f = \argmax_{W} f(W) = \{\tilde{W}_{\mathrm{bip}}\}$, and hence, $\mathcal N^* =  \{\tilde{W}_{\mathrm{bip}}\}$. Then by \eqref{eq:tempertauredistance}, the result in \eqref{eq:wedgeoptimizer} follows. 

Next, to show \eqref{eq:temperaturewedgeW}, recalling \eqref{eq:wedgetemperatureTW} first note that,  
$$\psi_\wedge(\beta_2)\ge T_{\wedge}(W_{\mathrm{bip}}) - \frac12\mathcal I(W_{\mathrm{bip}}) = 
\frac{\beta_1}{2} + \frac{\beta_2}{4} - \frac12\mathcal I(W_{\mathrm{bip}})  .$$
Dividing by $\beta_2$ and letting $\beta_2\to+\infty$, we obtain $\liminf_{\beta_2 \to+\infty}\frac{\psi_{\wedge}(\beta_2)}{\beta_2}\ge\frac14$. Conversely, since $-\frac12\mathcal I(W) 
\leq \frac{1}{2} \log 2$, for any maximizer $W^*$ of $\psi_{\wedge}(\beta_2)$,
$$\frac{\psi_{\wedge}(\beta_2)}{\beta_2} \leq \frac{\beta_1 t(K_2, W^*)}{\beta_2} +  t_{\mathrm{ind}}(K_{1,2},W^*)+\frac{\log 2}{2 \beta_2} \le \frac14+o(1), $$
again by Lemma~\ref{lem:wedge-max-bipartite}. Combining this with the lower bound above shows that $\lim_{\beta_2 \to+\infty}\frac{\psi_\wedge(\beta_2)}{\beta_2}=\frac14$, completing the proof of \eqref{eq:temperaturewedgeW}.  \hfill  $\Box$

\subsection{Proof of Theorem \ref{thm:wedge-ground-states} }
\label{sec:wedge-ground-statespf}

In the regime $\beta_2 \rightarrow - \infty$, we will apply Theorem \ref{ThmConcentrationBeta} with $f(W)= -t_{\mathrm{ind}}(K_{1, 2}, W)$ and $g(W)= \beta_1 t(K_2, W)$ (recall Remark \ref{remark:temperaturenegativeTW}). From \cite[Theorem 7.1]{janson2013graphlimits} we know that $t_{\mathrm{ind}}(K_{1,2},W)=0$ if and only if $W$ is, up to a measure preserving relabeling, a cluster graphon as in \eqref{eq:disjointW}.  
Hence, recalling \eqref{eq:largetemperaturemaximizerf}, 
$$\mathcal N_f = \{ \tilde W : t_{\mathrm{ind}}(K_{1, 2}, W) = 0\} = \tilde{\mathcal W}_{\mathrm{clus}},$$ 
where $\tilde{\mathcal W}_{\mathrm{clus}}$ is the collection of equivalence classes of cluster graphons (as defined before Theorem \ref{thm:wedge-ground-states}). Since a cluster graphon is $\{0, 1\}$-valued and $I_2((0,1))=I_2((1,0))=  0$, 
$$\mathcal I(W)=\int_{[0,1]^2} I_2(W(x,y), 1-W(x,y))\, \mathrm{d}x\, \mathrm{d}y
= 0  ,  $$ 
for any $W \in \mathcal W_{\mathrm{clus}}$. Further, since $0\leq t(K_2,W)\leq 1$, and since the endpoints are attained by $W\equiv 0$ and $W\equiv 1$, respectively, we have,  
\begin{align}\label{eq:wedgeoptimumW}  
\max_W \left\{ \beta_1 t(K_2,W)- \frac{1}{2} \mathcal I(W): t_{\mathrm{ind}}(K_{1,2},W)=0
\right\} =
\begin{cases}
0, & \text{ if } \beta_ 1 \le 0, \\ 
\beta_1 , & \text{ if }  \beta_1 > 0.
\end{cases}
\end{align}
Thus, recalling \eqref{eq:largetemperaturemaximizer}, 
\begin{align*}
\mathcal N^*=\argmax_{\wt{W}\in \mathcal{N}_f} \left( \beta_1 t(K_2, W) - \frac12\mathcal I(W)\right) = \mathcal N^*_{\wedge}, 
\end{align*}
for $\mathcal N^*_{\wedge}$ as defined in \eqref{eq:wedgenegativeN}. This proves \eqref{eq:wedgenegativeoptimizer}.

Next, to show \eqref{eq:wedgenegativeZ}, first note that 
\begin{align}\label{eq:wedgemaximumW}
\psi_{\wedge}(\beta_2) \geq M := \max_W \left\{ \beta_1 t(K_2,W)- \frac{1}{2} \mathcal I(W): t_{\mathrm{ind}}(K_{1,2},W)=0
\right\} . 
\end{align}
For the upper bound, let $\beta_2^{(L)}\to-\infty$, define $\alpha^{(L)} = - \beta_2^{(L)}$, and choose a maximizer $W^{(L)}$ of $\psi_{\wedge}(\beta_2^{(L)})$. By compactness of $\wt{\mathcal W}$, after passing to a subsequence we may assume that $\wt W^{(L)}\to\wt W_*$. Since $$ \alpha^{(L)} f(W^{(L)})+h(W^{(L)}) = \psi_{\wedge}(\beta_2^{(L)}) \geq M, $$ 
where $h(W^{(L)}) = g(W^{(L)}) - \frac{1}{2} \mathcal I(W^{(L)})$ is the tie-breaking function. Note that $h$ is bounded above, hence, we must have $f(W_*)=0$, otherwise $\alpha^{(L)} f(W^{(L)})\to-\infty$, which is a contradiction. This means that $\wt W_*\in\wt{\mathcal W}_{\mathrm{clus}}$. Since $f(W^{(L)})\leq0$, 
$$ \psi_{\wedge}(\beta_2^{(L)}) = \alpha^{(L)} f(W^{(L)})+h(W^{(L)}) \leq h(W^{(L)}). $$ Taking the upper limit and using the upper semicontinuity of $h$ gives, 
$$ \limsup_{L \to\infty}\psi_{\wedge}(\beta_2^{(L)}) \leq h(W_*) \leq \sup_{\wt W\in\wt{\mathcal W}_{\mathrm{clus}}}h(W) = M. $$ This combined with \eqref{eq:wedgeoptimumW} and \eqref{eq:wedgemaximumW}, proves the results in \eqref{eq:wedgenegativeZ}.  \hfill  $\Box$

\subsection{Proof of Theorem \ref{thm:concentrationrainbow} }
\label{sec:concentrationrainbowpf}

We will apply Theorem \ref{ThmConcentrationBeta} with 
$f(\bm W)= t_{\mathrm{rb}\triangle}(\bm W)$ and $g(\bm{W})= \beta_1 t(K_2^{(1)}, \bm W) + \beta_2 t(K_2^{(2)}, \bm W) + \beta_3 t(K_2^{(3)}, \bm W)$, where the parameters are as defined in \eqref{eq:rainbowT}. By Corollary~\ref{cor:graphon_extremality_rainbow}, 
$$\mathcal N_f = \argmax_{\tilde{\bm W} \in \tilde{\mathcal W}_3} f(\tilde{\bm W}) = \wt{\bm W}^{\mathrm{rb}},$$ where $\wt{\bm W}^{\mathrm{rb}}$ is the equivalence class (in $\tilde {\mathcal W}_3$) of the blowup 3-colored probability graphon of the $K_4$ template, as defined in Section \ref{sec:rt}. Hence, recalling \eqref{eq:largetemperaturemaximizer}, $\mathcal N^* =  \wt{\bm W}^{\mathrm{rb}}$. Then by \eqref{eq:tempertauredistance}, the result in \eqref{eq:rainbowdistanceN} follows.

It remains to prove \eqref{eq:rainbowZ}. From \eqref{eq:rbtemperatureTW}, recall that
\begin{align}\label{eq:rbenergy}
\psi_{\mathrm{rb}\triangle}(\beta_4)
=
\sup_{\bm W\in\wt{\mathcal W}_3}
\left\{
\beta_4t_{\mathrm{rb}\triangle}(\bm W)
+
\sum_{a=1}^3\beta_a t(K_2^{(a)},\bm W)
-
\frac12\mathcal I(\bm W)
\right\}.
\end{align}  
For the lower bound, evaluate the variational functional at the extremal
rainbow graphon $\bm W^{\mathrm{rb}}$. By
Corollary~\ref{cor:graphon_extremality_rainbow}, $t_{\mathrm{rb}\triangle}(\bm W^{\mathrm{rb}}) = \frac1{15}$. Consequently,
\begin{align*}
\psi_{\mathrm{rb}\triangle}(\beta_4) \ge T(\bm W^{\mathrm{rb}}) - \frac12\mathcal I(\bm W^{\mathrm{rb}}) = \frac{\beta_4}{15} + \sum_{a=1}^3 \beta_a t(K_2^{(a)},\bm W^{\mathrm{rb}}) - \frac12\mathcal I(\bm W^{\mathrm{rb}}).
\end{align*}
The edge parameters $\beta_1,\beta_2,\beta_3$ are fixed, and the terms on the
second line other than $\beta_4/15$ do not depend on $\beta_4$. Dividing by
$\beta_4$ and letting $\beta_4\to+\infty$ therefore gives
\[
\liminf_{\beta_4\to+\infty}
\frac{\psi_{\mathrm{rb}\triangle}(\beta_4)}{\beta_4}
\ge
\frac1{15}.
\]

For the upper bound, note that, for every $\bm W\in\mathcal W_3$, one has  
$0\le t(K_2^{(a)},\bm W)\le1$, for $a\in[3]$, and $-\frac12\mathcal I(\bm W) \le
\frac12\log3$. Moreover, Corollary~\ref{cor:graphon_extremality_rainbow} gives $t_{\mathrm{rb}\triangle}(\bm W)\le\frac1{15}$. Then recalling \eqref{eq:rbenergy} gives, $\psi_{\mathrm{rb}\triangle}(\beta_4) \le \frac{\beta_4}{15} + \sum_{a=1}^3|\beta_a| + \frac12\log3$. After division by $\beta_4$ and taking limit as $\beta_4\to+\infty$ gives,  
\[
\limsup_{\beta_4\to+\infty}
\frac{\psi_{\mathrm{rb}\triangle}(\beta_4)}{\beta_4}
\le
\frac1{15}.
\]
Combining the lower and upper bounds proves \eqref{eq:rainbowZ}.
\hfill $\Box$

\subsection{ Proof of Theorem \ref{thm:concentration_rainbow_negative} }
\label{sec:concentration_rainbow_negativepf}

Recall from \eqref{eq:tWrainbow_sym} that $t_{\mathrm{rb}\triangle}^{\mathrm{sym}}(\bm W) = 6t_{\mathrm{rb}\triangle}(\bm W)$. Hence, the constraints $t_{\mathrm{rb}\triangle}(\bm W)=0$ and  $t_{\mathrm{rb}\triangle}^{\mathrm{sym}}(\bm W)=0$ are equivalent. Further, recalling \eqref{eq:tWrainbow_sym} both are equivalent to the following: Outside a single null subset of $[0,1]^3$,
\begin{equation}\label{eq:zero-rb-ae}
W_a(x,y)W_b(y,z)W_c(z,x)=0,
\quad
\text{ for every ordered triple of distinct colors } a,b,c\in[3].
\end{equation}

We begin with the following consequence of this condition: 

\begin{lemma} 
For $x, y\in[0,1]$, define  $\mathrm{supp} \bm W(x,y) := \{a\in\{1,2,3\}: W_a(x,y)>0\}$.
Then 
\begin{equation} 
|\mathrm{supp} \bm W(x,y)|\le 2, 
\label{eq:support-at-most-two}  
\end{equation}
for almost every $(x,y)\in[0,1]^2$.
\end{lemma}

\begin{proof} 
Suppose by contradiction that the set $E:=\{(x,y): |\operatorname{supp}W(x,y)|=3\}$ has positive measure.  By Fubini, there exists $x\in[0,1]$ such that $E_x:=\{y:(x,y)\in E\}$ has positive measure.  Fix such an $x$. Then for almost every $y,z\in E_x$, all three
quantities $W_1(x,y),W_2(x,y),W_3(x,y)$ are positive, and, by symmetry, all
three quantities $W_1(z,x),W_2(z,x),W_3(z,x)$ are positive. Now, fix $b\in\{1,2,3\}$, and let $\{a,c\}=\{1,2,3\}\setminus\{b\}$. Then $W_a(x,y)>0$ and $W_c(z,x)>0$, and hence, \eqref{eq:zero-rb-ae} gives $W_b(y,z)=0$ for almost every $y, z\in E_x$. Since this holds for each $b\in\{1,2,3\}$, 
$$
W_1(y,z)=W_2(y,z)=W_3(y,z)=0
\qquad\text{for almost every } y, z \in E_x,
$$
contradicting the fact that $W(y,z)\in\Delta_3$.  This proves
\eqref{eq:support-at-most-two}.
\end{proof}

We now turn to the entropy bound. By \eqref{eq:support-at-most-two}, for almost every $(x,y)$ the vector $\bm W(x,y)$ is supported on at most two colors.
Fix a pair $\{a,b\}\subset[3]$. Maximizing over probability vectors supported
on this pair gives
\[
\sup_{\substack{p_a,p_b\ge0\\p_a+p_b=1}}
\left\{
\beta_a p_a+\beta_b p_b
-\frac12\left(p_a\log p_a+p_b\log p_b\right)
\right\}
=
\frac12\log\left(e^{2\beta_a}+e^{2\beta_b}\right).
\]
The maximizer is unique and is given by
\[ p_a= \frac{e^{2\beta_a}} {e^{2\beta_a}+e^{2\beta_b}}, \qquad p_b= \frac{e^{2\beta_b}} {e^{2\beta_a}+e^{2\beta_b}}. \]
Set
\begin{align}\label{eq:M} 
M
:=
\frac12\log\left(
\max_{1\le a<b\le3}
\left(e^{2\beta_a}+e^{2\beta_b}\right)
\right).
\end{align}
By \eqref{eq:support-at-most-two}, for almost every $(x,y)$ the probability
vector $\bm W(x,y)$ is supported on at most two colors. If its support is
contained in the pair $\{a,b\}$, then the preceding optimization on that
two-color face gives
\[
\sum_{c=1}^3\beta_cW_c(x,y)
-\frac12 I_3(\bm W(x,y))
\le
\frac12\log\left(e^{2\beta_a}+e^{2\beta_b}\right)
\le M.
\]
Thus,
\[
\sum_{a=1}^3\beta_aW_a(x,y)
-\frac12 I_3(\bm W(x,y))
\le M
\]
for almost every $(x,y)$. Integrating both sides of the above inequality over $[0,1]^2$ gives,  
\begin{align}\label{eq:coloredge}
\sum_{a=1}^3
\beta_a t(K_2^{(a)},\bm W)
-\frac12\mathcal I(\bm W)
\le M
\end{align}
for every $\bm W$ satisfying $t_{\mathrm{rb}\triangle}(\bm W)=0$. Note that for every $(a,b)\in\mathcal P_*$, the constant bichromatic probability
graphon $\bm W^{(ab)}$ defined in \eqref{eq:colorW} attains equality in
\eqref{eq:coloredge}. Hence, $M$ is the exact maximum of the tie-breaker
functional over the zero-rainbow set.

We now identify the equality cases. Suppose that equality holds in
\eqref{eq:coloredge}. Since the integrand is pointwise bounded above by $M$,
equality must hold in the pointwise bound almost everywhere. The optimization
on each two-color face of $\Delta_3$ is strictly concave. Consequently, for
almost every $(x,y)$ there exists a pair
\[
P(x,y)\in\mathcal P_*
:=
\argmax_{1\le a<b\le3}
\left(e^{2\beta_a}+e^{2\beta_b}\right)
\]
such that $$\bm W(x,y)= (q_1^{P(x,y)}, q_2^{P(x,y)}, q_3^{P(x,y)}),$$ where, for $P(x,y)=\{a,b\}$, $q^{ab}_a =
\frac{e^{2\beta_a}}{e^{2\beta_a}+e^{2\beta_b}}$, $q^{ab}_b = \frac{e^{2\beta_b}}{e^{2\beta_a}+e^{2\beta_b}}$, and $q^{ab}_c=0$, 
for $c\notin\{a,b\}$. In particular, $\operatorname{supp}\bm W(x,y)=P(x,y)$, almost everywhere. We claim that, for almost every $(x,y,z)$,
\begin{equation}\label{eq:palette-triangle-equality}
P(x,y)=P(y,z)=P(z,x).
\end{equation}
Indeed, consider three two-element subsets of $[3]$. If they are not all
equal, then one can choose one color from each subset in such a way that the
three chosen colors are distinct. We now apply this observation to $P(x,y)$, $P(y,z)$, and $P(z,x)$. If these three palettes were not all equal, there would exist distinct
$a,b,c\in[3]$ such that $a\in P(x,y), b\in P(y,z), c\in P(z,x)$. Since every coordinate belonging to the corresponding palette has strictly
positive mass, this would imply $W_a(x,y)W_b(y,z)W_c(z,x)>0$, contradicting \eqref{eq:zero-rb-ae}. This proves \eqref{eq:palette-triangle-equality}.

It remains to show that $P$ is almost everywhere constant. Let $X,Y,Z$ be
independent uniform random variables on $[0,1]$. From
\eqref{eq:palette-triangle-equality}, $P(X,Y)=P(X,Z)$, almost surely. For each palette $\alpha\in\{\{1,2\},\{1,3\},\{2,3\}\}$, define
\[
p_\alpha(x)
:=
\lambda\bigl(\{y:P(x,y)=\alpha\}\bigr).
\]
Conditional on $X=x$, the variables $P(x,Y)$ and $P(x,Z)$ are independent
and have common distribution $(p_\alpha(x))_\alpha$. Since they are equal
almost surely, $1 = \sum_\alpha p_\alpha(x)^2$, for almost every $x$. Together with $\sum_\alpha p_\alpha(x)=1$, this implies that, for almost every $x$, there
is a unique palette $\alpha(x)$ such that $P(x,y)=\alpha(x)$, for almost every $y$. Finally, the symmetry $P(x,y)=P(y,x)$ gives $\alpha(x)=\alpha(y)$, for almost every $(x,y)$. Hence, $\alpha$ is almost everywhere constant. Therefore, there exists a single pair $\{a,b\}\in\mathcal P_*$ such that $P(x,y)=\{a,b\}$, for almost every $(x,y)$, and consequently $\bm W(x,y)=\bm W^{(ab)}(x,y)$, almost everywhere. Thus, the maximizers of the tie-breaker over the zero-rainbow set are exactly
\[
\mathcal N^*_{\mathrm{rb}\triangle}
=
\left\{
\tilde{\bm W}^{(ab)}:
(a,b)\in\mathcal P_*
\right\}.
\]
Applying Theorem~\ref{ThmConcentrationBeta} with $\alpha=-\beta_4$, $f(\bm W)=-t_{\mathrm{rb}\triangle}(\bm W)$, and $g(\bm W)=
\sum_{a=1}^3\beta_a t(K_2^{(a)},\bm W)$ then proves \eqref{eq:rainbowfreedistanceN} (recall Remark \ref{remark:temperaturenegativeTW}).

It remains to prove \eqref{eq:rainbowfreeZ}. Set
\[
h(\bm W)
:=
\sum_{a=1}^3\beta_a t(K_2^{(a)},\bm W)
-\frac12\mathcal I(\bm W)
\]
and let $M$ be as in \eqref{eq:M}. For every $(a,b)\in\mathcal P_*$, the graphon $\bm W^{(ab)}$ satisfies $t_{\mathrm{rb}\triangle}(\bm W^{(ab)})=0$ and 
$h(\bm W^{(ab)})=M$. Therefore, for every $\beta_4<0$, $\psi_{\mathrm{rb}\triangle}(\beta_4) \ge M$, and hence, 
\[\liminf_{\beta_4\to-\infty} \psi_{\mathrm{rb}\triangle}(\beta_4) \ge M. \]

For the reverse inequality, let
$\beta_4^{(L)}\to-\infty$, and let
$\bm W^{(L)}$ be a maximizer of
$\psi_{\mathrm{rb}\triangle}(\beta_4^{(L)})$. Since $h$ is bounded above,
write
\[
B:=\sup_{\bm W\in\mathcal W_3}h(\bm W)<\infty.
\]
Using the lower bound $\psi_{\mathrm{rb}\triangle}(\beta_4^{(L)})\ge M$, we obtain
$\beta_4^{(L)} t_{\mathrm{rb}\triangle}(\bm W^{(L)}) + h(\bm W^{(L)})
\ge M$. Since $\beta_4^{(L)}<0$, this implies $$0\le t_{\mathrm{rb}\triangle}(\bm W^{(L)}) \le \frac{B-M}{-\beta_4^{(L)}}.$$ Consequently, $t_{\mathrm{rb}\triangle}(\bm W^{(L)}) \rightarrow0$. By compactness of $\tilde{\mathcal W}_3$, after passing to a subsequence
we may assume that $\tilde{\bm W}^{(L)} \rightarrow \tilde{\bm W}_*$, in cut distance. The continuity of the rainbow triangle homomorphism density
then gives $t_{\mathrm{rb}\triangle}(\bm W_*)=0$. Moreover, using $\beta_4^{(L)}<0$ and $t_{\mathrm{rb}\triangle}(\bm W^{(L)})\ge0$, gives $\psi_{\mathrm{rb}\triangle}(\beta_4^{(L)}) \le h(\bm W^{(L)})$. The edge-density terms are continuous, while
$-\frac12\mathcal I$ is upper semicontinuous. Hence, \(h\) is upper
semicontinuous, and therefore
\[
\limsup_{L\to\infty}
\psi_{\mathrm{rb}\triangle}(\beta_4^{(L)})
\le
\limsup_{L\to\infty}h(\bm W^{(L)})
\le
h(\bm W_*).
\]
Since $\bm W_*$ satisfies the zero-rainbow constraint,
\eqref{eq:coloredge} gives $h(\bm W_*)\le M$. Hence, 
\[
\limsup_{\beta_4\to-\infty}
\psi_{\mathrm{rb}\triangle}(\beta_4)
\le M.
\]
Combining this with the lower bound proves the result in \eqref{eq:rainbowfreeZ}.
\hfill $\Box$

\section{Proofs from Section \ref{sec:symmetry} }

\subsection{Proof of Proposition \ref{ppn:finite_beta_wedge_symmetry_breaking}} 
\label{sec:finite_beta_wedge_symmetry_breakingpf}

To prove (1), we view the induced wedge as a 2-coloring of the triangle $K_3$, where color $1$ represents an edge and color $2$ represents a non-edge. Hence, calculating \eqref{eq:Lambda_pair_def} with $k=2$, $r=1$, and $H_1=K_3$ colored as above gives 
$$\Lambda=|\beta_2| (|E(K_3)| -1) (|E_1(K_3)|+|E_2(K_3)|) = 6|\beta_2|,$$
since $|E(K_3)|=3$, $|E_1(K_3)|=2$, and $|E_2(K_3)|=1$. Hence, by Theorem~\ref{thm:highT_colored} replica symmetry holds in the induced wedge ERGM whenever
$\Lambda<2$, which is precisely $|\beta_2|<\frac{1}{3}$.

To prove (2), first recall that $-\frac12 I(c)$, where $I(c)=I_2(c,(1-c))$ as defined before Proposition \ref{ppn:finite_beta_wedge_symmetry_breaking}, is strictly concave on $(0,1)$.  Hence, for any tangent point $c_0\in(0,1)$, the tangent line gives the global upper bound
$$
-\frac12 I(c)
\le
A(c_0)-B(c_0)c,
$$
where $A(c_0):=-\frac12\log(1-c_0)$ and $B(c_0):=\frac12\log\frac{c_0}{1-c_0}$. Consequently,
$$
\beta_2 c^2(1-c)-\frac12 I(c)
\le
\beta_2 c^2(1-c)+A(c_0)-B(c_0)c.
$$
When $3B(c_0)\le \beta_2$, the relevant critical point of the right-hand side is
$$
c^*=\frac{1+\sqrt{1-3B(c_0)/\beta_2}}{3},
$$
and we obtain the upper bound
$$
\psi_{\wedge}^{\mathrm{const}}(\beta_2)
\le
M(\beta_2,c_0)
:=
\beta_2(c^*)^2(1-c^*)+A(c_0)-B(c_0)c^*.
$$
Now, denote by $W_{p, q}$ the balanced 2-block graphon with value $p$ on the two diagonal blocks and value $q$ on the two off-diagonal blocks.  With $p=\frac{17}{250}, q=\frac{199}{200}$, and $c_0=\frac{1557}{2500}$,
one computes
$$
T_{\wedge}(W_{p, q}) = \frac{p^2(1-p)+q^2(1-p)+2pq(1-q)}{4} = \frac{231922367}{1000000000}  ,   
$$
and
$$
-\frac12\mathcal I(W_{p, q})
=
-\frac14\left(I(p)+I(q)\right)
\approx 0.069978,
$$ 
upto six decimal places. Therefore, 
$$
\psi_{\wedge}(\beta_2)-\psi_{\wedge}^{\mathrm{const}}(\beta_2)
\ge
\beta_2 T_{\wedge}(W_{p, q})-\frac12\mathcal I(W_{p, q})-M(\beta_2,c_0) . 
$$
A direct numerical evaluation of the right-hand side gives a positive value/bound at
$\beta_2=3.054$ (as $-M(3.054,c_0)\approx  -0.778164$).  Since its derivative in $\beta_2$ is bounded below by
$$
T_{\wedge}(W_{p, q})-(c^*)^2(1-c^*)
\ge
T_{\wedge}(W_{p, q})-\frac{4}{27}>0,
$$
it remains positive for all larger $\beta_2$.  This proves the symmetry breaking criterion in (2). \hfill $\Box$

\begin{remark}[A weaker symmetry breaking bound]
Another crude approach to establishing symmetry breaking is as follows. Instead of using the tangent-line bound on the entropy, one may use the elementary bounds $\sup_{c\in[0,1]} c^2(1-c)=\frac{4}{27}$ and $\sup_{c\in[0,1]}\left(-\frac12 I(c)\right)=\frac12\log 2$ to obtain 
$$\psi_{\wedge}^{\mathrm{const}}(\beta_2) \leq \frac{4\beta_2}{27}+\frac12\log 2 . $$ 
Also, instead of using a general balanced 2-block graphon, one may use the complete balanced bipartite graphon $W_{\mathrm{bip}}$ in \eqref{eq:bipartiteW} as the nonconstant competitor. In this case, 
$ t_{\mathrm{ind}}(K_{1,2},W_{\mathrm{bip}})=\frac14$ and $\mathcal I(W_{\mathrm{bip}})=0$.  Therefore, 
\begin{align*} 
\psi_{\wedge}(\beta_2)-\psi_{\wedge}^{\mathrm{const}}(\beta_2) \geq \frac{\beta_2}{4} - \left(\frac{4\beta_2}{27}+\frac12\log 2\right) = \frac{11\beta_2}{108}-\frac12\log 2. 
\end{align*} 
This lower bound is positive whenever $\beta_2>\frac{54}{11}\log 2\approx 3.403$. As expected, this threshold is weaker than the value $3.054$ obtained in Proposition~\ref{ppn:finite_beta_wedge_symmetry_breaking}. 
\end{remark}

\subsection{Proof of Corollary \ref{cor:rainbow_highT_sharp}}
\label{sec:rainbow_highT_sharp_pf}
For $\bm p,\bm q\in\Delta_3$, define
\begin{align*}
G_1(\bm p,\bm q):=\frac12(p_3q_2+p_2q_3), \quad  G_2(\bm p,\bm q) :=\frac12(p_1q_3+p_3q_1), \text{ and } \quad  G_3(\bm p,\bm q) :=\frac12(p_1q_2+p_2q_1)  .  
\end{align*}
By the formulas in Example~\ref{ExampEulerLagrangeRainbowTriang},
\begin{equation}\label{eq:rainbow_G_representation}
\widehat\Delta_{\bm K_3^{\mathrm{rb}}}^{(a)}(\bm W)(x,y)
=
\int_0^1
G_a\bigl(\bm W(x,z),\bm W(y,z)\bigr)
\,\mathrm dz.
\end{equation}
We claim that
\begin{equation}\label{eq:rainbow_pair_lipschitz_sharp}
\max_{a\ne b}
\left|
(G_a-G_b)(\bm p,\bm q)
-
(G_a-G_b)(\bm p',\bm q')
\right|
\le
2\delta,
\end{equation}
where
$\delta:=\max\{\|\bm p-\bm p'\|_\infty,
\|\bm q-\bm q'\|_\infty\}$.
By color symmetry, it suffices to treat $(a,b)=(1,2)$. Set
\[
F_{12}(\bm p,\bm q)
:=
G_1(\bm p,\bm q)-G_2(\bm p,\bm q)
=
\frac12\left[
 p_3(q_2-q_1)+q_3(p_2-p_1)
\right].
\]
Let $\bm h=\bm p-\bm p'$ and $\bm k=\bm q-\bm q'$. Note that the sum of the coordinates of $\bm h$ and $\bm k$ are both zero. At an intermediate point $(\bm r,\bm s)\in\Delta_3^2$, the
differential is
\begin{align*}
D F_{12}(\bm r,\bm s)[\bm h,\bm k] =\frac12\bigl[h_3(s_2-s_1)+s_3(h_2-h_1) +k_3(r_2-r_1)+r_3(k_2-k_1)\bigr].
\end{align*} 
Here, $DF_{12}(\bm r,\bm s)[\bm h,\bm k]$ denotes the directional
derivative of $F_{12}$ at $(\bm r,\bm s)$ in the direction
$(\bm h,\bm k)$. For any $\bm c\in\mathbb R^3$ and any $\bm h$ with $h_1+h_2+h_3=0$, one has $|\bm c\cdot\bm h| \le
\osc(\bm c)\|\bm h\|_\infty$. The coefficient vector of $\bm h$ above is
$\frac12(-s_3,s_3,s_2-s_1)$, whose oscillation is at most $1$, since 
$\bm s\in\Delta_3$. The same also holds for the coefficient vector of $\bm k$.
The mean-value theorem therefore gives \eqref{eq:rainbow_pair_lipschitz_sharp}.
Integrating \eqref{eq:rainbow_pair_lipschitz_sharp} in
\eqref{eq:rainbow_G_representation} yields
\begin{equation}\label{eq:rainbow_derivative_lipschitz_sharp}
\left\|
\widehat{\bm\Delta}_{\bm K_3^{\mathrm{rb}}}(\bm W)
-
\widehat{\bm\Delta}_{\bm K_3^{\mathrm{rb}}}(\bm V)
\right\|_{\mathrm{osc},\infty}
\le
2\|\bm W-\bm V\|_\infty.
\end{equation}
The edge terms in \eqref{eq:rainbowT} have constant functional derivatives, so
they cancel when the fixed-point map is evaluated at $\bm W$ and $\bm V$.
Using Lemma~\ref{lem:softmax_osc_lip} and
\eqref{eq:rainbow_derivative_lipschitz_sharp},
\begin{align*}
\|\SoftMax(2\widehat{\bm\Psi}(\bm W))
-
\SoftMax(2\widehat{\bm\Psi}(\bm V))\|_\infty
&\le
\frac{|\beta_4|}{2}
\left\|
\widehat{\bm\Delta}_{\bm K_3^{\mathrm{rb}}}(\bm W)
-
\widehat{\bm\Delta}_{\bm K_3^{\mathrm{rb}}}(\bm V)
\right\|_{\mathrm{osc},\infty}\\
&\le
|\beta_4|\|\bm W-\bm V\|_\infty.
\end{align*}
Thus the fixed-point map is a strict contraction whenever $|\beta_4|<1$.
Every maximizer satisfies this fixed-point equation by
Theorem~\ref{thm:EL-prob-graphon}, so the maximizer is unique. Relabeling
invariance then implies that it is constant, exactly as at the end of the proof
of Theorem~\ref{thm:highT_colored}. \hfill $\Box$

\small

\subsubsection*{Acknowledgement}  
B. B. Bhattacharya was supported by NSF CAREER grant DMS 2046393 and a Sloan Research Fellowship. 

\subsubsection*{AI usage disclosure}
The research themes, ideas, theorem statements, and proof strategies in this work originated with the authors, who assume full responsibility for all content. Generative AI tools (OpenAI ChatGPT and Codex, using GPT-5.5 and GPT-5.6 Sol, and Anthropic Claude, primarily using Claude Fable 5 and Claude Opus 5) were used to cross-check some of the mathematical arguments, refine grammar and typography, and prepare the figures presented in this paper.

\bibliographystyle{abbrvnat}  
\bibliography{graphon_bib_final}

\normalsize

\appendix

\section{Some Useful Results }

In this section we collect the statements of some well-known results for ease of referencing. 

\subsection{Varadhan's Lemma}

In the proofs, we make use of Varadhan's lemma and one of its standard consequences. We begin with the direct form of Varadhan's lemma. For simplicity, we state it only in the uniformly bounded setting.

\begin{theorem}[Varadhan's Lemma: Theorem II.7.1(a) in \cite{ellis2012entropy}]
\label{thm::EVaradhanLemma}    
Let $\cX$ be a complete, separable metric space. Let $\{Q_n\}_{n \in \N}$, be a sequence of Borel probability measures that satisfy an LDP with speed $\{a_n\}_{n \in \N}$ and good rate function $I$. Let $F: \cX \to \R$ be a bounded, continuous function. Then,
$$\sup_{x \in \cX}\left(F(x) - I(x)\right) < \infty,$$
and
\begin{equation*}
\lim_{n\to\infty} \frac{1}{a_n}\log\int_{\cX} \exp\left(a_nF(x)\right)\,Q_n(dx) = \sup_{x \in \cX} \left(F(x) - I(x)\right).
\end{equation*}
\end{theorem}
\begin{proof}
As highlighted above, this is a restatement of \cite[Theorem II.7.1(a)]{ellis2012entropy}. We note that the text uses the term \emph{entropy function} for $I$. However, the definition for entropy function given in \cite[Definition II.3.1]{ellis2012entropy} exactly matches the standard definition of a good rate function.
\end{proof}

One useful application of Varadhan's lemma, given in \cite[Theorem II.7.2]{ellis2012entropy}, allows us to derive an LDP for exponentially tilted sequences of probability measures. Additionally, this theorem establishes the exponential decay of the probability that a random element with the resulting tilted distribution deviates from the solutions to a variational problem. A $k$-colored ERGM may be seen as  exponential tilt of a uniformly $k$-colored graph with respect to a bounded Hamiltonian (sufficient statistic). Since we assume throughout that the sufficient statistic is bounded, we choose to simplify the restatement of the theorem below by restricting our considerations to exponential tilts with respect to a bounded function.

\begin{theorem}[Theorem II.7.2 in \cite{ellis2012entropy}]\label{thm:TiltedLDP}
Let $\mathcal{X}$ be a complete separable metric space and $\{Q_n; n = 1, 2, \dots\}$ a sequence of probability measures on $\mathcal{B}(\mathcal{X})$. Assume that $\{Q_n\}$ satisfies a large deviation principle with speed $\{a_n\}$ and good rate function $\cI$. Let $F$ be a continuous and uniformly bounded function from $\mathcal{X}$ to $\mathbb{R}$.
For $n = 1, 2, \dots$ and $A \in \mathcal{B}(\mathcal{X})$ define probability measures
\begin{equation*}
    Q_{n,F}(A) = \frac{\int_A \exp(a_n F(x)) Q_n(dx)}{\int_{\mathcal{X}} \exp(a_n F(x)) Q_n(dx)}.
\end{equation*}
We will sometimes call $Q_{n,F}$ the tilting of $Q_{n}$ with respect to $F$.
The following conclusions hold.
\begin{enumerate}
    \item[(a)] The sequence $\{Q_{n,F}; n = 1, 2, \dots \}$ satisfies a large deviation property with constants $\{a_n\}$ and entropy function
    $$
    \cI_F(x) = \cI(x) - F(x) - \inf_{y \in \mathcal{X}} \{\cI(y) - F(y)\}.
    $$
    \item[(b)] Let $K$ be a closed set in $\mathcal{X}$ which does not contain a minimum point of $\cI_F$. Then there exists a number $N = N(K) > 0$ such that
    $$
    Q_{n,F}(K) \leq e^{-a_n N}
    $$
    for all sufficiently large $n$.
\end{enumerate}
In this sense, the asymptotic behavior of $\{Q_{n,F}\}$ is determined by the minimum points of $\cI_F$. In particular, if $\cI_F$ has a unique minimum point $x_0$, then $Q_{n,F} \Rightarrow \delta_{x_0}$.
\end{theorem}

\subsection{Generalized H\"{o}lder's Inequality}

We state here the H\"{o}lder's inequality that we use in the proof of Theorem \ref{thm:explicit_solvable_monotone}. The two-function case, including the equality condition, is given in
\cite[Section 2.3]{lieb2001analysis}. The stated finite-factor version
follows by induction.

\begin{lemma}[H\"older's inequality]
\label{lem:holder_equality}
Let $(\Omega,\mathcal F,\mu)$ be a measure space, let $p_1,\ldots,p_m\in(1,\infty)$ satisfy $\sum_{i=1}^m \frac{1}{p_i}=1$, and let $f_i\in L^{p_i}(\mu)$, for $i \in [m]$. Then
\[
\left|
\int_\Omega \prod_{i=1}^m f_i\,\mathrm d\mu
\right|
\leq
\prod_{i=1}^m \|f_i\|_{p_i}.
\]
Further, if $\|f_i\|_{p_i}>0$ for every $i \in [m]$, then equality holds if and only if
\[
\frac{|f_1|^{p_1}}{\|f_1\|_{p_1}^{p_1}}
=
\cdots
=
\frac{|f_m|^{p_m}}{\|f_m\|_{p_m}^{p_m}}
\qquad \mu\text{-almost everywhere},
\]
and the product $\prod_{i=1}^m f_i$ has constant sign almost
everywhere on the set where it is nonzero. In particular, when the
functions $f_1,\ldots,f_m$ are non-negative, the sign condition is
automatic. If $\|f_i\|_{p_i}=0$ for at least one $i \in [m]$, equality holds
trivially.
\end{lemma}

\section{ Positive Tensor Decompositions }

In this section we give a simple criterion for when a 3D coefficient tensor can be expressed as a non-negative rank-one tensor.

\begin{lemma} Let $\bm B=(\beta_{abc})_{a,b,c\in[k]}$ be a real symmetric $k\times k\times k$ tensor. Then there exists a vector $\bm\alpha=(\alpha_1,\ldots,\alpha_k)\in\mathbb R_{\geq0}^k$,  such that $\beta_{abc}=\alpha_a\alpha_b\alpha_c$, for $a,b,c\in[k]$, if and only if
$$\beta_{aaa}\geq0  \quad \text{ and } \quad  \beta_{abc} = (\beta_{aaa}\beta_{bbb}\beta_{ccc})^{\frac{1}{3}},$$ 
for $a,b,c\in[k].$ In this case, the vector $\bm\alpha$ is uniquely determined by 
$\alpha_a=\beta_{aaa}^{\frac{1}{3}}$ for $a\in[k]$. 
\label{lem:monotoneabc}
\end{lemma}  

\begin{proof} Suppose first, for all $a,b,c\in[k]$, that 
$\beta_{abc}=\alpha_a\alpha_b\alpha_c$, for some $\bm\alpha = (\alpha_1, \alpha_2, \ldots, \alpha_k) \in\mathbb R_{\geq0}^k$. Taking $a=b=c$, we obtain $\beta_{aaa}=\alpha_a^3\geq0$. Hence, $\alpha_a=\beta_{aaa}^{\frac{1}{3}}$. Substituting this expression into $\beta_{abc}=\alpha_a\alpha_b\alpha_c$, gives 
$$\beta_{abc} = \beta_{aaa}^{\frac{1}{3}}\beta_{bbb}^{\frac{1}{3}}\beta_{ccc}^{\frac{1}{3}} = ( \beta_{aaa}\beta_{bbb}\beta_{ccc} )^{\frac{1}{3}}.$$ This proves the necessity of the stated conditions. 

Conversely, suppose that $\beta_{aaa}\geq0$, for all $a\in[k]$, and that $\beta_{abc} = \left(\beta_{aaa}\beta_{bbb}\beta_{ccc}\right)^{\frac{1}{3}}$, for all $a,b,c\in[k]$. Define $\alpha_a:=\beta_{aaa}^{\frac{1}{3}}$ for $a\in[k]$. Then $\alpha_a\geq0$, for all $a$, and $$ \alpha_a\alpha_b\alpha_c = \beta_{aaa}^{\frac{1}{3}}\beta_{bbb}^{\frac{1}{3}}\beta_{ccc}^{\frac{1}{3}} = \beta_{abc}. $$ This proves the sufficiency of the result. The uniqueness of $\bm\alpha$ follows from the identity $\beta_{aaa}=\alpha_a^3$. 
\end{proof}

\section{Maximizing Rainbow Triangles } 

In this section, we recall the results of Balogh et al.~\cite{baloghRainbowTrianglesThreecolored2017}
on the maximum density of rainbow triangles and reformulate them in language of probability graphons. To this end, given a $3$-coloring $\mathscr X\in\mathcal C_n^3$, let $F(\mathscr X)$ denote the number of rainbow triangles in $\mathscr X$, and define $F(n):=\max_{\mathscr X\in\mathcal C_n^3}F(\mathscr X)$. With this notation, Balogh et al.~\cite{baloghRainbowTrianglesThreecolored2017}
proved the following result.

\begin{theorem}[{\cite[Theorem~6]{baloghRainbowTrianglesThreecolored2017}}]
\label{thm:rainbow}
For $F(n)$ as defined above,
\[
\lim_{n\to\infty}
\frac{F(n)}{\binom n3}
=
\frac25.
\]
Moreover, the unique limit homomorphism maximizing the density of rainbow
triangles is generated by the sequence of iterated blowups of the properly
$3$-colored $K_4$ described in
Definition~\ref{defn:rainbowmaximumgraph}.
\end{theorem}

The preceding result is formulated for the density of unordered rainbow
triangles. We now translate it into the fixed decorated graphon normalization
used in this paper. Recall that
$\tilde{\bm W}^{\mathrm{rb}}$ denotes the equivalence class of the
iterated $K_4$-blowup probability graphon defined in Section~\ref{sec:rt}. Let $\mathscr X\in\mathcal C_n^3$, and let $F(\mathscr X)$ be its number of unordered rainbow triangles. Under the diagonal-block convention in \eqref{eq:Wa}, the exact identity \eqref{eq:empirical_rainbow_exact} gives $t_{\mathrm{rb}\triangle}(\bm W^{\mathscr X}) = \frac{F(\mathscr X)}{n^3} + \frac{1}{27n^2}$. Consequently,
\[
t_{\mathrm{rb}\triangle}^{\mathrm{sym}}
   (\bm W^{\mathscr X})
=
6t_{\mathrm{rb}\triangle}(\bm W^{\mathscr X})
=
6\frac{F(\mathscr X)}{n^3}
+
\frac{2}{9n^2}.
\]
Since $F(\mathscr X)\le\binom n3$, it follows, uniformly over all
$\mathscr X\in\mathcal C_n^3$, that
\begin{equation*}
\frac{F(\mathscr X)}{\binom n3}
=
t_{\mathrm{rb}\triangle}^{\mathrm{sym}}
   (\bm W^{\mathscr X})
+
O\left(\frac1n\right)
=
6t_{\mathrm{rb}\triangle}(\bm W^{\mathscr X})
+
O\left(\frac1n\right).
\end{equation*}
Thus the graphon functional corresponding to the normalization of
Theorem~\ref{thm:rainbow} is
$t_{\mathrm{rb}\triangle}^{\mathrm{sym}}$, rather than the single fixed
decorated density $t_{\mathrm{rb}\triangle}$. With this normalization we have the following result, which is the colored graphon analogue of Theorem \ref{thm:rainbow}.

\begin{corollary}\label{cor:graphon_extremality_rainbow} 
Let $t_{\mathrm{rb}\triangle}$ be the fixed decorated density defined in
\eqref{eq:tWrainbow}. Then
\[
\sup_{\tilde{\bm W}\in\tilde{\mathcal W}_3}
t_{\mathrm{rb}\triangle}(\bm W)
=
\frac1{15}.
\]
Equivalently,
\[
\sup_{\tilde{\bm W}\in\tilde{\mathcal W}_3}
t_{\mathrm{rb}\triangle}^{\mathrm{sym}}(\bm W)
=
\frac25.
\]
Moreover, $\tilde{\bm W}^{\mathrm{rb}}$ is the unique maximizer in the
fixed-color quotient $\tilde{\mathcal W}_3$. 
\end{corollary}

\begin{proof}  
We first prove the lower bound. By
\eqref{eq:fixed_blowup_solution}, the finite-level graphons
$\bm W^{(t)}$ associated with the iterated $K_4$ construction satisfy
\[
t_{\mathrm{rb}\triangle}(\bm W^{(t)})
=
\frac1{15}
-
\frac{1}{540\,16^{t-1}},
\qquad t\ge1.
\]
Hence, $t_{\mathrm{rb}\triangle}(\bm W^{(t)}) \rightarrow
\frac1{15}$, as $t\to\infty$, and therefore, $\sup_{\tilde{\bm W}\in\tilde{\mathcal W}_3} t_{\mathrm{rb}\triangle}(\bm W)
\ge \frac1{15}$.

We next prove the reverse inequality. Let
$\bm W\in\mathcal W_3$, and, for every $N$, sample a random
$3$-coloring $\mathscr X_N$ of $K_N$ from $\bm W$. By the sampling
convergence theorem for probability graphons (\cite[Theorem 6.13]{abraham2023probabilitygraphons}),
\[
\frac{F(\mathscr X_N)}{\binom N3}
\rightarrow
t_{\mathrm{rb}\triangle}^{\mathrm{sym}}(\bm W)
=
6t_{\mathrm{rb}\triangle}(\bm W)
\]
in probability as $N\to\infty$. On the other hand, by the definition of
$F(N)$,
\[
\frac{F(\mathscr X_N)}{\binom N3}
\le
\frac{F(N)}{\binom N3}
\]
almost surely. Theorem~\ref{thm:rainbow} gives $$\frac{F(N)}{\binom N3} \rightarrow\frac25.$$ This implies that $t_{\mathrm{rb}\triangle}^{\mathrm{sym}}(\bm W) \le \frac25$, and, hence, $t_{\mathrm{rb}\triangle}(\bm W) \le
\frac16\cdot\frac25 =
\frac1{15}$. Combining the upper and lower bounds proves the two asserted extremal values.

It remains to identify the equality case. Let $\bm W$ be a maximizer of the
fixed decorated density. Then $t_{\mathrm{rb}\triangle}^{\mathrm{sym}}(\bm W) =
6t_{\mathrm{rb}\triangle}(\bm W) =
\frac25$. Therefore, $\bm W$ also maximizes the color-symmetrized rainbow triangle density. Choose a realization of the sampled colorings
$\{\mathscr X_N\}_{N\ge1}$ such that $\tilde{\bm W}^{\mathscr X_N} \rightarrow
\tilde{\bm W}$ and $$\frac{F(\mathscr X_N)}{\binom N3} \rightarrow
\frac25.$$ Thus, the sampled sequence is asymptotically extremal. By the uniqueness of
the extremal limit homomorphism in
\cite[Theorem~6]{baloghRainbowTrianglesThreecolored2017}, its limit is,
up to a global permutation of the three colors,
the iterated $K_4$-blowup limit
$\tilde{\bm W}^{\mathrm{rb}}$. Consequently,
$\tilde{\bm W}$ agrees with
$\tilde{\bm W}^{\mathrm{rb}}$ up to a global color permutation.

We finally verify that these global color permutations do not produce
different elements of the fixed-color quotient in this particular case. The
three colors in the properly colored $K_4$ template correspond to its three
perfect matchings. Every permutation of these three perfect matchings is
induced by a permutation of the four vertices of $K_4$. Applying the
corresponding vertex permutation at every level of the recursive construction
produces a measure-preserving relabeling of $[0,1]$. Therefore, for every
$\tau\in S_3$,
\[
\tau\bm W^{\mathrm{rb}}
\sim
\bm W^{\mathrm{rb}}
\]
in the fixed-color quotient
$\tilde{\mathcal W}_3$. It follows that
$\tilde{\bm W}^{\mathrm{rb}}$ is the unique maximizing graphon class.  
\end{proof}

\section{Proof of Lemma \ref{lem:wedge-max-bipartite}} 
\label{sec:wedge-max-bipartitepf}

Recall from \eqref{eq:wedgeW} that 
\begin{align}\label{eq:inducedwedge}
t_{\mathrm{ind}}(K_{1, 2}, W) = \int_{[0, 1]^3} W(x, y) W(y, z) (1-W(x, z)) \mathrm d x \mathrm d y \mathrm d z  & \leq \int_{[0, 1]^3} W(x, y) (1-W(x, z)) \mathrm d x \mathrm d y \mathrm d z   \nonumber \\  
& = t(K_2, W) - t(K_{1, 2} , W)  .  
\end{align}
Denote by $d_W(x) = \int_0^1 W(x, y) \mathrm d{y}$ the degree function of $W$. 
By the Cauchy-Schwarz inequality $t(K_{1, 2}, W) = \int_0^1 d_W(x)^2 \mathrm{d}x \ge ( \int_0^1 d_W(x) \mathrm{d}x)^2= t(K_2, W)^2$. Hence, from \eqref{eq:inducedwedge}, 
\begin{align*}
t_{\mathrm{ind}}(K_{1, 2}, W) & \leq t(K_2, W) - t(K_2, W)^2 \leq \frac{1}{4} , 
\end{align*}
which proves \eqref{eq:maximumwedge}. 

To prove the equality statement, first note that if $W = W_{\mathrm{bip}}$ is the balanced complete bipartite graphon as in \eqref{eq:bipartiteW}, then $t_{\mathrm{ind}}(K_{1,2},W_{\mathrm{bip}})=\frac{1}{4}$ by direct computation. Conversely, suppose that $t_{\mathrm{ind}}(K_{1,2},W)=\frac{1}{4}$. Then all intermediate inequalities must be equalities, which means: (1) $t(K_2, W) = \frac{1}{2}$, (2)  $d_W(x) = \frac{1}{2}$ for almost every $x \in [0, 1]$, and (3) $t(K_3, W) = 0$, which means,  
\begin{align}\label{eq:triangle}
W(x,y)\,W(x,z)\,W(y,z)=0 \qquad\text{for almost every }(x,y,z)\in[0,1]^3. 
\end{align}
Fix $x$ in a full-measure set where $d_W(x)=\frac{1}{2}$ and \eqref{eq:triangle} almost everywhere. Define 
$$A:=\{y\in[0,1]: W(x,y)>0\}.$$ Since 
\begin{align}\label{eq:degreeW}
\int_A W(x,y)\,\mathrm{d}y = d_W(x) = \frac{1}{2} 
\end{align}
and $0\le W\le1$, we must have $\lambda(A)\ge \frac{1}{2}$. On the other hand, if $y,z\in A$, then $W(x,y)>0$ and $W(x,z)>0$, so the
triangle-free condition \eqref{eq:triangle} gives $W(y,z)=0$ for almost every $(y,z)\in A\times A$.
Thus, $W=0$, almost everywhere on $A\times A$. For almost every $y\in A$, since $d_W(y)=\frac{1}{2}$
and $y$ has no edges inside $A$, all of its degree must come from $A^c$.
Hence
$$
\frac12=d_W(y)\le \lambda(A^c)=1-\lambda(A),
$$
so $\lambda(A)\le\frac{1}{2}$. Therefore $\lambda(A)=\frac{1}{2}$. Using this and \eqref{eq:degreeW} it follows that $W(x,y)=1$ for almost every $y\in A$. Now, fix $y\in A$ in a
full-measure subset. Since $W=0$ almost everywhere on $A\times A$, $d_W(y)=\frac{1}{2}$,
and $\lambda(A^c)=\frac{1}{2}$, we must have
$$
W(y,z)=1
\qquad\text{for almost every }z\in A^c.
$$
By symmetry, this shows that almost every $z\in A^c$ is connected to almost every
point of $A$. Since $d_W(z)=\frac{1}{2}$ and $\lambda(A)=\frac{1}{2}$, there can be no
edges inside $A^c$. Hence,  
$$
W=1 \text{ almost everywhere on }A\times A^c \text{ and } A^c\times A  \text{ and }  W=0 \text{almost everywhere on }A\times A \text{ and } A^c\times A^c.
$$
This shows that $W$ is equivalent to the balanced complete balanced bipartite graphon, as claimed.  \hfill $\Box$

\end{document}